\documentclass[10pt]{article}

\usepackage{amsmath}
\usepackage{amsthm}
\usepackage{amsfonts}
\usepackage{amssymb}
\usepackage{graphicx}
\usepackage{fullpage}
\usepackage{ytableau}
\usepackage[british]{babel}
\usepackage[utf8]{inputenc}
\usepackage[all,cmtip]{xy}
\usepackage{hyperref}
\usepackage[normalem]{ulem}
\usepackage{mathtools}
\usepackage{colonequals}
\usepackage{caption}
\usepackage{mathdots}
\usepackage{float}
\usepackage{comment}
\usepackage{placeins}
\usepackage{tikz-cd}
\usepackage{tikz}
\usepackage{pgfplots}
\usetikzlibrary{datavisualization}
\newtheorem{theorem}{Theorem}
\numberwithin{theorem}{section}

\newtheorem{conjecture}[theorem]{Conjecture}
\newtheorem{corollary}[theorem]{Corollary}

\newtheorem{definition}[theorem]{Definition}

\newtheorem{lemma}[theorem]{Lemma}
\newtheorem{notation}[theorem]{Notation}

\newtheorem{proposition}[theorem]{Proposition}

\newcommand{\F}{\mathbb{F}}

\newcommand{\Z}{\mathbb{Z}}
\newcommand{\Q}{\mathbb{Q}}

\newcommand{\Tr}{\operatorname{Tr}}

\newcommand{\Id}{\operatorname{Id}}
\newcommand{\Jac}{\operatorname{Jac}}
\newcommand{\Spec}{\operatorname{Spec}}

\newcommand{\Aut}{\operatorname{Aut}}
\newcommand{\Frob}{\operatorname{Frob}}
\newcommand{\GL}{\operatorname{GL}}
\newcommand{\GSp}{\operatorname{GSp}}
\newcommand{\Sp}{\operatorname{Sp}}

\newcommand{\ST}{\operatorname{ST}}
\newcommand{\charpol}{\operatorname{charpol}}
\newcommand{\Gal}{\operatorname{Gal}}
\newcommand{\mult}{\operatorname{mult}}
\newcommand{\sing}{\operatorname{sing}}
\newcommand{\smooth}{\operatorname{smooth}}
\newcommand{\abs}[1]{\left|#1\right|}

\newcommand{\mintr}{\mathbb{P}^{\operatorname{intr}}_{g, q}}
\newcommand{\mnaive}{\mathbb{P}^{\operatorname{naive}}_{g, q}}

\newcommand{\Mnh}{\mathcal{M}^{\nhyp}}

\newcommand\Fbar{\overline{\F}}

\newcommand{\nhyp}{\textup{nhyp}}
\newcommand\N{\mathcal{N}}

\newcommand{\abGal}[1]{\operatorname{Gal}\left(\overline{#1}/#1 \right)}

\theoremstyle{definition}
\newtheorem{remark}[theorem]{Remark}

\pgfplotsset{compat=1.18}

\begin{document}
\title{On the \texorpdfstring{$L$}{}-polynomials of curves over finite fields}
\date{}
\author{Francesco Ballini, Davide Lombardo, Matteo Verzobio}

\maketitle
\begin{abstract}
We discuss, in a non-Archimedean setting, the distribution of the coefficients of L-polynomials of curves of genus $g$ over $\mathbb{F}_q$. Among other results, this allows us
to prove that the $\mathbb{Q}$-vector space spanned by such characteristic polynomials has dimension $g+1$. We also state a conjecture about the Archimedean distribution of the number of rational points of curves over finite fields.

\medskip

\noindent\textit{Keywords. }{$L$-polynomials, curves, rational points, equidistribution}

\noindent\textit{2010 Mathematics subject classification.} Primary 11G20; Secondary 11G10, 14G10.

\end{abstract}

\section{Introduction}

Let $\mathbb{F}_q$ be a finite field of characteristic $p$ and order $q=p^f$. For every $g \geq 1$, we let $\mathcal{M}_g(\mathbb{F}_q)$ be the set of smooth projective curves of genus $g$ over $\mathbb{F}_q$, up to isomorphism over $\mathbb{F}_q$. 
Recall that, given a (smooth projective) curve $C/\mathbb{F}_q$, one may introduce its zeta function
\[
Z(C/\mathbb{F}_q,s) = \exp\left( \sum_{m \geq 1} \frac{\#C(\mathbb{F}_{q^m})}{m} q^{-ms} \right),
\]
and that by work of Schmidt \cite{MR1545199} and Weil \cite{MR0027151} we know that $Z(C/\mathbb{F}_q,s)$ is a rational function of $t := q^{-s}$. More precisely, we can write
\[
Z(C/\mathbb{F}_q,s)=\frac{P_C(t)}{(1-t)(1-qt)},
\]
where $P_{C}(t)$ is a polynomial (often called the $L$-polynomial of $C$) that satisfies the following:
\begin{lemma}\label{lemma: properties of L-polynomials}
\begin{enumerate}
\item $P_{C}(t)$ has integral coefficients and $P_{C}(0)=1$;
\item $\deg P_{C}(t)=2g$, where $g=g(C)$ is the genus of $C$;
\item writing $P_{C}(t)=\sum_{i=0}^{2g} a_i t^i$ we have the symmetry relations $a_{g+i}=q^ia_{g-i}$ for every $i=0,\ldots,g$.
\end{enumerate}
\end{lemma}

Our main object of interest in this paper is the set of $L$-polynomials of all the curves of a given genus over a finite field $\mathbb{F}_q$:
\begin{definition}\label{def:Pg}
Given a finite field $\mathbb{F}_q$ and a positive integer $g$ we define
\[
\mathcal{P}_g(\mathbb{F}_q) := \{ P_C(t) \, \bigm\vert \, C \in \mathcal{M}_g(\mathbb{F}_q) \}.
\]
\end{definition}

We will focus in particular on the non-Archimedean distribution of these $L$-polynomials. For a fixed integer $N \geq 2$, upon reduction modulo $N$ one obtains from $\mathcal{P}_g(\mathbb{F}_q)$ a set $\mathcal{P}_{g, N}(\mathbb{F}_q)$ of polynomials in $(\Z/N\Z)[t]$. Considering this set of reduced polynomials both for a fixed value of $q$ and in the limit $q \to \infty$, we obtain results in three different but related directions:
\begin{enumerate}
    \item We adapt results of Katz-Sarnak from the Archimedean to the non-Archimedean setting, obtaining equidistribution statements for $\mathcal{P}_{g, N}(\mathbb{F}_q)$ as $q \to \infty$ (Theorem \ref{thm: equidistribution of charpolys}). While special instances of this result appear in the literature (especially for the case of elliptic curves, see \cite{MR3017927, gekeler2003frobenius}), the general case does not seem to have been explored previously -- though see \cite{MR2459984} for related results.
    \item The previous result allows us to disprove a recent conjecture by Bergström--Howe--Lorenzo García--Ritzenthaler \cite[Conjecture 5.1]{bergström2023refinements} about the \textit{Archimedean} distribution of the number of rational points of non-hyperelliptic curves over finite fields {(see Proposition 
    \ref{prop: conjecture does not hold, intuitive version}} and the discussion before it).
    Theorem \ref{thm: equidistribution of charpolys}, combined with the general Lang-Trotter philosophy, leads us to propose a new conjecture (Conjecture \ref{conj:main}), which seems both more natural (in view of the general principles that seem to regulate statistical phenomena in arithmetic) and in better accord with the numerical evidence (see Section \ref{sec:numericalev}).
    \item Finally, Theorem \ref{thm: equidistribution of charpolys} easily implies that, for a fixed genus $g$ and for $q \gg_g 1$, the set $\mathcal{P}_g(\mathbb{F}_q)$ spans a $\Q$-vector space of dimension $g+1$ (Remark \ref{rmk: equidistribution implies Kaczorowski-Perelli}). By considering more carefully the set $\mathcal{P}_{g, 2}(\mathbb{F}_q)$ for every fixed value of $q$, we are able to prove that this statement does, in fact, hold for all pairs $(g, q)$ (Theorem \ref{thm:OddCharacteristic}), thus confirming a conjecture of Kaczorowski and Perelli \cite[Remark 8]{PerelliKaczorowski}. {The proof is based on properties of $L$-polynomials modulo $2$ which have also recently been explored, with different aims, in \cite{zbMATH07912221}.} Using Theorem \ref{thm: equidistribution of charpolys} we can also obtain an asymptotic result for non-linear relations among the coefficients of elements of $\mathcal{P}_g(\mathbb{F}_q)$, see Theorem \ref{thm:PolynomialRelations}.
\end{enumerate}

Recently, much attention has been devoted to questions close to those that we consider here: in addition to the aforementioned \cite{bergström2023refinements}, we also refer the reader to \cite{MR3338118}, as well as \cite{MR4595380}, \cite{ma2023refinements}, and \cite{shmakov2023cohomological}. We discuss some relations between our work and these latter papers in Remark \ref{rmk: recent work}. We believe that different parts of the mathematical community are approaching the same questions we discuss in this paper from complementary perspectives, and we hope that the present work will also encourage a fruitful exchange of ideas between these different points of view.

For this introduction, we focus more specifically on our contributions. The non-Archimedean behaviour of the $L$-polynomials is closely related to the (geometric version of the) Chebotarev density theorem, in the following sense. Let $\mathcal{C} \xrightarrow{\pi} S \to \operatorname{Spec} \mathbb{Z}$ be a versal family of curves of genus $g$, that is, a family in which every isomorphism class of curves of genus $g$ appears at least once (we use the tri-canonically embedded family, see Section \ref{sect: distribution of L-polynomials mod N} for details). Considering the $N$-torsion sections of $\Jac \mathcal{C} \to S$ gives rise to a Galois cover $S' \to S$ whose Galois group $G_N$ is a subgroup of $\GL_{2g}(\Z/N\Z)$ -- essentially, $S'$ is the minimal cover of $S$ over which all the $N$-torsion sections of $\Jac \mathcal{C}$ are defined.
For every closed point $s \in S$ we have a curve $C_{s}$, defined over the finite field $\kappa(s)$, and a Frobenius element $\operatorname{Frob}_{s, N} \in G_N$. Note that this Frobenius is an element of the Galois group of the cover, and is determined by the property of inducing the finite-field Frobenius $t \mapsto t^{(\#\kappa(s))}$ on the residue field at a point $s' \in S'$ lying over $s$. As usual, $\Frob_{s, N}$ is only well-defined up to conjugacy, or equivalently, up to the choice of the point $s' \in S'$ lying over $s$. The reduction modulo $N$ of the $L$-polynomial of $C_s$ is determined by the characteristic polynomial of $\operatorname{Frob}_{s, N}$, so equidistribution results for $\operatorname{Frob}_{s, N}$ translate into equidistribution results for $P_{C} \bmod N$. We make this precise in Section \ref{sect: distribution of L-polynomials mod N}, using Deligne and Katz's equidistribution theorem instead of Chebotarev's.

Having precise control over the non-Archimedean distribution of $L$-polynomials is sufficient to show that the values of $F_q(t) = \#\{C : C \in \mathcal{M}_g(\F_q), \#C(\F_q)=t \}$ show significant local oscillations -- consecutive values of $t \in \mathbb{N}$ can correspond to wildly different values of $F_q(t)$. As already mentioned, we use this to disprove \cite[Conjecture 5.1]{bergström2023refinements}.

We propose a new conjecture that takes these local oscillations into account to compute $F_q(t)$ (we achieve this by introducing a suitable product of local factors). Here we give an informal statement: for a precise version, see Conjecture \ref{conj:main} and Remark \ref{rmk: local factor at q} for an interpretation of the quantity $\nu_\ell(q,t)$. See also the remarks after Conjecture \ref{conj:main} for a more extended discussion of the motivation behind this conjecture.
\begin{conjecture}
    Let $g\geq 1$ and $q$ be a prime power. Let $H'(q,t)$ be the `probability' that a curve $C/\F_q$ of genus $g$ has $q+1-t$ rational points. Given a prime $\ell$ define $\nu_\ell(q,t)$ as the `normalised probability that a matrix $M\in \GSp_{2g}(\Z_\ell)$ with multiplier $q$ has trace $t$' (see Equations \eqref{eq: nu ell q t} and \eqref{eq: nu q q t} for a precise definition). 
    Let $\nu_\infty(q,t)=\ST_g(t/\sqrt{q})$, where $\ST_g$ is the Sato-Tate measure in dimension $g$. Let $\nu'(q,\cdot)$ be the measure     $c\cdot\nu_\infty(q,\cdot)\prod_{\ell<\infty}\nu_\ell(q,\cdot)$, where $c$ is the normalisation constant that ensures that $\nu'$ has total mass $1$ (i.e., that it is a probability measure). 
    The $L^1$-distance between $H'(q,\cdot)$ and $\nu'(q,\cdot)$ tends to $0$ as $q \to \infty$.
\end{conjecture}

Finally, Theorem \ref{thm:OddCharacteristic} answers the following natural question:
does Lemma \ref{lemma: properties of L-polynomials} capture all the (linear) relations among the coefficients of the polynomials $P_C(t)$? In other words, what is the dimension of the $\mathbb{Q}$-vector subspace of $\mathbb{Q}[t]$ spanned by the polynomials in $\mathcal{P}_g(\mathbb{F}_q)$? 
As a consequence of Lemma \ref{lemma: properties of L-polynomials}, it is immediate to see that this space has dimension at most $g+1$. Equality holds if and only if all the linear relations among the coefficients are already listed in Lemma \ref{lemma: properties of L-polynomials}. We show that equality does in fact hold for all genera and all finite fields: this extends work of Birch \cite{MR0230682} for curves of genus 1 and of Howe-Nart-Ritzenthaler \cite{MR2514865} for curves of genus 2, and confirms the aforementioned conjecture of Kaczorowski and Perelli \cite[Remark 8]{PerelliKaczorowski}:

\begin{theorem}\label{thm:OddCharacteristic}
Let $p$ be a prime, let $f \geq 1$, and denote by $\mathbb{F}_q$ the finite field with $q=p^f$ elements. Let $\mathcal{P}_g(\mathbb{F}_q)$ be as in Definition \ref{def:Pg} and let $L_g(\mathbb{F}_q)$ be the $\mathbb{Q}$-vector subspace of $\mathbb{Q}[t]$ spanned by $\mathcal{P}_g(\mathbb{F}_q)$.
We have
\[
\dim_{\mathbb{Q}} L_g(\mathbb{F}_q) = g+1.
\]
\end{theorem}

The proof is based on the following observation: in order to establish the linear independence of a set of polynomials with integral coefficients, it is certainly enough to show that they are linearly independent modulo 2. In the case of the $L$-polynomial of a curve $C$, the reduction modulo 2 can be read off the action of Galois on the set of $2$-torsion points of the Jacobian of $C$. In turn, when $C$ is hyperelliptic, this action is easy to write down explicitly in terms of a defining equation of $C$: one can then find $g+1$ curves whose $L$-polynomials form a basis of $L_g(\mathbb{F}_q)$. Since the properties of the $2$-torsion points are slightly different depending on whether the characteristic is odd or even, we split our proof into two parts, one for the case $p$ odd and one for the case $p=2$. 
We remark in particular that our proof is constructive: we explicitly give $g+1$ curves whose $L$-polynomials form a basis of $L_g(\mathbb{F}_q)$, see Corollary \ref{cor:PolyCongruences} for odd $p$ and the proof in Section \ref{sec:even} for $p=2$.

We conclude this introduction by briefly describing the structure of the paper. In Section \ref{sect: distribution of L-polynomials mod N}, we prove an equidistribution result for $\mathcal{P}_{g, N}$ (see Theorem \ref{thm: equidistribution of charpolys}). In Section \ref{sec:conj} we state our conjecture on the probability that a curve has a given number of rational points (see Conjecture \ref{conj:main}). We also explain why we believe this conjecture to be true and present some numerical evidence that supports it. We further discuss the difficulties that arise in formally defining the quantities involved in the conjecture (see in particular Remark \ref{rmk:lifttozl}). This justifies the work of Section \ref{sec:wellpos}, where we prove some technical results necessary to even state Conjecture \ref{conj:main}. Finally, in Section \ref{sec:poddeven}, we prove Theorem \ref{thm:OddCharacteristic} and in Section \ref{sec:algind} we study non-linear relations among the coefficients of the polynomials in $\mathcal{P}_g(\mathbb{F}_q)$. 

\smallskip

\noindent\textbf{Acknowledgments.} We thank Umberto Zannier for bringing the problem to our attention, for many useful suggestions, and especially for pointing out the relevance of the equidistribution results of Katz-Sarnak, noting that they imply the case $q \gg_g 0$ of Theorem \ref{thm:OddCharacteristic}. In addition, the first author would like to thank Umberto Zannier for his guidance during his undergraduate studies, on a topic that ultimately inspired much of the work in this paper. We are grateful to J.~Kaczorowski and A.~Perelli for sharing their work \cite{PerelliKaczorowski} before publication. We thank Christophe Ritzenthaler and Elisa Lorenzo Garc\'ia for their interesting comments on the first version of this paper, Zhao Yu Ma for a comment about Remark \ref{rem:PPAV}, and the anonymous referees for their helpful suggestions.

\smallskip

\noindent\textbf{Funding.}
The second and third authors have been partially supported by MIUR grant PRIN 2017 ``Geometric, algebraic and analytic
methods in arithmetic" and MUR grant PRIN-2022HPSNCR (funded by the European Union project Next Generation EU), and by the University of Pisa through PRA 2018 and 2022 ``Spazi di moduli, rappresentazioni e strutture combinatorie".
The third author has received funding from the European Union’s Horizon 2020 research and 
innovation program under the Marie Skłodowska-Curie Grant Agreement No. 
101034413.

\subsection{Notation and classical results}
We fix our notation for symplectic groups:

\begin{definition} 
    Let $g \geq 1$ and let $R$ be a commutative ring with identity. Fix a non-degenerate
    alternating bilinear form on $R^{2g}$, represented by the matrix $\Omega$ (note that the form is non-degenerate if and only if $\det \Omega \in R^\times$). The group $\GSp_{2g}(R)$ is by definition
    \[
    \GSp_{2g}(R) = \{ M \in \GL_{2g}(R) : \exists \lambda \in R^\times \text{ such that } {}^t M \Omega M = \lambda \Omega \}.
    \]
    The multiplier of a matrix $M \in \GSp_{2g}(R)$ is the uniquely determined $\lambda \in R^\times$ such that ${}^t M \Omega M = \lambda \Omega$. We denote it by $\mult(M)$.
    Given $q \in R$, we further let $\GSp_{2g}^q(R)$ be the subset of $\GSp_{2g}(R)$ consisting of those matrices that have multiplier equal to $q$ (equality in the group $R^\times$).
\end{definition}

\begin{remark}
    We will mostly be interested in the cases $R=\Z/\ell^n\Z, \Z_\ell$ or $\Q_\ell$, where $\ell$ is prime.
    By definition, the group $\GSp_{2g}(R)$ depends on the choice of $\Omega$, but when $R$ is a local ring,
    different choices of $\Omega$ lead to isomorphic groups \cite{MR0153749}. It follows easily that the same is true for $R=\Z/N\Z$ for any integer $N \geq 2$. When $R \in \{\Z/\ell^n\Z, \Z_\ell, \Q_\ell, \Z/N\Z\}$, we will therefore refer to $\GSp_{2g}(R)$ without necessarily specifying the choice of anti-symmetric form.
\end{remark}
It will be useful to recall the well-known connection between the $L$-polynomial of a (smooth projective) curve $C$ of genus $g$ and the Galois representations attached to the Jacobian $J$ of $C$. Let $p$ be a prime, let $q$ be a power of $p$, and let $C$ be a curve of genus $g$ defined over $\F_q$. Denote by $J$ the Jacobian of $C$. Let $\ell$ be any prime different from $p$ and let $T_\ell J$ be the $\ell$-adic Tate module of $J$, that is,
\[
T_\ell J := \varprojlim_{n} J(\overline{\mathbb{F}_q})[\ell^n].
\]
There is a natural action of $\Gal(\overline{\mathbb{F}_q}/\mathbb{F}_q)$ on $T_\ell J$ (induced by the action of $\Gal(\overline{\mathbb{F}_q}/\mathbb{F}_q)$ on the torsion points of $J$), and it can be shown that $T_\ell J$ is a free $\mathbb{Z}_\ell$-module of rank $2g$. Fixing a $\mathbb{Z}_\ell$-basis of $T_\ell J$ we thus obtain a representation $\rho_{\ell^\infty}: \Gal(\overline{\mathbb{F}_q}/\mathbb{F}_q) \to \operatorname{GL}_{2g}(\mathbb{Z}_\ell)$ whose image is contained in $\operatorname{GSp}_{2g}(\mathbb{Z}_{\ell})$; the relevant antisymmetric bilinear form is given by the Weil pairing. Since $\Gal(\overline{\mathbb{F}_q}/\mathbb{F}_q)$ is procyclic, generated by the Frobenius automorphism $\Frob_q$, we are mostly interested in the action of $\Frob_q$ on $T_\ell J$, which is captured by its characteristic polynomial
\[
f_{C, \ell^\infty}(t) = \det(t \operatorname{Id} - \rho_{\ell^\infty}(\Frob_q)) \in \mathbb{Z}_\ell[t].
\]
The matrix representing the action of Frobenius is symplectic with multiplier $q$.
Notice that we also have an action of $\Gal(\overline{\mathbb{F}_q}/\mathbb{F}_q)$ on the $\ell$-torsion points of $J(\overline{\mathbb{F}_q})$, which form an $\mathbb{F}_\ell$-vector space of dimension $2g$; we can thus obtain a mod-$\ell$ representation $\rho_\ell : \Gal(\overline{\mathbb{F}_q}/\mathbb{F}_q) \to \GL_{2g}(\mathbb{F}_\ell)$ and a corresponding characteristic polynomial
$f_{C, \ell}(t) = \det(t \operatorname{Id} - \rho_{\ell}(\Frob_q)) \in \mathbb{F}_\ell[t]$. It is clear from the definitions that $f_{C, \ell}(t)$ is nothing but the reduction modulo $\ell$ of $f_{C, \ell^\infty}(t)$. We can now recall the connection between $P_{C}(t)$ and $f_{C,\ell^\infty}(t)$:

\begin{theorem}[Grothendieck–Lefschetz formula, \cite{MR463174}]\label{GF-formula}
The equality $P_{C}(t)=t^{2g} f_{C, \ell^\infty}(1/t)$ holds for every prime $\ell \neq p$. In particular, the polynomial $f_{C, \ell^\infty}(t)\in\Z_\ell[t]$ has integer coefficients and does not depend on $\ell$.
\end{theorem}

\section{The distribution of \texorpdfstring{$L$}{}-polynomials modulo an integer \texorpdfstring{$N$}{}}\label{sect: distribution of L-polynomials mod N}

In this section we adapt \cite[\S 10]{MR1659828} to the problem of the distribution of characteristic polynomials of Frobenius modulo a fixed integer $N \geq 2$ (as opposed to the distribution of the coefficients with respect to the Archimedean metric which is considered in \cite{MR1659828}).
Fix a genus $g \geq 2$ and a finite field $\mathbb{F}_q$ of characteristic $p>0$ (not dividing $N$). We denote by $\mathcal{M}_g$ the stack of smooth projective curves of genus $g$, so that $\mathcal{M}_g(\F_q)$ denotes the set of $\mathbb{F}_q$-isomorphism classes of smooth projective curves of genus $g$ over $\F_q$.
We see $\mathcal{M}_g(\F_q)$ as a probability space by endowing it with one of the following two natural measures:
\begin{itemize}
    \item the `naive' counting measure $\mnaive$, which assigns equal measure to every singleton $\{C\}$, and which we normalise by requiring $\mnaive(\mathcal{M}_g(\F_q))=1$.
    \item 
the `intrinsic' measure $\mintr$ such that 
\[
\mintr(\{C\}) = \alpha \frac{1}{\#\Aut(C_{\F_q})},
\]
where $\Aut(C_{\F_q})$ is the group of automorphisms of $C$ defined over $\F_q$ and 
\[
\alpha = \left(\sum_{C \in \mathcal{M}_g(\F_q)} \frac{1}{\# \Aut(C_{\F_q})}\right)^{-1}
\]
is the uniquely determined normalisation constant that ensures \[\sum_{C \in \mathcal{M}_g(\F_q)} \mintr(\{C\})=\mintr(\mathcal{M}_g(\F_q))=1.\]
Note that $\alpha$ is simply the inverse of the (groupoid) cardinality of $\mathcal{M}_g(\mathbb{F}_q)$. In other words, it is the inverse of the number of points of the moduli space of curves of genus $g$ over $\mathbb{F}_q$, when these are counted with the correct weight (given by the inverse of the size of their automorphism group).
\end{itemize}
Our objective in this section is to study the random variable
\[
\begin{array}{cccc}
\charpol : &\mathcal{M}_g(\F_q) & \to & \mathbb{Z}[t] \\
& C & \mapsto & f_{C, \ell^\infty}(t),
\end{array}
\]
where $\ell$ is any auxiliary prime different from $p$ that we use to compute the characteristic polynomial of the Frobenius acting on $\Jac(C)$. More precisely, we will consider the (infinitely many) random variables
\[
\begin{array}{cccc}
\charpol_N : &\mathcal{M}_g(\F_q) & \to & \mathbb{Z}/N\mathbb{Z}[t] \\
& C & \mapsto & f_{C, \ell^\infty}(t) \bmod N
\end{array}
\]
obtained from $\charpol$ by reducing the characteristic polynomials modulo $N$, for all $N \not \equiv 0 \pmod p$. For simplicity, since $\charpol(C)$ is always a monic polynomial of degree $2g$, we restrict the codomain to be the finite set $\mathbb{Z}/N\mathbb{Z}[t]_{\leq 2g}$, the additive group of polynomials with coefficients in $\Z/N\Z$ and degree at most $2g$.
For each positive integer $N$ not divisible by $p$ we obtain a measure $\mu_N^q$ on $\mathbb{Z}/N\mathbb{Z}[t]_{\leq 2g}$ as follows. Consider the finite set $\GSp_{2g}^q(\Z/N\Z)$ and its natural counting measure $\mu_{\GSp_{2g}^q(\Z/N\Z)}$, normalised so that the total mass is $1$. {Concretely, this is given by
\[
\mu_{\GSp_{2g}^q(\Z/N\Z)}(X) = \frac{\#X}{\#\GSp_{2g}^q(\Z/N\Z)} \quad\quad \forall X \subseteq \GSp_{2g}^q(\Z/N\Z).
\]} The map
\[
\charpol : \GSp_{2g}^q(\Z/N\Z) \to \Z/N\Z[t]_{\leq 2g}
\]
that sends each matrix in $\GSp_{2g}^q(\Z/N\Z)$ to its characteristic polynomial allows us to define the measure 
\[
\mu_N^q \colonequals (\charpol)_* \mu_{\GSp_{2g}^q(\Z/N\Z)}.
\]
We will show:
\begin{theorem}\label{thm: equidistribution of charpolys}
    Let $N, g$ be positive integers with $g \geq 2$. With the notation above, as $q \to \infty$ along prime powers with $(q, N)=1$, the measures $(\charpol_N)_* \mnaive-\mu_N^q$ and $(\charpol_N)_* \mintr-\mu_N^q$ converge weakly to $0$.
\end{theorem}

\begin{remark}
    For $g=1$, very precise results about the distribution of characteristic polynomials modulo arbitrary integers $N$ are proven in \cite{MR3017927}. In particular, the results of that paper describe a very explicit measure $\tilde{\mu}_N^q$ and show that for $g=1$ the difference $(\charpol_N)_* \mathbb{P}_{1,g}^{\operatorname{naive}}-\tilde{\mu}_N^q$ converges to zero with an error of size at most $O_N(q^{-1/2})$. Thus, the case $g=1$ is very well understood. For this reason, and since Theorem \ref{thm: representability and universal family} below does not apply in genus $1$, we exclude the case $g=1$ from our discussion.
\end{remark}

We begin by recalling a version of Deligne's equidistribution theorem, as extended by Katz and Katz-Sarnak. We partially follow the presentation in \cite[\S 2]{MR2142226}.
We fix an integer $N \geq 2$ and a geometrically
connected, smooth, finite-type $\mathbb{Z}[1/N]$-scheme $U$ whose fibres are all geometrically connected of the same dimension. Denote by $\eta$ the generic point of $U$ and by $\overline{\eta}$ a corresponding
geometric generic point. Let $\mathcal{F}$ be a local system of symplectic free $\Z/N\Z$-modules of rank $2g$ on $U$ -- equivalently, a representation
\[
\rho_{\mathcal{F}} : \pi_1(U, \overline{\eta}) \to \operatorname{GSp}_{2g}(\Z/N\Z) \cong \operatorname{GSp}(\mathcal{F}_{\overline{\eta}}) \subset \Aut(\mathcal{F}_{\overline{\eta}}).
\]
Given a finite field $k$ of characteristic not dividing $N$, there is a unique map $\Spec k \to \Spec \Z[1/N]$. As in the introduction, a classical construction associates with every $u \in U(k)$ a (conjugacy class of) Frobenius $\operatorname{Frob}_{u, k} \in \pi_1(U, \overline{\eta})$.

\begin{theorem}
\label{thm: Katz equidistribution} 
In the situation above, suppose that the following holds.
    For every finite field $k$ (of characteristic not dividing $N$) and for the unique map $\Spec k \to \Spec \Z[1/N]$, denote by $\overline{\eta}_k$ a geometric generic point of $U_{k}$ and write $\pi_1^{\operatorname{geom}}(U_k,\overline{\eta}_k) = \pi_1\left( U_{\overline{k}}, \overline{\eta}_k \right)$.
    The representation $\rho_{\mathcal{F}}$ fits in a commutative diagram
\begin{equation}\label{eq: commutative diagram Katz-Deligne}
    \xymatrix{
1 \ar[r]  & \pi_1^{\operatorname{geom}}(U_k, \overline{\eta}_k) \ar[r] \ar[d]^{\rho_\mathcal{F}^{\operatorname{geom}}}  & \pi_1(U_k,\overline{\eta}_k) \ar[r]  \ar[d]^{\rho_\mathcal{F}} & \Gal(\overline{k}/k) \ar[r] \ar[d]^{\rho_\mathcal{F}^{k}}
& 1 \\
1 \ar[r] & \Sp_{2g}(\Z/N\Z) \ar[r]  & \GSp_{2g}(\Z/N\Z) \ar[r]_{\mult} & \mathbb{G}_m(\Z/N\Z) \ar[r] & 1
}
\end{equation}
where $\rho_\mathcal{F}^{\operatorname{geom}}$ is surjective and $\rho_{\mathcal{F}}^k$ sends the canonical generator $\operatorname{Frob}_k$ of $\operatorname{Gal}\left( \overline{k}/ k \right)$ to $\#k$. Suppose furthermore that the restriction of $\rho_{\mathcal{F}}$ to $\pi_1(U_{\overline{\Q}}, \overline{\eta})$ has image in $\Sp_{2g}(\Z/N\Z)$.

There
is a constant $C$ (depending at most on $U$, $\mathcal{F}$ and $N$) such that, for any union of conjugacy classes
$W\subseteq \GSp_{2g}(\Z/N\Z)$ and any finite field $k$ of characteristic not dividing $N$, we have
\[
\abs{
\frac{
\# \left\{ u \in U(k) : \rho_\mathcal{F}(\Frob_{u,k}) \in W \right\} }
{ \#U(k)}
-
\frac{ \#(W\cap \GSp_{2g}^{\gamma(k)}(\Z/N\Z))}{\#\Sp_{2g}(\Z/N\Z)} }
\le
\frac C{\sqrt{\#k}},
\]
where $\gamma(k)=\# k$ is the image of the canonical generator of
$\Gal(\overline{k}/k)$ under $\rho_{\mathcal{F}}^{k}$.
\end{theorem}
The deduction of this result from the work of Katz-Sarnak \cite{MR1659828} is certainly well known to experts, but it is difficult to find details in print: see for example \cite[Principle 2]{MR2946086}, where a similar result is labelled Principle `because no complete proof of this statement has appeared
in the literature to date'. We thus prefer to provide a short proof.

\begin{proof}
This is a special case of \cite[Theorem 9.7.13]{MR1659828}. More precisely, we fix an auxiliary prime $\ell$ dividing $N$ and a faithful $\overline{\Q_\ell}$-representation $\Lambda : \GSp_{2g}(\Z/N\Z) \to \operatorname{GL}(V)$ for some $\overline{\Q_\ell}$-vector space $V$, and apply \cite[Theorem 9.7.13]{MR1659828} to the $\ell$-adic sheaf $\mathcal{F}'$ corresponding to the representation $\rho \colonequals \Lambda \circ \rho_{\mathcal{F}}$. In the notation of \cite[\S 9.7.1]{MR1659828} we further take $S=\Spec \Z[1/N]$ and $X=U$.

We check that these data satisfy assumptions (1)-(4) of \cite[\S 9.7.2]{MR1659828}; set $G_{\operatorname{arith}} = \Lambda(\GSp_{2g}(\Z/N\Z)) \cong \operatorname{GSp}_{2g}(\Z/N\Z)$ and $G = \Lambda(\Sp_{2g}(\Z/N\Z)) \cong \operatorname{Sp}_{2g}(\Z/N\Z)$ (we identify these finite groups with constant algebraic subgroups of $\operatorname{GL}(V)$).
\begin{enumerate}
    \item The fact that $\rho(\pi_1(U, \overline{\eta})) \subset G_{\operatorname{arith}}(\overline{\Q_\ell})$ is true by definition. The Zariski density of $\rho(\pi_1(U, \overline{\eta}))$ in $G_{\operatorname{arith}}(\overline{\Q_\ell})$ is equivalent to the fact that $\Lambda \circ \rho_{\mathcal{F}}$ surjects onto $G_{\operatorname{arith}}$, or equivalently, that $\rho_{\mathcal{F}}$ surjects onto $\GSp_{2g}(\Z/N\Z)$. The image of $\rho_{\mathcal{F}}$ contains the image of $\rho_{\mathcal{F}}^{\operatorname{geom}}$ (for any finite field $k$ of characteristic prime to $N$), which is $\Sp_{2g}(\Z/N\Z)$ by assumption. On the other hand, by the commutative diagram in the statement, the image of $\operatorname{mult} \circ \rho_{\mathcal{F}}$ contains $\rho_{\mathcal{F}}^k(\operatorname{Frob}_k)=\#k$ for any finite field $k$ of characteristic prime to $N$. By Dirichlet's theorem, the quantity $\#k$ realises all invertible classes modulo $N$, hence the image of $\mult \circ \rho_{\mathcal{F}}$ contains all of $(\Z/N\Z)^\times$. Taken together, these facts imply that the image of $\rho_{\mathcal{F}}$ is $\GSp_{2g}(\Z/N\Z)$.
    \item The inclusion $\rho(\pi(U_{\overline{\Q}}, \overline{\eta})) \subseteq \Lambda(\Sp_{2g}(\Z/N\Z))$ is true by assumption.
    \item We have to check that for every finite field $k$ and every $k$-valued point $s$ of $\Z[1/N]$, the geometric monodromy group of $\mathcal{F}|_{U_s}$ is $\Lambda(\Sp_{2g}(\Z/N\Z))$. This is precisely the assumption that $\rho_{\mathcal{F}}^{\operatorname{geom}}$ is surjective for every finite field $k$.
    \item The image of $\Lambda$ is finite. Thus, all eigenvalues of any matrix in its image are roots of unity. This implies that $\mathcal{F}$ is $\iota$-pure of weight $0$, for any embedding $\iota$ of $\overline{\Q_\ell}$ into $\mathbb{C}$. See also the proof of \cite[Theorem 4.1]{MR1440067}.
\end{enumerate}
 
Since $G_{\operatorname{arith}}$ is a finite group (which implies that $K_{\operatorname{arith}}=G_{\operatorname{arith}}$ is finite, in the notation of \cite[Theorem 9.7.13]{MR1659828}, see \cite[Remark 9.7.11]{MR1659828}), the conclusion follows from  \cite[Theorem 9.7.13]{MR1659828}. Note that here we also use the obvious fact that $\#\GSp_{2g}^q(\Z/N\Z) = \Sp_{2g}(\Z/N\Z)$ for any $q$ prime to $N$.

\end{proof}

\def\cgn{{_N\mathcal{C}_g}}
\def\mgn{{_N\mathcal{M}_g}}
\def\cgln{{_{\ell N}\mathcal{C}_g}}
\def\mgln{{_{\ell N}\mathcal{M}_g}}

Let $\mathcal{C} \xrightarrow{\pi} U \to \Spec \Z[1/N]$ be a smooth, irreducible family of
projective curves of genus $g \ge 1$, {with the property that the map $U \to \Spec \Z[1/N]$ has} geometrically irreducible fibres, all of the same dimension. The étale sheaf $\mathcal{F} = \mathcal{F}_{\mathcal{C},N} := \Jac(\mathcal{C})[N]$ is a sheaf of $\Z/N\Z$-free symplectic modules of rank $2g$ whose fibre at a geometric point $\overline{x}\in
U$ is the $N$-torsion of the Jacobian $\Jac(\mathcal{C}_x)[N]$. Theorem \ref{thm: Katz equidistribution} applies to this situation provided that $\rho_{\mathcal{F}}^{\operatorname{geom}}$ is surjective for every finite field $k$ of characteristic not dividing $N$. The existence of a commutative diagram as in \eqref{eq: commutative diagram Katz-Deligne} is automatic thanks to well-known properties of the Weil pairing. The assumption $\rho_{\mathcal{F}}(\pi_1(U_{\overline{\Q}}, \overline{\eta})) \subseteq \Sp(\mathcal{F}_{\overline{\eta}})$ is also automatically satisfied, again by the properties of the Weil pairing. 
We will say that the family of curves $\mathcal{C} \to U$ has full $N$-monodromy if the
associated representation $\rho_\mathcal{F}: \pi_1^{\operatorname{geom}}(U_k, \overline{\eta}_k)
\to \Sp_{2g}(\Z/N\Z)$ is surjective for every finite field $k$ of characteristic not dividing $N$.

For the proof of Theorem \ref{thm: equidistribution of charpolys} we will rely on the functor $\mathcal{M}_{g, 3K}$ of tri-canonically embedded curves. Referring the reader to \cite[\S 10.6]{MR1659828} and \cite{MR0262240} for more details, we recall that for a field $k$ one has 
\[
\mathcal{M}_{g, 3K}(k) = \left\{ (C/k, \alpha) :
\begin{array}{c} C/k \text{ is a smooth projective} \\ \text{curve of genus $g$} \\ \alpha \text{ is a basis of } H^0\left(C, (\Omega^1_{C/k})^{\otimes 3}\right) \end{array} \right\} / \text{$k$-isomorphism}.
\]
The functor $\mathcal{M}_{g, 3K}$ was extensively studied by Mumford \cite{MR0214602} and Deligne-Mumford \cite{MR0262240}. We will need the following results:
\begin{theorem}[{Deligne-Mumford \cite[\S 5]{MR0262240}, see also \cite[Theorem 10.6.10]{MR1659828}}]\label{thm: representability and universal family}
Let $g\geq 2$. The following hold:
    \begin{enumerate}
        \item The functor $\mathcal{M}_{g, 3K}$ is represented by a smooth $\Z$-scheme of relative dimension $3g-3+(5g-5)^2$, with geometrically connected fibres.
        \item $\mathcal{M}_{g, 3K}$ is a fine moduli space: there exists a universal curve $\mathcal{C}_{g, 3K} \to \mathcal{M}_{g, 3K}$.
    \end{enumerate}
\end{theorem}

There is an obvious forgetful functor
$
\mathcal{M}_{g, 3K} \to \mathcal{M}_{g},
$
which on field-valued points is given by
\[
\begin{array}{ccc}
\mathcal{M}_{g, 3K}(k) & \to & \mathcal{M}_{g}(k) \\
(C/k, \alpha) & \mapsto & C/k.
\end{array}
\]
This map is surjective for every field $k$, and, when $k$ is finite, the fibre over any $C/k \in \mathcal{M}_g(k)$ has cardinality $\displaystyle \frac{\#\GL_{5g-5}(k)}{\#\Aut(C/k)}$ \cite[Lemma 10.6.8]{MR1659828}. As an immediate consequence \cite[Lemma 10.7.8]{MR1659828}, the intrinsic measure $\mintr$ on $\mathcal{M}_{g}(\F_q)$ can be described as
\begin{equation}\label{eq: express measure on Mg(Fq) via Mg3K(Fq)}
\frac{1}{\#\mathcal{M}_{g, 3K}(\F_q)} \sum_{(C, \alpha) \in \mathcal{M}_{g, 3K}(\F_q)} \delta_C,
\end{equation}
where $\delta_C$ is the characteristic function of the singleton $\{C\}$.
By Theorem \ref{thm: representability and universal family} (2), we have that the sum $\sum_{(C, \alpha) \in \mathcal{M}_{g, 3K}(\F_q)} \delta_C$ can be replaced by
\begin{equation}\label{eq: sum over Mg3K}
\sum_{u \in \mathcal{M}_{g, 3K}(\F_q)} \delta_{(\mathcal{C}_{g, 3K})_u},
\end{equation}
where $(\mathcal{C}_{g, 3K})_u$ is the fibre over $u \in \mathcal{M}_{g, 3K}(\F_q)$ of the universal curve $\mathcal{C}_{g, 3K} \to \mathcal{M}_{g, 3K}$. We will apply Theorem \ref{thm: Katz equidistribution} to $U=\left(\mathcal{M}_{g, 3K}\right)_{\Z[1/N]}$ and $\mathcal{F} = \mathcal{F}_{\mathcal{C}_{g, 3K}, N}$. For $g \geq 2$ and $p \nmid N$, this family has full $N$-monodromy by \cite[5.12]{MR0262240} (see also the discussion in \cite[\S 5]{MR3981312}).
We are almost ready to prove Theorem \ref{thm: equidistribution of charpolys}, but before doing so, we need a few estimates on the size of $\mathcal{M}_g(\F_q)$:

\begin{lemma}\label{lemma: mass of moduli space}
  For every $g \geq 3$, the following hold:
  \begin{enumerate}
  \item $\#\mathcal{M}_g(\F_q) = \sum_{C \in \mathcal{M}_g(\F_q)} 1 = q^{3g-3}(1+O_g(q^{-1/2}))$; 
  \item $\sum_{C \in \mathcal{M}_g(\F_q)} \frac{1}{\#\Aut(C_{\F_q})} = q^{3g-3}(1+O_g(q^{-1/2}))$;
  \item $\#\left\{ C \in \mathcal{M}_g(\F_q) : \#\Aut(C_{\overline{\F_q}}) \geq 2 \right\} = O_g(q^{3g-3-1})$.
  \end{enumerate}
  For $g=2$ one has
  \begin{enumerate}
  \item[1'.] $\#\mathcal{M}_2(\F_q) = \sum_{C \in \mathcal{M}_2(\F_q)} 1 = q^{3}(1+O(q^{-1/2}))$; 
  \item[2'.] $\sum_{C \in \mathcal{M}_2(\F_q)} \frac{1}{\#\Aut(C_{\F_q})} = \frac{1}{2}q^{3}(1+O(q^{-1/2}))$;
  \item[3'.] $\#\left\{ C \in \mathcal{M}_g(\F_q) : \#\Aut(C_{\overline{\F_q}}) > 2 \right\} = O(q^{2})$.
  \end{enumerate}
\end{lemma}
\begin{proof}
    For $g \geq 3$, all the statements follow from \cite[Lemmas 10.6.23, 10.6.25 and 10.6.26]{MR1659828}, together with the obvious asymptotic relation $\#\GL_{5g-5}(\F_q) \sim q^{(5g-5)^2}(1+O_g(q^{-1}))$. For $g=2$, one can adapt the proof of the same lemmas in \cite{MR1659828}, simply taking into account that the open subset $U_{\leq 2}$ of $\mathcal{M}_{2}$ parametrising curves whose geometric automorphism group has order 2 meets every geometric fibre of $\mathcal{M}_{2, 3K}/\mathbb{Z}$ \cite[Lemma 10.6.13, Remark 10.6.20]{MR1659828}. In particular, the generic value of $\#\Aut(C_{\F_q})$ for (smooth projective) curves of genus $2$ is $2$. Note that when the group $\Aut(C_{\F_q})$ has order 2 it is generated by the hyperelliptic involution.
\end{proof}

\begin{corollary}\label{cor: total variation distance between naive and intrinsic measure}
    For all $g \geq 2$ we have
        $\displaystyle 
        \sum_{C' \in \mathcal{M}_g(\F_q)} \left| \mnaive(\{C'\}) - \mintr(\{C'\}) \right| = O_g(q^{-1/2}).
        $
    \end{corollary}
\begin{proof}
    For $g \geq 3$, using the definition of $\mnaive$ and $\mintr$ and Lemma \ref{lemma: mass of moduli space} (1), (2) and (3) we obtain
            \[
        \begin{aligned}
            &\sum_{C' \in \mathcal{M}_g(\F_q)} \left| \mnaive(\{C'\}) - \mintr(\{C'\}) \right|\\ & = \sum_{C' \in \mathcal{M}_g(\F_q)} \left| \frac{1}{\#\mathcal{M}_g(\F_q)} - \frac{1/\#\Aut(C'_{\F_q})}{\sum_{C \in \mathcal{M}_g(\F_q)} 1/\#\Aut(C_{\F_q})} \right| \\
            & = \sum_{C' \in \mathcal{M}_g(\F_q)} \left| q^{3-3g}(1+O_g(q^{-1/2})) - \frac{q^{3-3g}(1+O_g(q^{-1/2}))}{\#\Aut(C'_{\F_q})}\right| \\
            & = \sum_{ \substack{C' \in \mathcal{M}_g(\F_q) \\ \#\Aut(C'_{\F_q})=1}} O_g(q^{3-3g-1/2})  + \sum_{\substack{C' \in \mathcal{M}_g(\F_q) \\ \#\Aut(C'_{\F_q}) \geq 2}} O_g(q^{3-3g}) \\
            & = O_g\left( \frac{\#\left\{ C \in \mathcal{M}_g(\F_q) : \#\Aut(C_{\overline{\F_q}}) \geq 2 \right\}}{q^{3g-3}} \right) +O_g\left( \frac{\#\mathcal{M}_g(\F_q)}{q^{3g-3}} q^{-1/2} \right) \\
            & = O_g(q^{-1}) + O_g(q^{-1/2}) = O_g(q^{-1/2}).
        \end{aligned}
        \]
        The same proof applies, with minimal changes, also to $g=2$, simply using (1'), (2'), and (3') of Lemma \ref{lemma: mass of moduli space} instead of (1), (2) and (3).
\end{proof}
\begin{proof}[Proof of Theorem \ref{thm: equidistribution of charpolys}]
By definition, the weak convergence in the statement means that -- for every continuous bounded function $f$ on $\Z/N\Z[x]_{\leq 2g}$ -- the integral of $f$ with respect to $(\charpol_N)_* \mnaive-\mu_N^q$ converges to $0$ as $q \to \infty$, and similarly for the sequence of measures $(\charpol_N)_* \mintr-\mu_N^q$.
We begin by treating the case of the measures $(\charpol_N)_* \mintr-\mu_N^q$. Since any function $f: \Z/N\Z[x]_{\leq 2g} \to \mathbb{R}$ is a linear combination of characteristic functions of singletons, it suffices to show the result when $f$ is of the form
\[
    f(h(t)) = \begin{cases}
        1, \text{ if }h(t) = h_0(t) \\
        0, \text{ otherwise}
    \end{cases}
\]
for some polynomial $h_0(t) \in \mathbb{Z}/N\Z[t]_{\leq 2g}$. Fix $h_0(t)$. The condition $\charpol(M) =h_0(t) \in \Z/N\Z[t]$ defines a (possibly empty) union of conjugacy classes $W_{h_0} \subseteq \GSp_{2g}(\Z/N\Z)$. For a curve $C/\F_q$, we denote by $\rho_{C, N}$ the natural representation of $\abGal{\F_q}$ on the $N$-torsion of $\Jac(C)$. We regard $\mathcal{M}_{g, 3K}$ as a $\mathbb{Z}[1/N]$-scheme. It will play the role of the scheme $U$ of our general discussion of the Deligne-Katz-Sarnak equidistribution theorem. We take as curve $\mathcal{C} \to U$ the universal curve $\mathcal{C}_{g, 3K}$ over $\mathcal{M}_{g, 3K}$.

Recall that we introduced the sheaf $\mathcal{F} = \mathcal{F}_{\mathcal{C}_{g, 3K}, N}$ and that the universal family over $\mathcal{M}_{g, 3K}$ has full $N$-monodromy \cite[5.12]{MR0262240}. Given a curve $\mathcal{C}_u$ in the family $\mathcal{C}$, lying over an $\F_q$-rational point $u$ of $U = \mathcal{M}_{g, 3K}$, the definitions imply that $\rho_{\mathcal{C}_u, N}(\Frob_q)$ and $\rho_{\mathcal{F}}(\operatorname{Frob}_{u, \mathbb{F}_q})$ represent the same conjugacy class.

For any fixed $q$, using Equations \eqref{eq: express measure on Mg(Fq) via Mg3K(Fq)} and \eqref{eq: sum over Mg3K} we have
    \begin{equation}
    \begin{aligned}
    &\int_{\mathcal{M}_g(\F_q)}  f(\charpol(C) \bmod N) \, d \mintr(C)
    \\ & = \frac{1}{\#\mathcal{M}_{g, 3K}(\F_q)} \sum_{(C, \alpha) \in \mathcal{M}_{g, 3K}(\F_q)} f(\charpol_N(C)) \\
    & = \frac{1}{\#\mathcal{M}_{g, 3K}(\F_q)} \sum_{u \in \mathcal{M}_{g, 3K}(\F_q)} \mathbf{1}_{{\rho_{\mathcal{C}_u, N}}({\operatorname{Frob}_q}) \in W_{h_0}} \\
    &{ = \frac{1}{\#\mathcal{M}_{g, 3K}(\F_q)} \sum_{u \in \mathcal{M}_{g, 3K}(\F_q)} \mathbf{1}_{\rho_{\mathcal{F}}(\Frob_{u, \F_q}) \in W_{h_0}}} \\
    & = \frac{\# \{u \in \mathcal{M}_{g, 3K}(\F_q) : { \rho_{\mathcal{F}}}(\Frob_{u,{ \mathbb{F}_q}}) \in W_{h_0} \}}{\#\mathcal{M}_{g, 3K}(\F_q)}.
    \end{aligned}
    \end{equation}
We now apply Theorem \ref{thm: Katz equidistribution} to rewrite the above as
    \begin{equation}\label{eq: sum of values of f(charpoly)}
    \int_{\mathcal{M}_g(\F_q)}  f(\charpol_N(C) ) \, d \mintr(C) = \frac{\#\left(W_{h_0} \cap \GSp_{2g}^q(\Z/N\Z) \right)}{\#\Sp_{2g}(\Z/N\Z)} + O_{g,  N}(q^{-1/2})
    \end{equation}
On the other hand, by definition we have
    \begin{equation}\label{eq: integral of f}
    \begin{aligned}
    \int_{\Z/N\Z[t]_{\leq 2g}} f(h(t)) \, d\mu_N^q(h) & = \int_{\GSp^q_{2g}(\Z/N\Z)} f(\charpol(M)) \, d\mu_{\GSp^q_{2g}(\Z/N\Z)}(M) \\
    & = \int_{\GSp^q_{2g}(\Z/N\Z)} \mathbf{1}_{\charpol(M) = h_0} \, d\mu_{\GSp^q_{2g}(\Z/N\Z)}(M) \\
    & = \frac{\#(W_{h_0} \cap \GSp^q_{2g}(\Z/N\Z))}{\#\GSp^q_{2g}(\Z/N\Z)} \\ 
    & = \frac{\#(W_{h_0} \cap \GSp^q_{2g}(\Z/N\Z))}{\#\Sp_{2g}(\Z/N\Z)}.
    \end{aligned}
    \end{equation}
    The claim follows upon comparing Equations \eqref{eq: sum of values of f(charpoly)} and \eqref{eq: integral of f}.
We now show that $(\charpol_N)_* \mnaive - \mu_N^q$ converges weakly to $0$. We have already established that $(\charpol_N)_* \mintr - \mu_N^q$ weakly converges to $0$. Thus, it suffices to show that $(\charpol_N)_* (\mintr - \mnaive)$ converges weakly to $0$, which in turn is implied by the following statement: for every $\varepsilon>0$ there exists $q_0$ such that, for all $q>q_0$ and all subsets $A$ of $\mathcal{M}_g(\F_q)$, one has $|\mintr(A) - \mnaive(A)| < \varepsilon$. This follows immediately from Corollary \ref{cor: total variation distance between naive and intrinsic measure}, because
\[
\begin{aligned}
|\mintr(A) - \mnaive(A)| & = \left| \sum_{C' \in A} \left( \mintr(\{C'\})-\mnaive(\{C'\}) \right) \right| \\
& \leq \sum_{C' \in A} |\mintr(\{C'\})-\mnaive(\{C'\})| = O_g(q^{-1/2}).
\end{aligned}
\]
\end{proof}

\begin{remark}\label{rmk: measures only depend on q mod N}
Note that the measure $\mu_N^q$ only depends on $q \bmod N$. In particular, if we take a sequence of prime powers $q_i$ such that $q_i \bmod N$ is constant (say equal to $r \bmod N$), Theorem \ref{thm: equidistribution of charpolys} shows that the measures $(\charpol_N)_* \mathbb{P}_{q_i, g}^{\operatorname{intr}}$ converge weakly to $\mu_N^{r}$. As a special case, taking $N=2$, this applies to any choice of $q_i$ that are not powers of $2$.    
\end{remark}

\begin{remark}\label{rmk: equidistribution modulo 2}
    Continuing from Remark \ref{rmk: measures only depend on q mod N}, we take $N=2$, let $q_i$ be the sequence of all odd primes, and apply the weak convergence of measures to the function $f=\mathbf{1}_{\Tr \equiv 0 \pmod{2}}$, where \[\Tr(x^{2g}-a_{2g-1}x^{2g-1}+\cdots+a_0)=a_{2g-1}.\] In this way, if $C$ is a curve over $\F_q$, 
    \[
    f(\charpol_2(C))=
    \begin{cases}
        1, \text{ if } \Tr(C) := q+1-\#C(\F_q) \equiv 0 \pmod{2} \\
        0, \text{ otherwise.}
    \end{cases}
    \]
    Note that this means $f(\charpol_2(C))=1$ if and only if $\#C(\F_q)$ is even.
    Applying Theorem \ref{thm: equidistribution of charpolys} to the case of the naive measure $\mnaive$ we obtain the convergence of
    \[
    \frac{1}{\#\mathcal{M}_g(\F_q)} \sum_{C \in \mathcal{M}_g(\F_q)} f(\charpol_2(C)) = \frac{\#\{C \in \mathcal{M}_g(\F_q) : \Tr(C) \equiv 0 \pmod{2}\}}{\#\mathcal{M}_g(\F_q)}
    \]
    to
    \begin{align*}
    &\mu^{1}_2\left( \{M \in \GSp_{2g}(\Z/2\Z) : \Tr(M) \equiv 0 \pmod{2} \} \right) = \frac{\#\{M \in \GSp_{2g}(\Z/2\Z) : \Tr(M) \equiv 0 \pmod{2} \}}{\#\GSp_{2g}(\Z/2\Z)}.
    \end{align*}
    Thus, we have proven
    \begin{align*}
    \lim_{q \to \infty} \frac{\#\{C \in \mathcal{M}_g(\F_q) : \Tr(C) \equiv 0 \pmod{2}\}}{\#\mathcal{M}_g(\F_q)} = \frac{\#\{M \in \GSp_{2g}(\Z/2\Z) : \Tr(M) \equiv 0 \pmod{2} \}}{\#\GSp_{2g}(\Z/2\Z)},
    \end{align*}
    where the limit is taken along the sequence of odd primes (or of their powers).
\end{remark}

\begin{remark}\label{rmk: equidistribution implies Kaczorowski-Perelli}
    Theorem \ref{thm: equidistribution of charpolys} implies Theorem \ref{thm:OddCharacteristic}, at least when the order $q$ of the finite field is sufficiently large compared to $g$. For simplicity, we only discuss the case of odd $q$. Using \cite{MR0258855}, or equivalently \cite[Theorem A.1]{MR2401624} (see also Proposition \ref{prop: Kirby-Rivin for general symplectic group} and Remark \ref{rmk: Rivin} below), one checks that the set of characteristic polynomials of matrices in $\GSp_{2g}(\F_2)$ is the $\F_2$-vector space of reciprocal polynomials (which has dimension $g+1$). Theorem \ref{thm: equidistribution of charpolys} with $N=2$ implies that, if $q \gg_g 1$, all characteristic polynomials of elements in $\GSp_{2g}(\F_2)$ are also the reduction modulo $2$ of the characteristic polynomial of Frobenius corresponding to some curve $C/\F_q$. This immediately implies that the $\Q$-vector space $L_g(\F_q)$ of Theorem \ref{thm:OddCharacteristic} has dimension at least $g+1$.
\end{remark}

\section{A conjecture on the distribution of \texorpdfstring{$\#C(\F_q)$}{}}\label{sec:conj}

In this section we describe a heuristic (motivated by the Lang-Trotter philosophy and by results of Gekeler \cite{gekeler2003frobenius} in genus 1) that gives precise predictions for the number of (smooth projective) curves over a finite field with a given number of rational points. We define the \textit{trace} of a curve $C/\F_q$ by the formula
\[
\Tr(C/\F_q) =\Tr(C) = q+1 - \#C(\F_q);
\]
by the Hasse-Weil bound, $\Tr(C)$ is an integer in the interval $[-2g\sqrt{q}, 2g\sqrt{q}]$.

We begin by recalling the definition of the Sato-Tate measure on the real interval $[-2g, 2g]$. Consider the complex Lie group $\operatorname{GSp}_{2g}(\mathbb{C})$ and let $\operatorname{USp}_{2g}$ be the maximal compact subgroup of $\operatorname{GSp}_{2g}(\mathbb{C})$ given by unitary symplectic matrices. The group $\operatorname{USp}_{2g}$, being compact, is canonically equipped with a unique Haar measure $\mu_{\operatorname{USp}_{2g}}$ normalised so that $\mu_{\operatorname{USp}_{2g}}(\operatorname{USp}_{2g})=1$.

The trace map $\operatorname{tr} : \operatorname{USp}_{2g} \to \mathbb{C}$ has image contained in the real interval $[-2g, 2g]$. We denote by $d\operatorname{ST}_g \colonequals \operatorname{tr}_* \mu_{\operatorname{USp}_{2g}}$ the push-forward of the Haar measure of $\operatorname{USp}_{2g}$ along the trace map, and we call it the \textit{Sato-Tate measure in dimension $g$}. It can be shown (for example using \cite[Lemma 8.5]{MR2920749}) that $d\operatorname{ST}_g$ is absolutely continuous with respect to the Lebesgue measure, so we also denote by $\operatorname{ST}_g : [-2g, 2g] \to \mathbb{R}$ the density function of $d\operatorname{ST}_g$. 

\begin{remark}
Explicit expressions for the function $\ST_2(x)$ can be found in \cite{MR3502944}, see in particular Theorem 5.2 of \textit{op.~cit.} We discuss the computation of $\operatorname{ST}_g(x)$ for general $g$ in Remark \ref{rmk: computation of ST_g}.
\end{remark}

Let $g\geq2$ and let $q=p^n$ be an odd prime power. We now introduce certain local factors, both at infinity and for each finite prime. We motivate the choice of these factors in Remarks \ref{rmk: motivations conjecture} and \ref{rmk: local factor at q}. First we need some notation: for an integer $t$ and a prime $\ell \neq p$, we define
\[
X_t^q(\Z_\ell) = \{ M \in \operatorname{GSp}_{2g}(\Z_\ell) : \operatorname{mult} M=q, \operatorname{tr} M =t \}.
\]
Similarly, for any prime $\ell$ (including $\ell=p$), we define
\[
\operatorname{GSp}_{2g, \Q_\ell}^q(\Q_\ell) = \{ M \in \operatorname{GSp}_{2g}(\Q_\ell) : \operatorname{mult} M=q \}
\]
and
\[
X_t^q(\Q_\ell) = \{ M \in \operatorname{GSp}_{2g}(\Q_\ell) : \operatorname{mult} M=q, \operatorname{tr} M =t \}.
\]
These notations are compatible with our later general definition of $\operatorname{GSp}_{2g, R}^q$ and $X_t^q$, see Notation \ref{notation: GSp2gm} and Definition \ref{def: X_t^m}. 
We are now ready to introduce our local factors. Given an integer $t$, we set
\[\nu_\infty(q, t) = \ST_g(t/\sqrt{q}).\]
For each prime $\ell \neq p$ we define
 	\begin{equation}\label{eq: nu ell q t}
  \nu_{\ell}(q, t) = \lim_{k \to \infty} \frac{ \# \operatorname{Im}\left( X_t^q(\Z_\ell) \to \operatorname{GSp}_{2g}(\Z/\ell^k\Z) \right)  }{\#\GSp_{2g}(\Z/\ell^k\Z) / (\ell^k\varphi(\ell^k))},
	\end{equation}
while for $\ell=p$ we set
\begin{equation}\label{eq: nu q q t}
    \nu_{p}(q, t) = \lim_{k \to \infty} \frac{\#\operatorname{Im}\left( X^q_t(\Q_p) \cap \operatorname{Mat}_{2g}(\Z_p) \to \operatorname{Mat}_{2g}(\Z/p^k\Z) \right)}{\#\operatorname{Im}\left( \operatorname{GSp}_{2g, \Q_p}^q(\Q_p) \cap \operatorname{Mat}_{2g}(\Z_p) \to \operatorname{Mat}_{2g}(\Z/p^k\Z) \right) / p^k}.
\end{equation}
In these formulas, $X_t^q(\Z_\ell) \to \operatorname{GSp}_{2g}(\Z/\ell^k\Z)$ and $\operatorname{Mat}_{2g}(\Z_p) \to \operatorname{Mat}_{2g}(\Z/p^k\Z)$ are the natural reduction maps modulo $\ell^k$ (or $p^k$), and $\operatorname{Im}$ denotes the image of a function.

\begin{remark}\label{rmk: nu ell q t is well-defined}
The limit in the definition of $\nu_{\ell}(q,t)$, including for $\ell=p$, exists thanks to \cite[Théorème 2]{MR0656627} (see also \cite[Equation (62), Page 348, Section 3]{MR0644559}). Indeed, the $\Q_\ell$-variety defined by $\{ \tilde{M} \in \GSp_{2g}(\Q_\ell) : \Tr(\tilde M)=t, \mult \tilde{M} = q \}$ has dimension $d := \dim \GSp_{2g, \Q_\ell}-2$, so by Oesterlé's theorem \cite[Théorème 2]{MR0656627} the numerators of \eqref{eq: nu ell q t} and \eqref{eq: nu q q t} are asymptotic to $c \ell^{dk}$ for some constant $c$. For a similar reason, the denominators also admit an asymptotic of the form $c' \ell^{dk}$ for some constant $c'$ (this is also easy to prove directly, at least for the case $\ell \neq p$). Therefore, the ratio converges when $k \to \infty$.
We justify the definition given in Equation \eqref{eq: nu q q t} in Remark \ref{rmk: local factor at q}.
\end{remark}

 We will work under the assumption that $q>4g^2-1$; see Remark \ref{rmk: q small wrt g} for a discussion of what happens when $q$ is small with respect to $g$.
Let 
\begin{equation}\label{eq: guess for the density}
\nu(q,t)=\nu_\infty(q, t)\prod_{\ell<\infty}\nu_{\ell}(q, t).
\end{equation}
Notice that $\nu_\infty(q, t)=0$ for $t\notin[-2g\sqrt{q}, 2g\sqrt{q}]$ and in particular $\nu(q,t)$ is non-zero for finitely many $t$ (for a fixed $q$).
The fact that the product \eqref{eq: guess for the density} converges for all $t$ is far from obvious. We will show this in Section \ref{sec:wellpos}.
Define \begin{equation}\label{eq:nu'}
    \nu'(q,t)=\frac{\nu(q,t)}{\sum_{t\in \Z}\nu(q,t)}.
\end{equation}
The denominator is non-zero, as we will show in Lemma \ref{lemma: nu q t is non-zero}.
By definition, we have
\[
\sum_{t\in \Z}\nu'(q,t)=1.
\]
\begin{definition}
	Let $g\geq2$, let $q$ be an odd prime power, and let $t$ be an integer. Denote by $H(q,t)$ the number of isomorphism classes of (smooth projective) curves of genus $g$ defined over $\F_q$ with trace $t$, that is, for which $q+1-\#C(\F_q)=t$.
 	Define 
 \begin{equation}\label{eq:H'}
    H'(q,t)=\frac{H(q,t)}{\sum_{t\in \Z}H(q,t)} = \frac{H(q,t)}{\#\mathcal{M}_g(\F_q)} =  \mnaive \left( \{ C \in \mathcal{M}_g(\F_q) : \Tr(C) = t\} \right).
 \end{equation}
	Thus, $H'(q,t)$ is the `naive probability' that a curve of genus $g$, defined over $\F_q$, has trace $t$.
\end{definition}
Notice that $H'(q, t)=0$ for $t\notin[-2g\sqrt{q}, 2g\sqrt{q}]$.
We conjecture that, for fixed $g$, as $q \to \infty$ the measures $\nu'(q,t)$ and $H'(q,t)$ converge to one another. To make this precise, we use the $L^1$-norm on the space of probability measures on $\mathbb{Z}$: since $\mathbb{Z}$ is countable, we define the $L^1$ distance $d(\mu_1, \mu_2)$ between two probability measures as
\[
d(\mu_1, \mu_2) \colonequals \sum_{t \in \Z}\abs{\mu_1(t) - \mu_2(t)}.
\]
By \cite[Proposition 4.2]{MR3726904}, the $L^1$ distance is equal up to a factor of $2$ to another natural distance on the space of probability measures, namely the total variation distance
\[
d^{\operatorname{tot.var.}}(\mu_1, \mu_2) = \sup_{A \subseteq \mathbb{Z}} \abs{\mu_1(A) - \mu_2(A)}.
\]
We can now formulate our conjecture: we phrase it in terms of $d$, but clearly we obtain an equivalent statement by replacing $d$ with $d^{\operatorname{tot.var.}}$

\begin{conjecture}\label{conj:main}
	Fix an integer $g\geq2$. As $q \to \infty$ along prime powers, we have
 \begin{equation}\label{eq:mainconj-new}
     d(H'(q, \cdot), \nu'(q, \cdot)) \to 0,
 \end{equation}
 where $H'(q,\cdot)$ and $\nu'(q, \cdot)$ are considered as probability measures on $\mathbb{Z}$.
\end{conjecture}

\begin{remark}\label{rmk: motivations conjecture}
    We now give our reasons for believing in Conjecture \ref{conj:main}. First of all, notice that by Corollary \ref{cor: total variation distance between naive and intrinsic measure} one may as well state Conjecture \ref{conj:main} using the intrinsic measure $\mintr$.
\begin{enumerate}
	\item For the case of elliptic curves and the intrinsic measure $\mintr$, the analogue of our conjecture has been proved in \cite[Theorem 5.5]{gekeler2003frobenius}, at least when $q$ is a prime number. In the proof, the author computes the value of $\nu'(q,t)$ (see \cite[Corollary 4.8]{gekeler2003frobenius}) and shows that it is \textit{equal} to $H'(q,t)$, which is computed in \cite{MR0005125}.

	\item Let $C$ be a curve of genus $g$ defined over $\F_q$. The trace $t$ of $C$ modulo $\ell^k$ is equal to the trace of the matrix $M\in \GSp_{2g}(\Z/\ell^k\Z)$ that represents the action of the Frobenius $\Frob_q$ on the $\ell^k$-torsion points of the Jacobian of $C$. Notice that there exists $\tilde{M} \in \GSp_{2g}(\Z_\ell)$ such that $\tilde{M} \equiv M \pmod{\ell^k}$ with $\operatorname{tr}(\tilde{M})=t$ and $\operatorname{mult}(\tilde{M})=q$: indeed, it suffices to take as $\tilde{M}$ the matrix representing the action of Frobenius on the full Tate module $T_\ell \Jac(C) \cong \Z_\ell^{2g}$. Hence, by Theorem \ref{thm: equidistribution of charpolys}, as $q \to \infty$ the probability that a curve $C$ has trace $t$ modulo $\ell^k$ converges to
	\[
	 \frac{ \# \left\{ M \in \GSp_{2g}(\Z/\ell^k\Z) : 
  \operatorname{tr}(M)=t, \operatorname{mult}(M)=q \right\} }{\#\GSp_{2g}(\Z/\ell^k\Z) / (\ell^k\varphi(\ell^k))}.
	\]
Taking the limit $k \to \infty$,
$\nu_\ell(q,t)$ should represent the probability that, given a random curve $C$, the Frobenius endomorphism acts on the $\ell^\infty$-torsion points of the Jacobian of the curve with trace $t$. (The numerator of $\nu_\ell(q,t)$ counts those matrices in $\operatorname{GSp}_{2g}(\Z/\ell^k\Z)$ with trace $t$ and multiplier $q$ which can be lifted to $X^q_t(\Z_\ell)$. See Remark \ref{rmk:lifttozl} for this condition, and Remark \ref{rmk: local factor at q} for the case $\ell=p$).

Our conjecture can then be seen as a minimalist one: we are essentially claiming that the distributions of the trace of Frobenius in $\Z_\ell$ for different primes $\ell$ are independent of each other (which we know is the case by Theorem \ref{thm: equidistribution of charpolys}, at least for $\ell \neq p$), and that (as $q \to \infty$) they also become independent of the distribution of the absolute value of $\Tr(\operatorname{Frob}) \in \mathbb{R}$. To put it in another way, Conjecture \ref{conj:main} is the simplest joint distribution that reproduces the correct (known) `marginal' distributions for $\Tr(C) \bmod N$ and for $\frac{|\Tr(C)|}{|\sqrt{q}|} \in [-2g,2g]$. 

\item The `minimalist' philosophy just outlined is, of course, the same that underlies the widely believed Lang-Trotter conjecture \cite[Part I, Section 3]{MR0568299}.

	\item 
Finally, numerical evidence points in the direction of the conjecture being true, see Section \ref{sec:numericalev}.

\end{enumerate}
\end{remark}

  Our conjecture should be contrasted with \cite[Conjecture 5.1]{bergström2023refinements}, which makes a different prediction for $H'(q,t)$.
The authors of \cite{bergström2023refinements} define (the analogue of our) $\nu(q,t)$ purely in terms of the Sato-Tate density $\nu_\infty$.
We believe that -- as happens for $g=1$ -- one should also take into account the measures $\nu_\ell$ for all finite $\ell$. In fact, even though we cannot prove Conjecture \ref{conj:main}, the results of Section \ref{sect: distribution of L-polynomials mod N} are enough to show that \cite[Conjecture 5.1]{bergström2023refinements} is not correct. 
The proof of this fact
is a bit technical: \cite[Conjecture 5.1]{bergström2023refinements} refers only to non-hyperelliptic curves and replaces $t/\sqrt{q}$ with the nearest integer, both of which introduce formal difficulties. However, the key idea is comparatively simple, so we isolate it in the next proposition, which shows that the measures $\nu_\infty$ and $H'$ are substantially different infinitely often. Intuitively, this contradicts \cite[Conjecture 5.1]{bergström2023refinements}. 
{A complete argument showing that \cite[Conjecture 5.1]{bergström2023refinements} does not hold is given in the preprint version of this paper \cite{ballini2024lpolynomialscurvesfinitefields}. In particular, in \cite[Appendix A]{ballini2024lpolynomialscurvesfinitefields}, we prove all the technical details necessary to show that an argument very similar to that of Proposition \ref{prop: conjecture does not hold, intuitive version} disproves \cite[Conjecture 5.1]{bergström2023refinements}. For the sake of brevity, and since that proof does not add much to the mathematical content of the paper, we decided to omit it here.}
The following proposition is stated for $g=3$, but we suspect it should hold for all $g\geq 3$.

\begin{proposition}\label{prop: conjecture does not hold, intuitive version}
    Let $g = 3$. There exists $\varepsilon>0$ such that for all odd prime powers $q$ bigger than a constant $q_0>0$ there exists $t \in [-2g\sqrt{q}, 2g\sqrt{q}] \cap \mathbb{Z}$ such that
	 \begin{equation*}
	 	\left| \sqrt{q}\mnaive( \Tr C/\F_q = t) - \ST_g(t/\sqrt{q}) \right| \geq \varepsilon.
	 \end{equation*}
\end{proposition}
\begin{proof}
We denote simply by $\mathbb{P}$ the naive probability measure $\mnaive$ on $\mathcal{M}_{g}(\F_q)$.
We assume that \[\forall \varepsilon>0 \,  \forall q_0>0 \,\exists q > q_0 \text{ odd prime power such that } \forall t \in [-2g\sqrt{q}, 2g\sqrt{q}] \cap \mathbb{Z}\] one has
	 \begin{equation}\label{eq:assumption}
	 	\left| \mathbb{P}( \Tr C/\F_q = t) - \frac{\ST_g(t/\sqrt{q})}{\sqrt{q}} \right| < \frac{\varepsilon}{\sqrt{q}}
	 \end{equation}
	 and aim for a contradiction. Fix $\varepsilon>0$ and let $p$ be an odd prime. Let $q=p^n$.
We have
	 	\[
	 	\begin{aligned}
	 		\mathbb{P}( \Tr (C/\F_q) & \equiv 0  \pmod 2 ) = \sum_{\substack{ t \in [-2g\sqrt{q}, 2g\sqrt{q}] \cap \mathbb{Z} \\ t \equiv 0 \pmod 2 }} \left( \mathbb{P}(\Tr (C/\F_q) = t) - \frac{\ST_g(t/\sqrt{q})}{\sqrt{q}} + \frac{\ST_g(t/\sqrt{q})}{\sqrt{q}} \right) \\
	 		& = \sum_{\substack{ t \in [-2g\sqrt{q}, 2g\sqrt{q}] \cap \mathbb{Z} \\ t \equiv 0 \pmod 2 }}  \frac{\ST_g(t/\sqrt{q})}{\sqrt{q}} + \sum_{\substack{ t \in [-2g\sqrt{q}, 2g\sqrt{q}] \cap \mathbb{Z} \\ t \equiv 0 \pmod 2 }}   \left( \mathbb{P}(\Tr C/\F_q = t) - \frac{\ST_g(t/\sqrt{q})}{\sqrt{q}} \right) \\
	 		& = \frac{1}{\sqrt{q}} \sum_{\substack{ t \in [-2g\sqrt{q}, 2g\sqrt{q}] \cap \mathbb{Z} \\ t \equiv 0 \pmod 2 }} \ST_g(t/\sqrt{q}) + E,
	 	\end{aligned}
	 	\]
	 	with \begin{equation}\label{eq:E}|E| \leq (4g+1)\sqrt{q} \cdot \frac{\varepsilon}{\sqrt{q}} \leq (4g+1)\varepsilon\end{equation} by \eqref{eq:assumption}. On the other hand, some basic analysis shows that (since $\ST_g$ is Riemann-integrable) \[\frac{1}{\sqrt{q}} \sum_{\substack{ t \in [-2g\sqrt{q}, 2g\sqrt{q}] \cap \mathbb{Z} \\ t \equiv 0 \pmod 2 }} \ST_g(t/\sqrt{q})\] converges, as $q=p^n$ goes to infinity, to 
   \[
   \frac{1}{2\sqrt{q}} \int_{-2g\sqrt{q}}^{2g\sqrt{q}} \ST_g(t/\sqrt{q}) dt =\frac{1}{2} \int_{-2g}^{2g} \ST_g(t) dt = \frac{1}{2}.
   \]
Therefore,
   \begin{equation}\label{eq:12}
    \abs{\mathbb{P}(\Tr (C/\F_q) \equiv 0 \pmod{2})-\frac 12}\leq \abs{E}+\varepsilon
   \end{equation}
for $q=p^n$ large enough.
	 	Let
	 \[\begin{aligned}
	 L_1(g)&\coloneqq \frac{\#\{ M \in \GSp_{2g}(\F_2) : \Tr M \equiv 0 \hspace{-0.2cm} \pmod 2, \mult M = q \equiv 1 \hspace{-0.2cm} \pmod 2 \}}{\#\GSp_{2g}(\F_2)}.
	 \end{aligned}\]
 Note that the condition $\mult M = q \equiv 1 \pmod{2}$ is actually automatic, since $1$ is the only invertible element in $\F_2$.
By Remark \ref{rmk: equidistribution modulo 2}, as $q \to \infty$ we have
   $
   \left|L_1(g) - \mathbb{P}(\Tr (C/\F_q) \equiv 0 \pmod{2}) \right| = o(1),
   $
and in particular, for $q$ large enough, we have 
   \begin{equation}\label{eq: L1(g) vs probability of even trace}
   \left|L_1(g) - \mathbb{P}(\Tr (C/\F_q) \equiv 0 \pmod{2}) \right| <  \varepsilon.  
   \end{equation}
   
We now prove that the initial claim does not hold for $g=3$. It seems likely that a similar strategy can be applied for every $g> 3$.
By direct computation, $L_1(3) = \frac{1436}{2835} \approx 0.5065\ldots$ is strictly greater than $1/2$. Fix $0< \varepsilon < \frac{|L_1(g)-1/2|}{8g}$ for $g=3$. For $q=p^n$ large enough, by Equations \eqref{eq:E}, \eqref{eq:12} and \eqref{eq: L1(g) vs probability of even trace} we get
	 	\[
   \begin{aligned}
	 	\abs{L_1(3) - \frac{1}{2}} & \leq \abs{L_1(3) - \mathbb{P}(\Tr (C/\F_q) \equiv 0 \pmod{2})} + \abs{\mathbb{P}(\Tr (C/\F_q) \equiv 0 \pmod{2}) - \frac{1}{2}} \\
   & \leq \abs{E}+2\varepsilon \leq (4g+3)\varepsilon < \abs{L_1(3)-\frac 12},
   \end{aligned}
	 	\]
	contradiction.
	 \end{proof} 
\subsection{Further remarks on Conjecture \ref{conj:main}}

In this section, we collect several other remarks on Conjecture \ref{conj:main} and the possible limits of its validity. As all the material in this section is speculative, we do not go into much detail, but we hope that this discussion will encourage others to investigate the issues raised here.

Since the statistics of the distribution of the trace of principally polarised abelian varieties (PPAV) of a fixed dimension $g$ over finite fields are the same as those of Jacobians (equivalently, of curves of genus $g$), it seems reasonable to extend Conjecture \ref{conj:main} to the family of all PPAVs of a fixed dimension. More precisely and more generally, we formulate the following conjecture, of which Conjecture \ref{conj:main} is a special case.
\begin{conjecture}\label{conj: for PPAVs}
    Let $U$ be a scheme of finite type over $\Z$ and let $\mathcal{A} \to U$ be a family of $g$-dimensional, principally polarised abelian varieties with full monodromy. For a prime power $q$, let 
    \[
    H'(q,t) = \frac{\#\{u \in U(\F_q) : q+1-t=\#\mathcal{A}_u(\F_q) \}}{\#U(\F_q)},
    \]
    seen as a measure on $\Z$. Let $\nu'$ be as in \eqref{eq:nu'}.
    As $q \to \infty$ along prime powers, we have $d(H'(q, \cdot), \nu'(q, \cdot)) \to 0$, where $H'(q, \cdot)$ and $\nu'(q, \cdot)$ are considered as probability measures on $\Z$.
\end{conjecture}

In particular, Gekeler's results \cite{gekeler2003frobenius} should perhaps be interpreted in this light. From this perspective, one should perhaps ask if Conjecture \ref{conj:main} could not be upgraded to an actual \textit{equality} for fixed $q$ (as opposed to an asymptotic statement for $q \to \infty$) when one considers the better-behaved family of all PPAVs. We will see that, while the measures $H'(q,t)$ and $\nu'(q,t)$
\textit{cannot} be equal in general, even for abelian varieties (Remark \ref{rmk: q small wrt g}), this point of view can still be helpful.

In this section we mostly focus on Conjecture \ref{conj:main}, but -- with minimal modifications -- similar comments also apply to Conjecture \ref{conj: for PPAVs}. Given the limited evidence we have in support of Conjecture \ref{conj:main}, it seems safer to restrict our discussion to the special case of the family of all curves (but we have no reason to expect a substantially different behaviour for any other family of abelian varieties with full monodromy).

\begin{remark}[Local factor at $p$]\label{rmk: local factor at q}
    We justify the choice of the local factor \eqref{eq: nu q q t}. Observe first that the more general formula
    \[
        \nu_{\ell}(q, t) = \lim_{k \to \infty} \frac{\#\operatorname{Im}\left( X^q_t(\Q_\ell) \cap \operatorname{Mat}_{2g}(\Z_\ell) \to \operatorname{Mat}_{2g}(\Z/\ell^k\Z) \right)}{\#\operatorname{Im}\left( \operatorname{GSp}_{2g, \Q_\ell}^q(\Q_\ell) \cap \operatorname{Mat}_{2g}(\Z_\ell) \to \operatorname{Mat}_{2g}(\Z/\ell^k\Z) \right) / \ell^k}
    \]
    reduces to \eqref{eq: nu ell q t} and \eqref{eq: nu q q t} respectively when $\ell \neq p$ and $\ell=p$. The denominator of this formula is essentially the average over $t \in \{0, \ldots, \ell^{k}-1\}$ of the numerator, so the ratio measures the deviation from the average of the number of symplectic matrices with a given trace. For $g=1$, Gekeler shows \cite{gekeler2003frobenius} that this formula does give the correct local factor at $p$. For $g>1$, at least when the field of definition is the prime field $\mathbb{F}_p$, one can consider the action of Frobenius on rigid (or crystalline) cohomology, which is a free $W(\mathbb{F}_p)=\mathbb{Z}_p$-module of rank $2g$: in this way, Frobenius acts symplectically on a $2g$-dimensional $\Q_p$-vector space (the cohomology group tensored with $\Q_p$) preserving a $\Z_p$-lattice, so it defines a matrix with entries in $\Z_p$ and multiplier $q$ (any such matrix does \textit{not} lie in $\GSp_{2g}(\Z_p)$, because the multiplier is not invertible in $\Z_p$ -- in fact, such a matrix does not even lie in $\GL_{2g}(\Z_p)$). Note that we cannot simply consider the Frobenius action on the Tate module $T_p$, because this has rank at most $g$, so it doesn't provide a good $p$-adic analogue of $T_\ell$ for $\ell \neq p$. It seems likely that an equidistribution result similar to Theorem \ref{thm: Katz equidistribution} should also hold in rigid cohomology (see \cite{MR4410030, hartl2020crystalline}), which would lead to the local factor \eqref{eq: nu q q t}, just like Theorem \ref{thm: Katz equidistribution} leads to \eqref{eq: nu ell q t}, see Remark \ref{rmk: motivations conjecture}.
\end{remark}

\begin{remark}[$q$ small with respect to $g$]\label{rmk: q small wrt g}
Notice that $\nu'(q, t)$ can be positive also for values of $t$ such that $q+1-t < 0$. Of course, this does not make sense, because $q+1-t$ should represent the number of $\F_q$-rational points of a curve. The point is that the support of $\nu'(q,t)$ is the full interval $[-2g\sqrt{q},2g\sqrt{q}]$, and when $q$ is small with respect to $g$ it may well happen that $q+1-2g\sqrt{q}< 0$. 

There are also subtler issues. The Sato-Tate distribution arises as the pushforward via the trace map of the Haar measure on $\operatorname{USp}_{2g}$. Suppose that $M \in \operatorname{USp}_{2g}$ corresponds to the unitarised Frobenius $\frac{\operatorname{Frob}_{C/\F_q}}{\sqrt{q}}$, where $C/\F_q$ is a smooth projective curve of genus $g$. Then, for every $m \geq 1$ one has
\[
\#C(\F_{q^m}) = q^m+1 - q^{m/2}\operatorname{tr}(M^m),
\]
and in particular, for all integers $m_1 \mid m_2$ we must have
\begin{align*}
\#C(\F_{q^{m_1}}) = q^{m_1}+1 - q^{m_1/2}\operatorname{tr}(M^{m_1}) \leq q^{m_2}+1 - q^{m_2/2}\operatorname{tr}(M^{m_2}) = \#C(\F_{q^{m_2}}).
\end{align*}
When $q$ is small with respect to $g$, there are matrices in $\operatorname{USp}_{2g}$ and integers $m_1 \mid m_2$ for which this inequality does not hold. In this regime, one should perhaps replace the usual Sato-Tate measure with the following. Let $X$ be the subset of $\operatorname{USp}_{2g}$ consisting of those matrices that satisfy all the inequalities
\[
0 \leq q^{m_1}+1 - q^{m_1/2}\operatorname{tr}(M^{m_1}) \leq q^{m_2}+1 - q^{m_2/2}\operatorname{tr}(M^{m_2}) = \#C(\F_{q^{m_2}})
\]
for all $m_1 \mid m_2$. A candidate to replace $\ST_g$ is the pushforward via the trace of the restriction of the Haar measure to the set $X$ (renormalised so as to have mass $1$).

Recall that we are fixing $g$ and sending $q$ to infinity, so this issue does not affect our Conjecture \ref{conj:main}.
\end{remark}
\begin{remark}[Asymmetry of the distribution $H'(q,t)$]
    An advantage of working with PPAVs rather than curves is that the former always admit quadratic twists, which implies that the distribution of their traces is always symmetric around $0$. This is further indication that perhaps Conjecture \ref{conj:main} is more natural for the family of PPAVs. In fact, we remark that while $\nu'(q, t)$ is symmetric (that is, $\nu'(q,-t)=\nu'(q,t)$), this is not necessarily the case for $H'(q,t)$ as soon as $g \geq 3$, as one can see for example in \cite[Figure 4]{bergström2023refinements}, or below in our own Figure \ref{fig:53}. See also \cite[\S 5]{bergström2023refinements} for a more extensive discussion of the asymmetry of $H'(q,t)$. In particular, we note again that one cannot have an exact equality $H'(q,t)=\nu'(q,t)$ for general $g$, because the right-hand side is easily seen to be symmetric. All the same, we expect the two measures to be arbitrarily close in the limit $q \to \infty$.
\end{remark}

\begin{remark}[Speed of convergence] The limit in Conjecture \ref{conj:main} cannot converge too quickly. We briefly show why.
Given a measure $\mu$ on $\Z$, let $(-1)^*\mu(\cdot)$ be the measure defined as $(-1)^*\mu(t)=\mu(-t)$ for all $t\in \Z$. By definition, $(-1)^*\nu'(q, \cdot)-\nu'(q, \cdot)=0$ since $\nu'(q,\cdot)$ is symmetric. In particular, the moments of $(-1)^*(\sqrt{q}\nu'(q, \cdot))-(\sqrt{q}\nu'(q, \cdot))$ are $0$ for all $q$. Assume that $d(H'(q, \cdot), \nu'(q, \cdot))$ converges to zero sufficiently quickly (for example, assume that the difference is $O(q^{-k-1})$ for some $k \geq 0$): the first $2k$ moments of $(-1)^*(\sqrt{q}H'(q, \cdot))-(\sqrt{q}H'(q, \cdot))$ then also converge to zero as $q$ goes to infinity. By \cite[Corollary 5.3]{bergström2023refinements}, the $n$-th moment of $(-1)^*(\sqrt{q}H'(q, \cdot))-(\sqrt{q}H'(q, \cdot))$ converges, for $n$ odd, to a real number $b_n$ and $b_n$ is non-zero for $n$ large enough (see \cite[Proposition 2.3]{bergström2023refinements}). Hence, for $n$ large enough, the $n$-th moment of $(-1)^*(\sqrt{q}H'(q, \cdot))-(\sqrt{q}H'(q, \cdot))$ does not tend to zero as $q$ goes to infinity. If $b_n \neq 0$ and $2k \geq n$, this is a contradiction. 

We thank Christophe Ritzenthaler and Elisa Lorenzo Garc\'ia for their comments that led to this remark.
\end{remark}

\begin{remark}[Jacobians among PPAVs]\label{rem:PPAV}
We again take the view that Conjecture \ref{conj:main} should be a shadow of a (possibly sharper) statement for the family of PPAVs of a given dimension. From this point of view, it is important to note that -- asymptotically -- 100\% of PPAVs of dimension 2 are Jacobians (those that are not are either products of PPAVs of lower dimension or Weil restrictions of elliptic curves). Thus, for $g =2$, the two conjectures that one can formulate (for curves of genus $2$ and principally polarised abelian surfaces) are equivalent. For $g=3$, 100\% of PPAVs are either Jacobians or quadratic twists of Jacobians (this is explained by the so-called \textit{Serre obstruction}, see e.g.~Serre's appendix to \cite{MR1795548}), so Conjectures \ref{conj:main} and \ref{conj: for PPAVs} for $g=3$ are still closely related.
As the dimension grows, Conjecture \ref{conj:main} can be interpreted as saying that Jacobians are `typical' among PPAVs -- the distribution of the trace on the subfamily of Jacobians is the same as the distribution among all PPAVs. While we believe that Conjecture \ref{conj:main} holds for all genera $g$, we should point out that it is very hard to get numerical evidence when the genus/dimension is $4$ or more. This is precisely the threshold above which the difference between PPAVs that are geometrically Jacobians and general PPAVs becomes (asymptotically) relevant, so it would be interesting to study this regime more closely. See Figure \ref{fig:PPAV} for an example in which we show the difference between taking into account only Jacobians or all PPAVs.
\end{remark}

\begin{remark}[Principally polarised abelian surfaces with trace zero]\label{rmk: PP abelian surfaces of trace 0}
In dimension two, PPAVs that are not Jacobians are either products of elliptic curves (with the product polarisation) or Weil restrictions of elliptic curves defined over a quadratic extension. In particular, over the finite field with $q$ elements, there are $\gg q^2$ abelian surfaces that are Weil restrictions of elliptic curves defined over $\mathbb{F}_{q^2}$, but not over $\mathbb{F}_q$. The Galois representation attached to $A := \operatorname{Res}_{\mathbb{F}_{q^2}/\mathbb{F}_q}(E)$ is the induction from $\Gal(\overline{\F_q} / \F_{q^2})$ to $\abGal{\F_q}$ of the representation attached to $E/\F_{q^2}$, which implies that the Frobenius trace of $A$ is zero for any such Weil restriction. Since the total number of genus-2 curves over $\mathbb{F}_q$ is of order $q^3$ (see Lemma \ref{lemma: mass of moduli space}), we expect that the proportion of PP abelian surfaces with trace $0$ should be significantly higher than the proportion of genus-2 curves with trace $0$ (both the number of genus-2 curves and the number of PP abelian surfaces is $\approx q^3$. The number of PP abelian surfaces with trace $0$ is $\gg q^2$ more than the corresponding number of curves. In particular, we expect the proportion of PP abelian surfaces of trace $0$ to be $\gg 1/q$ more than the corresponding proportion of curves). If we interpret Conjecture \ref{conj:main} as a prediction for the distribution of the number of points of PPAVs, this helps in explaining the peak at 0 in Figure \ref{fig:PPAV} (this peak is particularly noticeable since for $q=37$ the quantity $1/q$ is not at all negligible). Similar comments apply in higher dimensions, but the proportion of PPAVs having trace zero for geometric reasons becomes less significant as the dimension increases. 
\end{remark}

\begin{remark}[Lift to $\Z_\ell$]\label{rmk:lifttozl}
Equation \eqref{eq: nu ell q t} requires that we only count those matrices $M \in \GSp_{2g}(\Z/\ell^k\Z)$ with trace $t$ and multiplier $q$ that lift to a matrix $\tilde{M} \in X_t^q(\Z_\ell)$. While this condition is natural in our setting (since Frobenius is in fact represented by an $\ell$-adic matrix with the given trace and multiplier), we believe that omitting this condition should lead to the same result, that is, we conjecture that
    \[
    \tilde{\nu}_\ell(q,t) \colonequals \lim_{k \to \infty} \frac{ \# \left\{ M \in \GSp_{2g}(\Z/\ell^k\Z) : \operatorname{tr}(M)=t, \operatorname{mult}(M)=q \right\} }{\#\GSp_{2g}(\Z/\ell^k\Z) / (\ell^k\varphi(\ell^k))}
    \]
coincides with $\nu_\ell(q,t)$. It is not hard to check that this holds for $g=1$, but we have been unable to prove the result in general. The difficulties that arise lie in understanding the singularities of the variety $X_t^q$, 
that is, the $\Z_\ell$-subscheme of $\GSp_{2g, \Z_\ell}^{m}$ defined by the equation $\Tr(M)=t$. When $X_t^q$ is smooth over $\Z_\ell$, an application of Hensel's lemma shows that $\nu_\ell(q,t)$ and $\tilde{\nu}_\ell(q,t)$ both coincide with
        \[
\frac{ \# \left\{ M \in \GSp_{2g}(\F_\ell) : \operatorname{tr}(M)=t, \operatorname{mult}(M)=q \right\} }{\#\GSp_{2g}(\F_\ell) / (\ell\varphi(\ell))}.
    \]
Note that, without any information on the singularities of a variety $X/\Z_\ell$, it is very hard to control the point-counts $\#X(\Z/\ell^n\Z)$: for example, for the reduced variety defined by the equation $x^4=\ell y^4$ in the affine plane, we have $\dim X = 1$ and $\#X(\Z/\ell^n\Z) \gg \ell^{3/2n}$, with the point-count dominated by the singular points with $x \equiv y \equiv 0 \pmod{\ell^{n/4}}$. Without control on the singularities of $X$, it seems to us that no version of Hensel's lemma can be applied to understand the ratio $\#X(\Z/\ell^n\Z) / \ell^{n \dim X}$ as $n \to \infty$.
\end{remark}

\begin{remark}[Comparison to other recent work]\label{rmk: recent work}
The recent preprint \cite{shmakov2023cohomological} relates the moments
    \[
    M_n(g, q) = \mathbb{E}_{\mintr}[ \#A(\F_q)^n]
    \]
of the random variable `number of rational points of $A$' (here $A$ is drawn at random from $\mathcal{A}_g(\F_q)$ using a suitable intrinsic measure)
to the higher cohomology of certain moduli spaces, see \cite[p.~2]{shmakov2023cohomological}. This yields explicit formulas for these moments for small $g$ and $n$ \cite[Corollaries 4.3 and 5.4]{shmakov2023cohomological} and it would be interesting to compare these results with the predictions of Conjecture \ref{conj:main}. It may be possible to carry out this comparison by using the techniques of \cite{MR3582398, MR4595380}.

In particular, \cite[Theorem A]{MR4595380} comes near to proving Conjecture \ref{conj:main} in the context of principally polarised abelian varieties. However, we point out that to establish Conjecture \ref{conj:main} one would still need to overcome several obstacles: the formula of \cite[Theorem A]{MR4595380} only applies to certain isogeny classes of abelian varieties and involves Tamagawa numbers that would have to be averaged; even more substantially, it is not clear how one would isolate Jacobians among all abelian varieties. Finally, even though this is perhaps only a technical problem, the existence of the limits \eqref{eq: nu ell q t} and \eqref{eq: nu q q t} seems substantially easier to prove in the context of \cite[Theorem A]{MR4595380} than it is in the general case we consider here (essentially because in the setting of \cite[Theorem A]{MR4595380} the expression appearing under the limit sign in \eqref{eq: nu ell q t} is constant for $k \gg 0$, which is not necessarily true in our generality).
\end{remark}

\subsection{Numerical evidence}\label{sec:numericalev}

In this section we report on numerical experiments that seem to support Conjecture \ref{conj:main}. The data are computed using MAGMA \cite{MR1484478}. All the MAGMA scripts to verify our data are available online \cite{OurScripts}.

In the graphs below we plot the distribution $t \mapsto H'(q,t)$ for various values of $g$ and $q$. These distributions are obtained by directly counting all isomorphism classes of curves of the given genus over the given finite field (the data for $q=53, g=3$ is taken from \cite{MR3240800}).
In addition, on the same graphs, we also plot an approximation of the Sato-Tate density and of $\nu'(q,t)$. We briefly explain how we obtain these approximations, starting with a general technique to compute the Sato-Tate density in arbitrary dimension.
\begin{remark}[{Computation of $\operatorname{ST}_g(x)$ for arbitrary $g$}]\label{rmk: computation of ST_g}
    For general $g$, the density $\operatorname{ST}_g(x)$ can be calculated up to arbitrary precision by using a technique due to Kedlaya-Sutherland \cite{MR2555991} and Lachaud \cite{MR3502944}. One can first use \cite[Section 4.1]{MR2555991} to compute the \textit{moments} of $\operatorname{ST}_g$, that is,
\[
m_n = \int_{-2g}^{2g} x^n \, d\operatorname{ST}_g(x).
\]
Once the moments (or at least, sufficiently many moments) are known, we can recover $\ST_g(x)$ as follows.
Let $L_n(x)$ be the Legendre polynomials, which form a complete orthogonal basis of $L^2([-1, 1])$. By rescaling, the polynomials \[\tilde{L}_n(x) := \left(\int_{-2g}^{2g} L_n(x/2g)^2 \right)^{-1/2}L_n(x/2g)\] form an orthonormal basis of $L^2([-2g, 2g])$. From the explicit expression of $\tilde{L}_n(x) = \sum_{i=0}^n a_{n, i}x^i$ as a polynomial, one can easily compute \[c_n = \int_{-2g}^{2g} \tilde{L}_n(x) \, d\operatorname{ST}_g(x) = \sum_{i=0}^n a_{n, i}m_i.\] Finally, we have the convergent expansion in $L^2([-2g, 2g])$
\begin{equation}\label{eq:ST}
    \operatorname{ST}_g(x) = \sum_{n \geq 0} c_n \tilde{L}_n(x),
\end{equation}
which allows the computation of $\operatorname{ST}_g(x)$ to arbitrary precision. In our numerical experiments, we use this technique to approximate $\operatorname{ST}_3(x)$.
\end{remark}

In our numerical experiments, we approximate the Sato-Tate density with the value of the series in Equation \eqref{eq:ST} truncated at $n\leq 100$. For $\nu'(q,t)$, we approximate the value of $\nu(q,t)$ (see Equation \ref{eq: guess for the density}) by considering the product of $\nu_\ell(q,t)$ for $\ell\leq 100$ and $\ell=\infty$. To compute an approximation of $\nu_\ell(q,t)$ for $\ell$ prime, we compute the value of the expression appearing under the limit sign in Equation \ref{eq: nu ell q t} for $k=1$ or $2$. To compute an approximation of $\nu_\infty(q,t)$, we use our approximation of the Sato-Tate density.

Let \[H_{\operatorname{intr}}'(q,t)= \mintr \left( \{ C \in \mathcal{M}_g(\F_q) : \Tr(C) = t\} \right).\] We compute the value of $H_{\operatorname{intr}}'(q,t)$ by direct enumeration of all the curves of genus $g$ defined over $\F_q$.

Finally, below each graph we also give the distance $d$ between the measures $H' \colonequals H'_{\operatorname{intr}}(q, \cdot)$ and $\nu' \colonequals \nu'(q, \cdot)$, as well as the distance between $H'$ and the Sato-Tate measure. Our conjecture predicts that $d(H',\nu')$ should go to $0$ as $q$ goes to infinity. As a consequence of \cite[Conjecture 5.1]{bergström2023refinements}, $d(H',\nu_\infty)$ should go to $0$. We proved in Proposition \ref{prop: conjecture does not hold, intuitive version} that the conjecture does not hold.



\begin{figure}[H]
\begin{center}
\begin{tikzpicture}[xscale=1.1, yscale=1.1]
  \begin{axis} 
  [scaled ticks=false, ymin=0, ymax=0.02, ytick={0,0.005,0.010,0.015,0.02}, yticklabels={0,0.005,0.01,0.015,0.02}]
\addplot[only marks, red, mark size=0.3mm] coordinates {
( -127 , 0 )
( -126 , 0 )
( -125 , 0 )
( -124 , 0 )
( -123 , 0 )
( -122 , 0 )
( -121 , 0 )
( -120 , 0 )
( -119 , 0 )
( -118 , 0 )
( -117 , 0 )
( -116 , 0 )
( -115 , 0 )
( -114 , 0 )
( -113 , 0 )
( -112 , 0 )
( -111 , 0 )
( -110 , 0 )
( -109 , 0 )
( -108 , 0 )
( -107 , 0 )
( -106 , 0 )
( -105 , 0 )
( -104 , 0 )
( -103 , 0 )
( -102 , 0 )
( -101 , 0 )
( -100 , 0 )
( -99 , 0 )
( -98 ,0 )
( -97 , 0 )
( -96 , 0.000160457 )
( -95 , 0.000118009 )
( -94 , 0.000166654 )
( -93 , 0.000158623 )
( -92 , 0.000224349 )
( -91 , 0.000181983 )
( -90 , 0.000305394 )
( -89 , 0.000217623 )
( -88 , 0.000346303 )
( -87 , 0.000295240 )
( -86 , 0.000382650 )
( -85 , 0.000331610 )
( -84 , 0.000556022 )
( -83 , 0.000378994 )
( -82 , 0.000542396 )
( -81 , 0.000507806 )
( -80 , 0.000726802 )
( -79 , 0.000524937 )
( -78 , 0.000838315 )
( -77 , 0.000628860 )
( -76 , 0.000927977 )
( -75 , 0.000823377 )
( -74 , 0.000999626 )
( -73 , 0.000812042 )
( -72 , 0.00140780 )
( -71 , 0.000929097 )
( -70 , 0.00138358 )
( -69 , 0.00118197 )
( -68 , 0.00158383 )
( -67 , 0.00119421 )
( -66 , 0.00187836 )
( -65 , 0.00140812 )
( -64 , 0.00204927 )
( -63 , 0.00173911 )
( -62 , 0.00209869 )
( -61 , 0.00168391 )
( -60 , 0.00290623 )
( -59 , 0.00187240 )
( -58 , 0.00259246 )
( -57 , 0.00231823 )
( -56 , 0.00317847 )
( -55 , 0.00240081 )
( -54 , 0.00357928 )
( -53 , 0.00251179 )
( -52 , 0.00372453 )
( -51 , 0.00308685 )
( -50 , 0.00395339 )
( -49 , 0.00307551 )
( -48 , 0.00503816 )
( -47 , 0.00327781 )
( -46 , 0.00450637 )
( -45 , 0.00419748 )
( -44 , 0.00526819 )
( -43 , 0.00385696 )
( -42 , 0.00601219 )
( -41 , 0.00416929 )
( -40 , 0.00643218 )
( -39 , 0.00504029 )
( -38 , 0.00613423 )
( -37 , 0.00481820 )
( -36 , 0.00796154 )
( -35 , 0.00549076 )
( -34 , 0.00705419 )
( -33 , 0.00620941 )
( -32 , 0.00819059 )
( -31 , 0.00589658 )
( -30 , 0.00928617 )
( -29 , 0.00626989 )
( -28 , 0.00926764 )
( -27 , 0.00752109 )
( -26 , 0.00904147 )
( -25 , 0.00732344 )
( -24 , 0.0115267 )
( -23 , 0.00742592 )
( -22 , 0.0100858 )
( -21 , 0.00890391 )
( -20 , 0.0117309 )
( -19 , 0.00821902 )
( -18 , 0.0124915 )
( -17 , 0.00860599 )
( -16 , 0.0125989 )
( -15 , 0.0104174 )
( -14 , 0.0122899 )
( -13 , 0.00937515 )
( -12 , 0.0149335 )
( -11 , 0.00975925 )
( -10 , 0.0134726 )
( -9 , 0.0113276 )
( -8 , 0.0145038 )
( -7 , 0.0105092 )
( -6 , 0.0153348 )
( -5 , 0.0109904 )
( -4 , 0.0150248 )
( -3 , 0.0120045 )
( -2 , 0.0142432 )
( -1 , 0.0108885 )
( 0 , 0.0189306 )
( 1 , 0.0108885 )
( 2 , 0.0142432 )
( 3 , 0.0120045 )
( 4 , 0.0150248 )
( 5 , 0.0109904 )
( 6 , 0.0153348 )
( 7 , 0.0105092 )
( 8 , 0.0145038 )
( 9 , 0.0113276 )
( 10 , 0.0134726 )
( 11 , 0.00975925 )
( 12 , 0.0149335 )
( 13 , 0.00937515 )
( 14 , 0.0122899 )
( 15 , 0.0104174 )
( 16 , 0.0125989 )
( 17 , 0.00860599 )
( 18 , 0.0124915 )
( 19 , 0.00821902 )
( 20 , 0.0117309 )
( 21 , 0.00890391 )
( 22 , 0.0100858 )
( 23 , 0.00742592 )
( 24 , 0.0115267 )
( 25 , 0.00732344 )
( 26 , 0.00904147 )
( 27 , 0.00752109 )
( 28 , 0.00926764 )
( 29 , 0.00626989 )
( 30 , 0.00928617 )
( 31 , 0.00589658 )
( 32 , 0.00819059 )
( 33 , 0.00620941 )
( 34 , 0.00705419 )
( 35 , 0.00549076 )
( 36 , 0.00796154 )
( 37 , 0.00481820 )
( 38 , 0.00613423 )
( 39 , 0.00504029 )
( 40 , 0.00643218 )
( 41 , 0.00416929 )
( 42 , 0.00601219 )
( 43 , 0.00385696 )
( 44 , 0.00526819 )
( 45 , 0.00419748 )
( 46 , 0.00450637 )
( 47 , 0.00327781 )
( 48 , 0.00503816 )
( 49 , 0.00307551 )
( 50 , 0.00395339 )
( 51 , 0.00308685 )
( 52 , 0.00372453 )
( 53 , 0.00251179 )
( 54 , 0.00357928 )
( 55 , 0.00240081 )
( 56 , 0.00317847 )
( 57 , 0.00231823 )
( 58 , 0.00259246 )
( 59 , 0.00187240 )
( 60 , 0.00290623 )
( 61 , 0.00168391 )
( 62 , 0.00209869 )
( 63 , 0.00173911 )
( 64 , 0.00204927 )
( 65 , 0.00140812 )
( 66 , 0.00187836 )
( 67 , 0.00119421 )
( 68 , 0.00158383 )
( 69 , 0.00118197 )
( 70 , 0.00138358 )
( 71 , 0.000929097 )
( 72 , 0.00140780 )
( 73 , 0.000812042 )
( 74 , 0.000999626 )
( 75 , 0.000823377 )
( 76 , 0.000927977 )
( 77 , 0.000628860 )
( 78 , 0.000838315 )
( 79 , 0.000524937 )
( 80 , 0.000726802 )
( 81 , 0.000507806 )
( 82 , 0.000542396 )
( 83 , 0.000378994 )
( 84 , 0.000556022 )
( 85 , 0.000331610 )
( 86 , 0.000382650 )
( 87 , 0.000295240 )
( 88 , 0.000346303 )
( 89 , 0.000217623 )
( 90 , 0.000305394 )
( 91 , 0.000181983 )
( 92 , 0.000224349 )
( 93 , 0.000158623 )
( 94 , 0.000166654 )
( 95 , 0.000118009 )
( 96 , 0.000160457 )
( 97 , 0 )
( 98 , 0 )
( 99 , 0 )
( 100 , 0 )
( 101 , 0 )
( 102 , 0 )
( 103 , 0 )
( 104 , 0 )
( 105 , 0 )
( 106 , 0)
( 107 , 0 )
( 108 , 0 )
( 109 , 0 )
( 110 ,0 )
( 111 , 0 )
( 112 , 0 )
( 113 , 0 )
( 114 , 0 )
( 115 , 0 )
( 116 , 0 )
( 117 , 0 )
( 118 , 0 )
( 119 , 0 )
( 120 , 0 )
( 121 , 0 )
( 122 , 0 )
( 123 , 0 )
( 124 , 0 )
( 125 , 0 )
( 126 , 0 )
( 127 , 0 )

    };

        \addplot[only marks, mark=star, mark size=0.5mm] coordinates {
( -127 , 0 )
( -126 , 0 )
( -125 , 0 )
( -124 , 0 )
( -123 , 0 )
( -122 , 0 )
( -121 , 0 )
( -120 , 0 )
( -119 , 0 )
( -118 , 0 )
( -117 , 0 )
( -116 , 0 )
( -115 , 0 )
( -114 , 0 )
( -113 , 0 )
( -112 , 0 )
( -111 , 0 )
( -110 , 0 )
( -109 , 0 )
( -108 , 0 )
( -107 , 0 )
( -106 , 0 )
( -105 , 0 )
( -104 , 0 )
( -103 , 0 )
( -102 , 0 )
( -101 , 0 )
( -100 , 0 )
( -99 , 0 )
( -98 , 0.000103855 )
( -97 , 0 )
( -96 , 0.000159551 )
( -95 , 0.000118361 )
( -94 , 0.000167973 )
( -93 , 0.000160173 )
( -92 , 0.000226491 )
( -91 , 0.000182148 )
( -90 , 0.000303050 )
( -89 , 0.000218453 )
( -88 , 0.000344933 )
( -87 , 0.000297272 )
( -86 , 0.000382665 )
( -85 , 0.000332730 )
( -84 , 0.000566971 )
( -83 , 0.000379630 )
( -82 , 0.000543556 )
( -81 , 0.000503188 )
( -80 , 0.000718760 )
( -79 , 0.000525731 )
( -78 , 0.000842466 )
( -77 , 0.000629717 )
( -76 , 0.000938981 )
( -75 , 0.000826663 )
( -74 , 0.000999510 )
( -73 , 0.000812491 )
( -72 , 0.00138575 )
( -71 , 0.000929895 )
( -70 , 0.00138708 )
( -69 , 0.00118777 )
( -68 , 0.00160137 )
( -67 , 0.00119365 )
( -66 , 0.00188793 )
( -65 , 0.00140858 )
( -64 , 0.00202495 )
( -63 , 0.00172512 )
( -62 , 0.00209673 )
( -61 , 0.00168258 )
( -60 , 0.00294495 )
( -59 , 0.00187048 )
( -58 , 0.00259108 )
( -57 , 0.00232791 )
( -56 , 0.00315768 )
( -55 , 0.00239908 )
( -54 , 0.00353882 )
( -53 , 0.00251110 )
( -52 , 0.00375836 )
( -51 , 0.00310102 )
( -50 , 0.00394567 )
( -49 , 0.00307271 )
( -48 , 0.00500523 )
( -47 , 0.00327629 )
( -46 , 0.00450524 )
( -45 , 0.00415829 )
( -44 , 0.00531087 )
( -43 , 0.00385433 )
( -42 , 0.00603607 )
( -41 , 0.00416771 )
( -40 , 0.00639703 )
( -39 , 0.00506303 )
( -38 , 0.00613045 )
( -37 , 0.00481755 )
( -36 , 0.00795624 )
( -35 , 0.00548849 )
( -34 , 0.00705189 )
( -33 , 0.00623853 )
( -32 , 0.00809337 )
( -31 , 0.00589033 )
( -30 , 0.00932774 )
( -29 , 0.00626500 )
( -28 , 0.00934345 )
( -27 , 0.00744108 )
( -26 , 0.00904272 )
( -25 , 0.00731377 )
( -24 , 0.0115098 )
( -23 , 0.00741965 )
( -22 , 0.0100844 )
( -21 , 0.00893806 )
( -20 , 0.0118355 )
( -19 , 0.00821833 )
( -18 , 0.0123731 )
( -17 , 0.00860249 )
( -16 , 0.0124631 )
( -15 , 0.0104622 )
( -14 , 0.0122793 )
( -13 , 0.00936978 )
( -12 , 0.0151340 )
( -11 , 0.00975676 )
( -10 , 0.0134828 )
( -9 , 0.0112211 )
( -8 , 0.0144120 )
( -7 , 0.0104988 )
( -6 , 0.0153950 )
( -5 , 0.0109857 )
( -4 , 0.0151566 )
( -3 , 0.0120620 )
( -2 , 0.0142424 )
( -1 , 0.0108822 )
( 0 , 0.0190100 )
( 1 , 0.0108822 )
( 2 , 0.0142424 )
( 3 , 0.0120620 )
( 4 , 0.0151566 )
( 5 , 0.0109857 )
( 6 , 0.0153950 )
( 7 , 0.0104988 )
( 8 , 0.0144120 )
( 9 , 0.0112211 )
( 10 , 0.0134828 )
( 11 , 0.00975676 )
( 12 , 0.0151340 )
( 13 , 0.00936978 )
( 14 , 0.0122793 )
( 15 , 0.0104622 )
( 16 , 0.0124631 )
( 17 , 0.00860249 )
( 18 , 0.0123731 )
( 19 , 0.00821833 )
( 20 , 0.0118355 )
( 21 , 0.00893806 )
( 22 , 0.0100844 )
( 23 , 0.00741965 )
( 24 , 0.0115098 )
( 25 , 0.00731377 )
( 26 , 0.00904272 )
( 27 , 0.00744108 )
( 28 , 0.00934345 )
( 29 , 0.00626500 )
( 30 , 0.00932774 )
( 31 , 0.00589033 )
( 32 , 0.00809337 )
( 33 , 0.00623853 )
( 34 , 0.00705189 )
( 35 , 0.00548849 )
( 36 , 0.00795624 )
( 37 , 0.00481755 )
( 38 , 0.00613045 )
( 39 , 0.00506303 )
( 40 , 0.00639703 )
( 41 , 0.00416771 )
( 42 , 0.00603607 )
( 43 , 0.00385433 )
( 44 , 0.00531087 )
( 45 , 0.00415829 )
( 46 , 0.00450524 )
( 47 , 0.00327629 )
( 48 , 0.00500523 )
( 49 , 0.00307271 )
( 50 , 0.00394567 )
( 51 , 0.00310102 )
( 52 , 0.00375836 )
( 53 , 0.00251110 )
( 54 , 0.00353882 )
( 55 , 0.00239908 )
( 56 , 0.00315768 )
( 57 , 0.00232791 )
( 58 , 0.00259108 )
( 59 , 0.00187048 )
( 60 , 0.00294495 )
( 61 , 0.00168258 )
( 62 , 0.00209673 )
( 63 , 0.00172512 )
( 64 , 0.00202495 )
( 65 , 0.00140858 )
( 66 , 0.00188793 )
( 67 , 0.00119365 )
( 68 , 0.00160137 )
( 69 , 0.00118777 )
( 70 , 0.00138708 )
( 71 , 0.000929895 )
( 72 , 0.00138575 )
( 73 , 0.000812491 )
( 74 , 0.000999510 )
( 75 , 0.000826663 )
( 76 , 0.000938981 )
( 77 , 0.000629717 )
( 78 , 0.000842466 )
( 79 , 0.000525731 )
( 80 , 0.000718760 )
( 81 , 0.000503188 )
( 82 , 0.000543556 )
( 83 , 0.000379630 )
( 84 , 0.000566971 )
( 85 , 0.000332730 )
( 86 , 0.000382665 )
( 87 , 0.000297272 )
( 88 , 0.000344933 )
( 89 , 0.000218453 )
( 90 , 0.000303050 )
( 91 , 0.000182148 )
( 92 , 0.000226491 )
( 93 , 0.000160173 )
( 94 , 0.000167973 )
( 95 , 0.000118361 )
( 96 , 0.000159551 )
( 97 , 0 )
( 98 , 0.000103855 )
( 99 , 0 )
( 100 , 0 )
( 101 , 0 )
( 102 ,0 )
( 103 , 0 )
( 104 , 0 )
( 105 , 0 )
( 106 , 0 )
( 107 , 0 )
( 108 , 0 )
( 109 , 0 )
( 110 , 0 )
( 111 , 0 )
( 112 , 0 )
( 113 , 0 )
( 114 , 0 )
( 115 , 0 )
( 116 , 0 )
( 117 , 0 )
( 118 , 0 )
( 119 , 0 )
( 120 , 0 )
( 121 , 0 )
( 122 , 0 )
( 123 , 0 )
( 124 , 0 )
( 125 , 0 )
( 126 , 0 )
( 127 , 0 )

    };
\addplot[blue] coordinates {
( -127 , 0 )
( -126 , 0 )
( -125 , 0 )
( -124 , 0 )
( -123 , 0 )
( -122 , 0 )
( -121 , 0 )
( -120 , 0 )
( -119 , 0 )
( -118 , 0 )
( -117 , 0 )
( -116 , 0 )
( -115 , 0 )
( -114 , 0 )
( -113 , 0 )
( -112 , 0 )
( -111 , 0 )
( -110 , 0 )
( -109 , 0 )
( -108 , 0 )
( -107 ,0 )
( -106 , 0 )
( -105 , 0 )
( -104 , 0 )
( -103 , 0 )
( -102 , 0 )
( -101 ,0 )
( -100 , 0 )
( -99 , 0 )
( -98 , 0 )
( -97 , 0.000109854 )
( -96 , 0.000124895 )
( -95 , 0.000141488 )
( -94 , 0.000159237 )
( -93 , 0.000178166 )
( -92 , 0.000198730 )
( -91 , 0.000221364 )
( -90 , 0.000246095 )
( -89 , 0.000272608 )
( -88 , 0.000300707 )
( -87 , 0.000330606 )
( -86 , 0.000362765 )
( -85 , 0.000397497 )
( -84 , 0.000434736 )
( -83 , 0.000474204 )
( -82 , 0.000515777 )
( -81 , 0.000559675 )
( -80 , 0.000606320 )
( -79 , 0.000656013 )
( -78 , 0.000708735 )
( -77 , 0.000764245 )
( -76 , 0.000822385 )
( -75 , 0.000883278 )
( -74 , 0.000947290 )
( -73 , 0.00101477 )
( -72 , 0.00108583 )
( -71 , 0.00116029 )
( -70 , 0.00123795 )
( -69 , 0.00131874 )
( -68 , 0.00140291 )
( -67 , 0.00149082 )
( -66 , 0.00158273 )
( -65 , 0.00167866 )
( -64 , 0.00177839 )
( -63 , 0.00188170 )
( -62 , 0.00198860 )
( -61 , 0.00209933 )
( -60 , 0.00221422 )
( -59 , 0.00233351 )
( -58 , 0.00245717 )
( -57 , 0.00258496 )
( -56 , 0.00271666 )
( -55 , 0.00285221 )
( -54 , 0.00299180 )
( -53 , 0.00313575 )
( -52 , 0.00328429 )
( -51 , 0.00343743 )
( -50 , 0.00359491 )
( -49 , 0.00375644 )
( -48 , 0.00392186 )
( -47 , 0.00409124 )
( -46 , 0.00426486 )
( -45 , 0.00444297 )
( -44 , 0.00462565 )
( -43 , 0.00481267 )
( -42 , 0.00500365 )
( -41 , 0.00519828 )
( -40 , 0.00539644 )
( -39 , 0.00559829 )
( -38 , 0.00580409 )
( -37 , 0.00601399 )
( -36 , 0.00622787 )
( -35 , 0.00644530 )
( -34 , 0.00666580 )
( -33 , 0.00688899 )
( -32 , 0.00711476 )
( -31 , 0.00734329 )
( -30 , 0.00757478 )
( -29 , 0.00780926 )
( -28 , 0.00804637 )
( -27 , 0.00828551 )
( -26 , 0.00852594 )
( -25 , 0.00876717 )
( -24 , 0.00900909 )
( -23 , 0.00925180 )
( -22 , 0.00949544 )
( -21 , 0.00973986 )
( -20 , 0.00998445 )
( -19 , 0.0102282 )
( -18 , 0.0104700 )
( -17 , 0.0107092 )
( -16 , 0.0109454 )
( -15 , 0.0111786 )
( -14 , 0.0114091 )
( -13 , 0.0116362 )
( -12 , 0.0118588 )
( -11 , 0.0120752 )
( -10 , 0.0122834 )
( -9 , 0.0124819 )
( -8 , 0.0126701 )
( -7 , 0.0128478 )
( -6 , 0.0130148 )
( -5 , 0.0131700 )
( -4 , 0.0133101 )
( -3 , 0.0134299 )
( -2 , 0.0135226 )
( -1 , 0.0135815 )
( 0 , 0.0136018 )
( 1 , 0.0135815 )
( 2 , 0.0135226 )
( 3 , 0.0134299 )
( 4 , 0.0133101 )
( 5 , 0.0131700 )
( 6 , 0.0130148 )
( 7 , 0.0128478 )
( 8 , 0.0126701 )
( 9 , 0.0124819 )
( 10 , 0.0122834 )
( 11 , 0.0120752 )
( 12 , 0.0118588 )
( 13 , 0.0116362 )
( 14 , 0.0114091 )
( 15 , 0.0111786 )
( 16 , 0.0109454 )
( 17 , 0.0107092 )
( 18 , 0.0104700 )
( 19 , 0.0102282 )
( 20 , 0.00998445 )
( 21 , 0.00973986 )
( 22 , 0.00949544 )
( 23 , 0.00925180 )
( 24 , 0.00900909 )
( 25 , 0.00876717 )
( 26 , 0.00852594 )
( 27 , 0.00828551 )
( 28 , 0.00804637 )
( 29 , 0.00780926 )
( 30 , 0.00757478 )
( 31 , 0.00734329 )
( 32 , 0.00711476 )
( 33 , 0.00688899 )
( 34 , 0.00666580 )
( 35 , 0.00644530 )
( 36 , 0.00622787 )
( 37 , 0.00601399 )
( 38 , 0.00580409 )
( 39 , 0.00559829 )
( 40 , 0.00539644 )
( 41 , 0.00519828 )
( 42 , 0.00500365 )
( 43 , 0.00481267 )
( 44 , 0.00462565 )
( 45 , 0.00444297 )
( 46 , 0.00426486 )
( 47 , 0.00409124 )
( 48 , 0.00392186 )
( 49 , 0.00375644 )
( 50 , 0.00359491 )
( 51 , 0.00343743 )
( 52 , 0.00328429 )
( 53 , 0.00313575 )
( 54 , 0.00299180 )
( 55 , 0.00285221 )
( 56 , 0.00271666 )
( 57 , 0.00258496 )
( 58 , 0.00245717 )
( 59 , 0.00233351 )
( 60 , 0.00221422 )
( 61 , 0.00209933 )
( 62 , 0.00198860 )
( 63 , 0.00188170 )
( 64 , 0.00177839 )
( 65 , 0.00167866 )
( 66 , 0.00158273 )
( 67 , 0.00149082 )
( 68 , 0.00140291 )
( 69 , 0.00131874 )
( 70 , 0.00123795 )
( 71 , 0.00116029 )
( 72 , 0.00108583 )
( 73 , 0.00101477 )
( 74 , 0.000947290 )
( 75 , 0.000883278 )
( 76 , 0.000822385 )
( 77 , 0.000764245 )
( 78 , 0.000708735 )
( 79 , 0.000656013 )
( 80 , 0.000606320 )
( 81 , 0.000559675 )
( 82 , 0.000515777 )
( 83 , 0.000474204 )
( 84 , 0.000434736 )
( 85 , 0.000397497 )
( 86 , 0.000362765 )
( 87 , 0.000330606 )
( 88 , 0.000300707 )
( 89 , 0.000272608 )
( 90 , 0.000246095 )
( 91 , 0.000221364 )
( 92 , 0.000198730 )
( 93 , 0.000178166 )
( 94 , 0.000159237 )
( 95 , 0.000141488 )
( 96 , 0.000124895 )
( 97 , 0.000109854 )
( 98 , 0 )
( 99 , 0 )
( 100 , 0 )
( 101 , 0 )
( 102 , 0 )
( 103 , 0 )
( 104 , 0 )
( 105 , 0 )
( 106 , 0 )
( 107 , 0 )
( 108 , 0 )
( 109 , 0 )
( 110 , 0 )
( 111 , 0 )
( 112 , 0 )
( 113 , 0 )
( 114 , 0 )
( 115 , 0 )
( 116 , 0 )
( 117 , 0 )
( 118 , 0 )
( 119 , 0 )
( 120 , 0 )
( 121 , 0 )
( 122 , 0 )
( 123 , 0 )
( 124 , 0 )
( 125 , 0 )
( 126 , 0 )
( 127 , 0 )
};  
  \end{axis}
\end{tikzpicture}
\end{center}
\caption{Case $g=2$ and $q=1009$. The red dots are the values of $H'$. The black stars are the values of the approximation of $\nu'(q,t)$. The blue graph is the approximation of the Sato-Tate density. In this case, $d(H',\nu')\approx0.00439$ and $d(H',\nu_\infty)\approx0.15528$.}
\end{figure}


\begin{figure}[h]
\begin{center}
\begin{tikzpicture}[xscale=1.1, yscale=1.1]
  \begin{axis}
    [scaled ticks=false, ymin=0, ymax=0.065, ytick={0,0.015,0.03,0.045,0.06}, yticklabels={0,0.015,0.03,0.045,0.06}]
    \addplot[only marks, red, mark size=0.3mm] coordinates {
( -40 , 0 )
( -39 , 0 )
( -38 , 0 )
( -37 , 0 )
( -36 , 0.000015651 )
( -35 , 0.000027662 )
( -34 , 0.000069883 )
( -33 , 0.000098191 )
( -32 , 0.00020152 )
( -31 , 0.00023595 )
( -30 , 0.00055858 )
( -29 , 0.00049850 )
( -28 , 0.00099522 )
( -27 , 0.0011010 )
( -26 , 0.0017053 )
( -25 , 0.0016990 )
( -24 , 0.0033134 )
( -23 , 0.0025856 )
( -22 , 0.0042524 )
( -21 , 0.0044856 )
( -20 , 0.0068780 )
( -19 , 0.0056624 )
( -18 , 0.0099868 )
( -17 , 0.0077327 )
( -16 , 0.012666 )
( -15 , 0.012039 )
( -14 , 0.015587 )
( -13 , 0.013183 )
( -12 , 0.023504 )
( -11 , 0.016556 )
( -10 , 0.025070 )
( -9 , 0.023037 )
( -8 , 0.031537 )
( -7 , 0.024504 )
( -6 , 0.038462 )
( -5 , 0.029061 )
( -4 , 0.041859 )
( -3 , 0.035569 )
( -2 , 0.043142 )
( -1 , 0.034024 )
( 0 , 0.056193 )
( 1 , 0.034024 )
( 2 , 0.043142 )
( 3 , 0.035569 )
( 4 , 0.041859 )
( 5 , 0.029061 )
( 6 , 0.038462 )
( 7 , 0.024504 )
( 8 , 0.031537 )
( 9 , 0.023037 )
( 10 , 0.025070 )
( 11 , 0.016556 )
( 12 , 0.023504 )
( 13 , 0.013183 )
( 14 , 0.015587 )
( 15 , 0.012039 )
( 16 , 0.012666 )
( 17 , 0.0077327 )
( 18 , 0.0099868 )
( 19 , 0.0056624 )
( 20 , 0.0068780 )
( 21 , 0.0044856 )
( 22 , 0.0042524 )
( 23 , 0.0025856 )
( 24 , 0.0033134 )
( 25 , 0.0016990 )
( 26 , 0.0017053 )
( 27 , 0.0011010 )
( 28 , 0.00099522 )
( 29 , 0.00049850 )
( 30 , 0.00055858 )
( 31 , 0.00023595 )
( 32 , 0.00020152 )
( 33 , 0.000098191 )
( 34 , 0.000069883 )
( 35 , 0.000027662 )
( 36 , 0.000015651 )
( 37 , 0 )
( 38 , 0 )
( 39 , 0 )
( 40 , 0 )

    };

        \addplot[only marks, mark=star, mark size=0.5mm] coordinates {
( -40 , 0 )
( -39 , 0 )
( -38 , 0 )
( -37 , 0 )
( -36 , 0 )
( -35 , 0 )
( -34 , 0.000063130 )
( -33 , 0.00010650 )
( -32 , 0.00021866 )
( -31 , 0.00024369 )
( -30 , 0.00056706 )
( -29 , 0.00051927 )
( -28 , 0.0010466 )
( -27 , 0.0011038 )
( -26 , 0.0016918 )
( -25 , 0.0017179 )
( -24 , 0.0033791 )
( -23 , 0.0026131 )
( -22 , 0.0042602 )
( -21 , 0.0045211 )
( -20 , 0.0069269 )
( -19 , 0.0055687 )
( -18 , 0.0097959 )
( -17 , 0.0076584 )
( -16 , 0.012525 )
( -15 , 0.011989 )
( -14 , 0.015477 )
( -13 , 0.013124 )
( -12 , 0.023676 )
( -11 , 0.016483 )
( -10 , 0.024857 )
( -9 , 0.022712 )
( -8 , 0.031127 )
( -7 , 0.024335 )
( -6 , 0.038536 )
( -5 , 0.028808 )
( -4 , 0.041896 )
( -3 , 0.035485 )
( -2 , 0.042942 )
( -1 , 0.033742 )
( 0 , 0.060567 )
( 1 , 0.033742 )
( 2 , 0.042942 )
( 3 , 0.035485 )
( 4 , 0.041896 )
( 5 , 0.028808 )
( 6 , 0.038536 )
( 7 , 0.024335 )
( 8 , 0.031127 )
( 9 , 0.022712 )
( 10 , 0.024857 )
( 11 , 0.016483 )
( 12 , 0.023676 )
( 13 , 0.013124 )
( 14 , 0.015477 )
( 15 , 0.011989 )
( 16 , 0.012525 )
( 17 , 0.0076584 )
( 18 , 0.0097959 )
( 19 , 0.0055687 )
( 20 , 0.0069269 )
( 21 , 0.0045211 )
( 22 , 0.0042602 )
( 23 , 0.0026131 )
( 24 , 0.0033791 )
( 25 , 0.0017179 )
( 26 , 0.0016918 )
( 27 , 0.0011038 )
( 28 , 0.0010466 )
( 29 , 0.00051927 )
( 30 , 0.00056706 )
( 31 , 0.00024369 )
( 32 , 0.00021866 )
( 33 , 0.00010650 )
( 34 , 0.000063130 )
( 35 , 0 )
( 36 , 0 )
( 37 , 0 )
( 38 , 0 )
( 39 , 0 )
( 40 , 0 )

    };
\addplot[blue] coordinates {
( -40 , 0 )
( -39 , 0 )
( -38 , 0 )
( -37 , 0 )
( -36 , 0 )
( -35 , 0 )
( -34 , 0.000060067 )
( -33 , 0.00011658 )
( -32 , 0.00019307 )
( -31 , 0.00030575 )
( -30 , 0.00045696 )
( -29 , 0.00065159 )
( -28 , 0.00090536 )
( -27 , 0.0012181 )
( -26 , 0.0016062 )
( -25 , 0.0020705 )
( -24 , 0.0026261 )
( -23 , 0.0032742 )
( -22 , 0.0040299 )
( -21 , 0.0048919 )
( -20 , 0.0058762 )
( -19 , 0.0069783 )
( -18 , 0.0082139 )
( -17 , 0.0095761 )
( -16 , 0.011076 )
( -15 , 0.012709 )
( -14 , 0.014474 )
( -13 , 0.016374 )
( -12 , 0.018391 )
( -11 , 0.020531 )
( -10 , 0.022769 )
( -9 , 0.025095 )
( -8 , 0.027491 )
( -7 , 0.029914 )
( -6 , 0.032354 )
( -5 , 0.034739 )
( -4 , 0.037030 )
( -3 , 0.039151 )
( -2 , 0.040968 )
( -1 , 0.042390 )
( 0 , 0.042989 )
( 1 , 0.042390 )
( 2 , 0.040968 )
( 3 , 0.039151 )
( 4 , 0.037030 )
( 5 , 0.034739 )
( 6 , 0.032354 )
( 7 , 0.029914 )
( 8 , 0.027491 )
( 9 , 0.025095 )
( 10 , 0.022769 )
( 11 , 0.020531 )
( 12 , 0.018391 )
( 13 , 0.016374 )
( 14 , 0.014474 )
( 15 , 0.012709 )
( 16 , 0.011076 )
( 17 , 0.0095761 )
( 18 , 0.0082139 )
( 19 , 0.0069783 )
( 20 , 0.0058762 )
( 21 , 0.0048919 )
( 22 , 0.0040299 )
( 23 , 0.0032742 )
( 24 , 0.0026261 )
( 25 , 0.0020705 )
( 26 , 0.0016062 )
( 27 , 0.0012181 )
( 28 , 0.00090536 )
( 29 , 0.00065159 )
( 30 , 0.00045696 )
( 31 , 0.00030575 )
( 32 , 0.00019307 )
( 33 , 0.00011658 )
( 34 , 0.000060067 )
( 35 , 0 )
( 36 , 0 )
( 37 , 0 )
( 38 , 0 )
( 39 , 0 )
( 40 , 0 )

};  
  \end{axis}
\end{tikzpicture}
\caption{Case $g=2$ and $q=101$. The red dots are the values of $H'$. The black stars are the values of the approximation of $\nu'(q,t)$. The blue graph is the approximation of the Sato-Tate density. In this case, $d(H',\nu')\approx 0.01117$ and $d(H',\nu_\infty)\approx0.15166$.}
\end{center}
\end{figure}


\FloatBarrier

\begin{figure}
\begin{center}
\begin{tikzpicture}[xscale=1.1, yscale=1.1]
  \begin{axis} 
  [scaled ticks=false, ymin=0, ymax=0.06, ytick={0,0.015,0.03,0.045,0.06}, yticklabels={0,0.015,0.03,0.045,0.06}]
\addplot[only marks, red, mark size=0.3mm] coordinates {
( -43 , 0 )
( -42 , 0 )
( -41 , 0 )
( -40 , 0 )
( -39 , 0 )
( -38 , 0 )
( -37 , 0 )
( -36 , 0 )
( -35 , 0 )
( -34 , 0 )
( -33 , 0 )
( -32 , 0 )
( -31 , 0 )
( -30 , 0 )
( -29 , 0 )
( -28 , 0 )
( -27 , 0 )
( -26 , 0 )
( -25 , 0 )
( -24 , 0.0001366 )
( -23 , 0.0002253 )
( -22 , 0.0003775 )
( -21 , 0.0005913 )
( -20 , 0.0009333 )
( -19 , 0.001384 )
( -18 , 0.002101 )
( -17 , 0.002967 )
( -16 , 0.004294 )
( -15 , 0.005864 )
( -14 , 0.008132 )
( -13 , 0.01061 )
( -12 , 0.01419 )
( -11 , 0.01763 )
( -10 , 0.02239 )
( -9 , 0.02663 )
( -8 , 0.03203 )
( -7 , 0.03615 )
( -6 , 0.04178 )
( -5 , 0.04487 )
( -4 , 0.04945 )
( -3 , 0.05133 )
( -2 , 0.05418 )
( -1 , 0.05391 )
( 0 , 0.05516 )
( 1 , 0.05278 )
( 2 , 0.05198 )
( 3 , 0.04831 )
( 4 , 0.04578 )
( 5 , 0.04100 )
( 6 , 0.03782 )
( 7 , 0.03261 )
( 8 , 0.02896 )
( 9 , 0.02434 )
( 10 , 0.02087 )
( 11 , 0.01692 )
( 12 , 0.01415 )
( 13 , 0.01111 )
( 14 , 0.009007 )
( 15 , 0.006920 )
( 16 , 0.005435 )
( 17 , 0.004048 )
( 18 , 0.003109 )
( 19 , 0.002236 )
( 20 , 0.001653 )
( 21 , 0.001155 )
( 22 , 0.0008196 )
( 23 , 0.0005468 )
( 24 , 0.0003722 )
( 25 , 0.0002355 )
( 26 , 0.0001513 )
( 27 , 0 )
( 28 , 0 )
( 29 , 0 )
( 30 , 0 )
( 31 , 0 )
( 32 , 0 )
( 33 , 0 )
( 34 , 0 )
( 35 , 0 )
( 36 , 0 )
( 37 , 0 )
( 38 , 0 )
( 39 , 0 )
( 40 , 0 )
( 41 , 0 )
( 42 , 0 )
( 43 , 0 )
    };

        \addplot[only marks, mark=star, mark size=0.3mm] coordinates {

( -43 , 0 )
( -42 , 0 )
( -41 , 0 )
( -40 , 0 )
( -39 , 0 )
( -38 , 0 )
( -37 , 0 )
( -36 , 0 )
( -35 , 0 )
( -34 , 0 )
( -33 , 0 )
( -32 , 0 )
( -31 , 0 )
( -30 , 0 )
( -29 , 0 )
( -28 , 0 )
( -27 , 0 )
( -26 , 0 )
( -25 , 0.0001556 )
( -24 , 0.0002556 )
( -23 , 0.0003856 )
( -22 , 0.0006006 )
( -21 , 0.0008723 )
( -20 , 0.001299 )
( -19 , 0.001808 )
( -18 , 0.002613 )
( -17 , 0.003502 )
( -16 , 0.004882 )
( -15 , 0.006380 )
( -14 , 0.008591 )
( -13 , 0.01084 )
( -12 , 0.01421 )
( -11 , 0.01724 )
( -10 , 0.02167 )
( -9 , 0.02543 )
( -8 , 0.03057 )
( -7 , 0.03430 )
( -6 , 0.03988 )
( -5 , 0.04283 )
( -4 , 0.04773 )
( -3 , 0.04970 )
( -2 , 0.05317 )
( -1 , 0.05322 )
( 0 , 0.05529 )
( 1 , 0.05322 )
( 2 , 0.05317 )
( 3 , 0.04970 )
( 4 , 0.04773 )
( 5 , 0.04283 )
( 6 , 0.03988 )
( 7 , 0.03430 )
( 8 , 0.03057 )
( 9 , 0.02543 )
( 10 , 0.02167 )
( 11 , 0.01724 )
( 12 , 0.01421 )
( 13 , 0.01084 )
( 14 , 0.008591 )
( 15 , 0.006380 )
( 16 , 0.004882 )
( 17 , 0.003502 )
( 18 , 0.002613 )
( 19 , 0.001808 )
( 20 , 0.001299 )
( 21 , 0.0008723 )
( 22 , 0.0006006 )
( 23 , 0.0003856 )
( 24 , 0.0002556 )
( 25 , 0.0001556 )
( 26 , 0 )
( 27 , 0 )
( 28 , 0 )
( 29 , 0 )
( 30 , 0 )
( 31 , 0 )
( 32 , 0 )
( 33 , 0 )
( 34 , 0 )
( 35 , 0 )
( 36 , 0 )
( 37 , 0 )
( 38 , 0 )
( 39 , 0 )
( 40 , 0 )
( 41 , 0 )
( 42 , 0 )
( 43 , 0 )

    };
\addplot[blue] coordinates {
( -43 , 0 )
( -42 , 0 )
( -41 , 0 )
( -40 , 0 )
( -39 , 0 )
( -38 , 0 )
( -37 , 0 )
( -36 , 0 )
( -35 , 0 )
( -34 , 0 )
( -33 , 0 )
( -32 , 0 )
( -31 , 0 )
( -30 , 0 )
( -29 , 0 )
( -28 , 0 )
( -27 , 0 )
( -26 , 0 )
( -25 , 0.0001578 )
( -24 , 0.0002517 )
( -23 , 0.0003911 )
( -22 , 0.0005937 )
( -21 , 0.0008817 )
( -20 , 0.001284 )
( -19 , 0.001834 )
( -18 , 0.002574 )
( -17 , 0.003552 )
( -16 , 0.004825 )
( -15 , 0.006450 )
( -14 , 0.008489 )
( -13 , 0.01099 )
( -12 , 0.01400 )
( -11 , 0.01749 )
( -10 , 0.02142 )
( -9 , 0.02571 )
( -8 , 0.03022 )
( -7 , 0.03480 )
( -6 , 0.03927 )
( -5 , 0.04345 )
( -4 , 0.04716 )
( -3 , 0.05024 )
( -2 , 0.05255 )
( -1 , 0.05397 )
( 0 , 0.05446 )
( 1 , 0.05397 )
( 2 , 0.05255 )
( 3 , 0.05024 )
( 4 , 0.04716 )
( 5 , 0.04345 )
( 6 , 0.03927 )
( 7 , 0.03480 )
( 8 , 0.03022 )
( 9 , 0.02571 )
( 10 , 0.02142 )
( 11 , 0.01749 )
( 12 , 0.01400 )
( 13 , 0.01099 )
( 14 , 0.008489 )
( 15 , 0.006450 )
( 16 , 0.004825 )
( 17 , 0.003552 )
( 18 , 0.002574 )
( 19 , 0.001834 )
( 20 , 0.001284 )
( 21 , 0.0008817 )
( 22 , 0.0005937 )
( 23 , 0.0003911 )
( 24 , 0.0002517 )
( 25 , 0.0001578 )
( 26 , 0 )
( 27 , 0 )
( 28 , 0 )
( 29 , 0 )
( 30 , 0 )
( 31 , 0 )
( 32 , 0 )
( 33 , 0 )
( 34 , 0 )
( 35 , 0 )
( 36 , 0 )
( 37 , 0 )
( 38 , 0 )
( 39 , 0 )
( 40 , 0 )
( 41 , 0 )
( 42 , 0 )
( 43 , 0 )

};  
  \end{axis}
\end{tikzpicture}
\end{center}
\caption{Case $g=3$ and $q=53$. The red dots are the values of $H'$. The black stars are the values of the approximation of $\nu'(q,t)$. The blue graph is the approximation of the Sato-Tate density. In this case, $d(H',\nu')\approx0.03842$ and $d(H',\nu_\infty)\approx0.03940$.
}
\label{fig:53}
\end{figure}


\begin{figure}
\begin{center}
\begin{tikzpicture}[xscale=1.1, yscale=1.1]
  \begin{axis}
    [ymin=0, ymax=0.3, ytick={0,0.1,0.2,0.3}, yticklabels={0,0.1,0.2,0.3}]
    \addplot[only marks, red, mark size=0.3mm] coordinates {
( -8 , 0 )
( -7 , 0 )
( -6 , 0.0040000 )
( -5 , 0.012800 )
( -4 , 0.045000 )
( -3 , 0.070666 )
( -2 , 0.13200 )
( -1 , 0.14800 )
( 0 , 0.17507 )
( 1 , 0.14800 )
( 2 , 0.13200 )
( 3 , 0.070666 )
( 4 , 0.045000 )
( 5 , 0.012800 )
( 6 , 0.0040000 )
( 7 , 0 )
( 8 , 0 )

    };

        \addplot[only marks, mark=star, mark size=0.5mm] coordinates {
( -8 , 0 )
( -7 , 0.00090298 )
( -6 , 0.0064755 )
( -5 , 0.013400 )
( -4 , 0.041195 )
( -3 , 0.061752 )
( -2 , 0.11628 )
( -1 , 0.12634 )
( 0 , 0.26730 )
( 1 , 0.12634 )
( 2 , 0.11628 )
( 3 , 0.061752 )
( 4 , 0.041195 )
( 5 , 0.013400 )
( 6 , 0.0064755 )
( 7 , 0.00090298 )
( 8 , 0 )

    };
\addplot[blue] coordinates {
( -8 , 0 )
( -79/10 , 0 )
( -39/5 , 0 )
( -77/10 , 0 )
( -38/5 , 0.00036331 )
( -15/2 , 0.00025034 )
( -37/5 , 0.00041402 )
( -73/10 , 0.00054269 )
( -36/5 , 0.00072535 )
( -71/10 , 0.00090992 )
( -7 , 0.0011300 )
( -69/10 , 0.0013683 )
( -34/5 , 0.0016434 )
( -67/10 , 0.0019666 )
( -33/5 , 0.0023256 )
( -13/2 , 0.0027215 )
( -32/5 , 0.0031729 )
( -63/10 , 0.0036829 )
( -31/5 , 0.0042408 )
( -61/10 , 0.0048502 )
( -6 , 0.0055278 )
( -59/10 , 0.0062779 )
( -29/5 , 0.0070911 )
( -57/10 , 0.0079681 )
( -28/5 , 0.0089239 )
( -11/2 , 0.0099672 )
( -27/5 , 0.011091 )
( -53/10 , 0.012292 )
( -26/5 , 0.013579 )
( -51/10 , 0.014965 )
( -5 , 0.016452 )
( -49/10 , 0.018030 )
( -24/5 , 0.019701 )
( -47/10 , 0.021477 )
( -23/5 , 0.023368 )
( -9/2 , 0.025370 )
( -22/5 , 0.027475 )
( -43/10 , 0.029686 )
( -21/5 , 0.032016 )
( -41/10 , 0.034472 )
( -4 , 0.037047 )
( -39/10 , 0.039735 )
( -19/5 , 0.042536 )
( -37/10 , 0.045462 )
( -18/5 , 0.048520 )
( -7/2 , 0.051703 )
( -17/5 , 0.055001 )
( -33/10 , 0.058413 )
( -16/5 , 0.061945 )
( -31/10 , 0.065609 )
( -3 , 0.069397 )
( -29/10 , 0.073295 )
( -14/5 , 0.077295 )
( -27/10 , 0.081402 )
( -13/5 , 0.085626 )
( -5/2 , 0.089962 )
( -12/5 , 0.094394 )
( -23/10 , 0.098903 )
( -11/5 , 0.10349 )
( -21/10 , 0.10816 )
( -2 , 0.11291 )
( -19/10 , 0.11773 )
( -9/5 , 0.12258 )
( -17/10 , 0.12746 )
( -8/5 , 0.13236 )
( -3/2 , 0.13728 )
( -7/5 , 0.14222 )
( -13/10 , 0.14713 )
( -6/5 , 0.15197 )
( -11/10 , 0.15672 )
( -1 , 0.16140 )
( -9/10 , 0.16598 )
( -4/5 , 0.17042 )
( -7/10 , 0.17465 )
( -3/5 , 0.17860 )
( -1/2 , 0.18225 )
( -2/5 , 0.18560 )
( -3/10 , 0.18857 )
( -1/5 , 0.19101 )
( -1/10 , 0.19264 )
( 0 , 0.19321 )
( 1/10 , 0.19264 )
( 1/5 , 0.19101 )
( 3/10 , 0.18857 )
( 2/5 , 0.18560 )
( 1/2 , 0.18225 )
( 3/5 , 0.17860 )
( 7/10 , 0.17465 )
( 4/5 , 0.17042 )
( 9/10 , 0.16598 )
( 1 , 0.16140 )
( 11/10 , 0.15672 )
( 6/5 , 0.15197 )
( 13/10 , 0.14713 )
( 7/5 , 0.14222 )
( 3/2 , 0.13728 )
( 8/5 , 0.13236 )
( 17/10 , 0.12746 )
( 9/5 , 0.12258 )
( 19/10 , 0.11773 )
( 2 , 0.11291 )
( 21/10 , 0.10816 )
( 11/5 , 0.10349 )
( 23/10 , 0.098903 )
( 12/5 , 0.094394 )
( 5/2 , 0.089962 )
( 13/5 , 0.085626 )
( 27/10 , 0.081402 )
( 14/5 , 0.077295 )
( 29/10 , 0.073295 )
( 3 , 0.069397 )
( 31/10 , 0.065609 )
( 16/5 , 0.061945 )
( 33/10 , 0.058413 )
( 17/5 , 0.055001 )
( 7/2 , 0.051703 )
( 18/5 , 0.048520 )
( 37/10 , 0.045462 )
( 19/5 , 0.042536 )
( 39/10 , 0.039735 )
( 4 , 0.037047 )
( 41/10 , 0.034472 )
( 21/5 , 0.032016 )
( 43/10 , 0.029686 )
( 22/5 , 0.027475 )
( 9/2 , 0.025370 )
( 23/5 , 0.023368 )
( 47/10 , 0.021477 )
( 24/5 , 0.019701 )
( 49/10 , 0.018030 )
( 5 , 0.016452 )
( 51/10 , 0.014965 )
( 26/5 , 0.013579 )
( 53/10 , 0.012292 )
( 27/5 , 0.011091 )
( 11/2 , 0.0099672 )
( 28/5 , 0.0089239 )
( 57/10 , 0.0079681 )
( 29/5 , 0.0070911 )
( 59/10 , 0.0062779 )
( 6 , 0.0055278 )
( 61/10 , 0.0048502 )
( 31/5 , 0.0042408 )
( 63/10 , 0.0036829 )
( 32/5 , 0.0031729 )
( 13/2 , 0.0027215 )
( 33/5 , 0.0023256 )
( 67/10 , 0.0019666 )
( 34/5 , 0.0016434 )
( 69/10 , 0.0013683 )
( 7 , 0.0011300 )
( 71/10 , 0.00090992 )
( 36/5 , 0.00072535 )
( 73/10 , 0.00054269 )
( 37/5 , 0.00041402 )
( 15/2 , 0.00025034 )
( 38/5 , 0.00036331 )
( 77/10 , 0 )
( 39/5 , 0 )
( 79/10 , 0 )
( 8 , 0 )

};  
  \end{axis}
\end{tikzpicture}
\caption{Case $g=2$ and $q=5$. As pointed out in Remark \ref{rmk: q small wrt g}, there is an issue when $q+1-t<0$ (for example when $t=7$). Indeed, $H'(q, 7)=0$ because $q+1-t$ represents the number of $\F_q$-rational points of a curve. Instead, both $\nu'(q,7)\approx0.0009$ and $\nu_\infty(q,7)\approx0.0011$ are strictly positive.}
\end{center}
\end{figure}

\begin{figure}
\begin{center}
\begin{tikzpicture}[xscale=1.2, yscale=1.2]
  \begin{axis}
    [ymin=0, ymax=0.11, ytick={0,0.025,0.05,0.075,0.1}, yticklabels={0,0.025,0.05,0.075,0.1}]
    \addplot[only marks, red, mark size=0.25mm] coordinates {
( -24 , 0 )
( -23 , 0 )
( -22 , 0 )
( -21 , 0 )
( -20 , 0.0001851 )
( -19 , 0.0002764 )
( -18 , 0.0009394 )
( -17 , 0.001106 )
( -16 , 0.002456 )
( -15 , 0.003333 )
( -14 , 0.005705 )
( -13 , 0.005745 )
( -12 , 0.01264 )
( -11 , 0.01094 )
( -10 , 0.01918 )
( -9 , 0.01955 )
( -8 , 0.03029 )
( -7 , 0.02689 )
( -6 , 0.04544 )
( -5 , 0.03764 )
( -4 , 0.05804 )
( -3 , 0.05239 )
( -2 , 0.06766 )
( -1 , 0.05589 )
( 0 , 0.08728 )
( 1 , 0.05589 )
( 2 , 0.06766 )
( 3 , 0.05239 )
( 4 , 0.05804 )
( 5 , 0.03764 )
( 6 , 0.04544 )
( 7 , 0.02689 )
( 8 , 0.03029 )
( 9 , 0.01955 )
( 10 , 0.01918 )
( 11 , 0.01094 )
( 12 , 0.01264 )
( 13 , 0.005745 )
( 14 , 0.005705 )
( 15 , 0.003333 )
( 16 , 0.002456 )
( 17 , 0.001106 )
( 18 , 0.0009394 )
( 19 , 0.0002764 )
( 20 , 0.0001851 )
( 21 , 0 )
( 22 , 0 )
( 23 , 0 )
( 24 , 0 )

    };

\addplot[only marks, green, mark size=0.25mm] coordinates {
( -24 , 0 )
( -23 , 0 )
( -22 , 0.00003)
( -21 , 0.00007)
( -20 , 0.00023032 )
( -19 , 0.00031786 )
( -18 , 0.0010120 )
( -17 , 0.0011738 )
( -16 , 0.0025397 )
( -15 , 0.0033860 )
( -14 , 0.0057712 )
( -13 , 0.0058088 )
( -12 , 0.012607 )
( -11 , 0.010895 )
( -10 , 0.019022 )
( -9 , 0.019382 )
( -8 , 0.029940 )
( -7 , 0.026535 )
( -6 , 0.044744 )
( -5 , 0.037105 )
( -4 , 0.057110 )
( -3 , 0.051474 )
( -2 , 0.066426 )
( -1 , 0.054918 )
( 0 , 0.098991 )
( 1 , 0.054918 )
( 2 , 0.066426 )
( 3 , 0.051474 )
( 4 , 0.057110 )
( 5 , 0.037105 )
( 6 , 0.044744 )
( 7 , 0.026535 )
( 8 , 0.029940 )
( 9 , 0.019382 )
( 10 , 0.019022 )
( 11 , 0.010895 )
( 12 , 0.012607 )
( 13 , 0.0058088 )
( 14 , 0.0057712 )
( 15 , 0.0033860 )
( 16 , 0.0025397 )
( 17 , 0.0011738 )
( 18 , 0.0010120 )
( 19 , 0.00031786 )
( 20 , 0.00023032 )
( 21 , 0.00007)
( 22 , 0.00003)
( 23 , 0 )
( 24 , 0 )
};  
        \addplot[only marks, mark=star, mark size=0.3mm] coordinates {
( -24 , 0 )
( -23 , 0 )
( -22 , 0 )
( -21 , 0 )
( -20 , 0.0002209 )
( -19 , 0.0003434 )
( -18 , 0.0009804 )
( -17 , 0.001166 )
( -16 , 0.002682 )
( -15 , 0.003372 )
( -14 , 0.005649 )
( -13 , 0.005926 )
( -12 , 0.01275 )
( -11 , 0.01064 )
( -10 , 0.01861 )
( -9 , 0.01917 )
( -8 , 0.02990 )
( -7 , 0.02600 )
( -6 , 0.04479 )
( -5 , 0.03697 )
( -4 , 0.05795 )
( -3 , 0.05157 )
( -2 , 0.06689 )
( -1 , 0.05485 )
( 0 , 0.09912 )
( 1 , 0.05485 )
( 2 , 0.06689 )
( 3 , 0.05157 )
( 4 , 0.05795 )
( 5 , 0.03697 )
( 6 , 0.04479 )
( 7 , 0.02600 )
( 8 , 0.02990 )
( 9 , 0.01917 )
( 10 , 0.01861 )
( 11 , 0.01064 )
( 12 , 0.01275 )
( 13 , 0.005926 )
( 14 , 0.005649 )
( 15 , 0.003372 )
( 16 , 0.002682 )
( 17 , 0.001166 )
( 18 , 0.0009804 )
( 19 , 0.0003434 )
( 20 , 0.0002209 )
( 21 , 0 )
( 22 , 0 )
( 23 , 0 )
( 24 , 0 )

    };
  \end{axis}
\end{tikzpicture}
\caption{This graph shows the difference between considering all PPAVs or only Jacobians of curves (see Remark \ref{rem:PPAV}). We take $g=2$ and $q=37$. We plot in red the distribution $H'$ and in black (an approximation of) the distribution $\nu'(q, t)$. The green dots represent the probabilities of the various traces when we take into account all principally polarised abelian surfaces over $\mathbb{F}_q$. Call this distribution $H''$.
The distance between the distributions $H'$ and $\nu'(q, t)$ is $\approx0.02673$. {The distance between $H''$ and $\nu'(q, t)$ is $\approx  0.00777$.}
{Note in particular the considerable difference between the data at $t=0$, where the inclusion of all PPAVs gives a much better agreement with our prediction}. An explanation for this phenomenon is given in Remark \ref{rmk: PP abelian surfaces of trace 0}.}
\label{fig:PPAV}
\end{center}
\end{figure}


\FloatBarrier


\section{Well-posedness of Equation \texorpdfstring{\eqref{eq: guess for the density}}{}}\label{sec:wellpos}
In this section we prove that the quantity $\nu(q,t)$ is well-defined. We have already observed (Remark \ref{rmk: nu ell q t is well-defined}) that $\nu_{\ell}(q, t)$ is well-defined for all $\ell \leq \infty$, so it suffices to show that, as $\ell \to \infty$ among the prime numbers, we have $\nu_{\ell}(q, t)=1+O(\ell^{-2})$. This suffices to ensure that the product \eqref{eq: guess for the density} converges.

As a preparation for the proof, we introduce the following notation and make some remarks.
\begin{notation}\label{notation: GSp2gm}
Let $R$ be a (commutative unitary) ring and let $m \in R^\times$ be a fixed element. We define $\GSp_{2g, R}^{m}$ as the subscheme of $\GSp_{2g, R}$ cut by the equation $\mult(M)=m$.
\end{notation}
\begin{remark}\label{remark:shiftingGsp}
Let us fix the antisymmetric form
$\begin{pmatrix}
        0 & \operatorname{Id}_g \\
        -\operatorname{Id}_g & 0
    \end{pmatrix}$
. The matrix
\[
M_m \colonequals \begin{pmatrix}
    m \\
    & \ddots \\
    && m \\
    &&& 1 \\
    &&&& \ddots \\
    &&&&& 1
\end{pmatrix}
\]
is in $\GSp_{2g}(R)$ and has multiplier $m$. Multiplication by $M_m$ gives an algebraic isomorphism between the $R$-schemes $\Sp_{2g, R}$ and $\GSp_{2g, R}^{m}$ . The same applies for any matrix $M_m \in \GSp_{2g}(R)$ with multiplier $m$.
In particular, $\GSp_{2g, R}^{m}$ is smooth for any value of $m$. If $R$ is a field, the dimension of $\GSp_{2g, R}^{m}$ is equal to $\dim \Sp_{2g, R}$.
\end{remark}
In what follows we will be interested in the subschemes of $\GSp_{2g, R}^{m}$ defined by the equation $\Tr(M)=t$ for a fixed value of $t \in R$. We will mostly work with $R=\Z_\ell$ and $R=\F_\ell$. 
\begin{definition}\label{def: X_t^m}
For $m \in R^\times, t \in R$, we define the $R$-scheme $X^m_t$ as the subscheme of $\GSp_{2g, R}^{m}$ defined by the equation $\Tr(M)=t$.
\end{definition}
Notice that, if we fix $m \in \Z \setminus \{0\}$, then $m$ is invertible in $\Z[1/m]$, and hence $X_t^m$ makes sense as a scheme over $\Spec \Z[1/m]$. We will be able to reduce this scheme modulo any prime that does not divide $m$.

\subsection{Number of points of \texorpdfstring{$X^m_t$}{} over finite fields}
In this section we study the number of $\F_\ell$-points of $X_m^t$ (Theorem \ref{thm: counting symplectic matrices with given trace and multiplier} and Lemma \ref{lemma:counteq}) and show that a large proportion of them correspond to smooth points of $X_m^t$ (Lemma \ref{lemma:Xsmooth}).
For the first objective, our approach is inspired by \cite{MR1812305}. More precisely, the main result of \cite{MR1812305} gives a formula counting the number of elements in $\GSp_{2g}(\F_\ell)$ with given trace and \textit{determinant}. The same strategy allows us to prove the following version, where we count matrices with given trace and \textit{multiplier}. Before stating the result, 
we remind the reader that the $q$-binomial coefficient ${\displaystyle {\begin{bmatrix}n\\r\end{bmatrix}}_{q}}$ is defined as $\prod_{j=0}^{r-1} \frac{q^{n-j}-1}{q^{r-j}-1}$. For ease of comparison with \cite{MR1812305}, we adopt the same notation as in op.~cit.

\begin{theorem}\label{thm: counting symplectic matrices with given trace and multiplier}
    Let $q$ be a prime power, $\zeta \in \F_q^\times$, and $\eta \in \F_q$. Let
    \[
    T_m(\zeta, \eta) = q  \sum_{\alpha_1, \ldots, \alpha_m \in \F_q^\times} t\left( \alpha_1+\zeta\alpha_1^{-1} + \cdots + \alpha_m+\zeta\alpha_m^{-1} \right) - (q-1)^m,
    \]
    where
    \[
    t(x) = \begin{cases}
        1, \text{ if }x = \eta \\
        0, \text{ otherwise},
    \end{cases}
    \]
    and the sum is regarded as
    $t(0)$ for $m=0$. Let
    \[
    C(\zeta, \eta) \colonequals \left| \{ g \in \GSp_{2n}(\F_q) \bigm\vert \mult g = \zeta, \operatorname{tr} g = \eta \} \right| = \left| X_\eta^\zeta(\F_q) \right|.
    \]
    We have the following exact formula for $C(\zeta, \eta)$:

    \begin{equation}\label{eq: exact formula number of matrices}
    \begin{aligned}
        C(\zeta, \eta) = q^{n^2-1}\prod_{j=1}^n\left(q^{2j}-1\right) + E,
    \end{aligned}
    \end{equation}
    where 
    \begin{equation}\label{eq: E}
  E = q^{n^2-1} \sum_{b=0}^{\lfloor n/2 \rfloor} \left( q^{b^2+b} {\displaystyle {\begin{bmatrix}n\\2b\end{bmatrix}}_{q}} \prod_{j=1}^b (q^{2j-1}-1) \sum_{l=0}^{\lfloor n/2-b \rfloor} q^l R(n-2b+1,l)T_{n-2b-2l}(\zeta, \eta) \right),
    \end{equation}
    $R(m, l)$ denotes
    \[
    R(m, l)=\sum_{0 < j_1 < \cdots < j_l < m-l} \prod_{\nu=1}^l (q^{m-\nu-j_\nu}-1),
    \]
    and we set by convention $R(m, 0)=1$.
\end{theorem}
\begin{proof}
    The proof is virtually identical to that of \cite[Theorem 1]{MR1812305}: if one simply replaces every occurrence of $\det$ with $\mult$ in the proof of \cite[Theorem 1]{MR1812305} everything goes through without difficulty. 
    More precisely, let
    \[
    e(x) = \begin{cases}
        1 \text{ if } x = \zeta, \\
        0 \text{ otherwise}.
    \end{cases}
    \]
    Throughout the proof, several instances of $\det(d_\alpha)=\alpha^n$ are replaced by $\mult(d_\alpha)=\alpha$, where $d_\alpha = \begin{pmatrix}
        \operatorname{Id}_n & 0 \\
        0 & \alpha \operatorname{Id}_n
    \end{pmatrix}$. In particular, the sums $\sum_{\alpha \in \F_q^\times} e(\alpha^n)$ are replaced by $\sum_{\alpha \in \F_q^\times} e(\alpha)$. In the proof of \cite[Theorem 1]{MR1812305}, the sum $\sum_{\alpha \in \F_q^\times} e(\alpha^n)$ evaluates to the number $S$ of $n$-th roots of $\zeta$ in $\F_q^\times$; in our case, the sum $\sum_{\alpha \in \F_q^\times} e(\alpha)$ simply evaluates to $1$ for all $\zeta \in \F_q^\times$.
\end{proof}

We will think of the expression $E$ appearing in Equation \eqref{eq: E} as an error term. We now proceed to bound this error.
We work with a fixed value of $n$: this implies in particular that the number of summands (resp.~factors) in the sum (resp.~products) appearing in \eqref{eq: E} is $O(1)$. We then have the following estimates (where the implicit constants may depend on $n$, but not on $q$):
\begin{enumerate}
    \item ${\displaystyle {\begin{bmatrix}n\\r\end{bmatrix}}_{q}} = \prod_{j=0}^{r-1} \frac{q^{n-j}-1}{q^{r-j}-1} \ll \prod_{j=0}^{r-1} \frac{q^{n-j}}{q^{r-j}} = \prod_{j=0}^{r-1} q^{n-r} = q^{nr-r^2}$, and hence in particular ${\displaystyle {\begin{bmatrix}n\\2b \end{bmatrix}}_{q}} \ll q^{2bn-4b^2}$.
    \item $\prod_{j=1}^b (q^{2j-1}-1) \leq \prod_{j=1}^b q^{2j-1} = q^{\sum_{j=1}^b (2j-1)} = q^{b^2}$.
    \item We claim that $R(m, l) \ll q^{ml - l(l+1)}$ for $m\leq n$. To see this, notice that the length of the sum defining $R(m, l)$ is $O(1)$, so it suffices to estimate the largest summand. (The length of the sum is $O(1)$ because it is bounded by a function of $m$, and $m$ is bounded in terms of $n$.)
    Clearly the condition $j_k > j_{k-1}$ for $k=2, \ldots, l$ yields $j_\nu \geq \nu$, so $q^{m-\nu-j_\nu} \leq q^{m-2\nu}$. We can then estimate
    \[
    R(m, l) \ll \prod_{\nu=1}^l q^{m-2\nu} = q^{ml-l(l+1)},
    \]
    as claimed.
    \item We also claim that $|T_m(\zeta, \eta)| \ll q^m$. To show this, we first remark that, for fixed $\alpha_1, \ldots, \alpha_{m-1} \in \F_q^\times$, the equation
    \[
    \alpha_1+\zeta\alpha_1^{-1} + \cdots + \alpha_m+\zeta\alpha_m^{-1} = \eta
    \]
    has at most 2 solutions $\alpha_m \in \F_q^\times$.
    We can then rewrite and estimate $|T_m(\zeta, \eta)|$ as follows:
        \begin{align*}
    &\big| q \sum_{\alpha_1, \ldots, \alpha_{m-1} \in \F_q^\times} \sum_{\substack{\alpha_m \in \F_q^\times \\  \alpha_1+\zeta\alpha_1^{-1} + \cdots + \alpha_m+\zeta\alpha_m^{-1} = \eta}} 1 -(q-1)^m \big| \\&\leq  q  \cdot (q-1)^{m-1} \cdot 2 + (q-1)^m \ll q^m,
    \end{align*}
    as desired.
\end{enumerate}
We now give an upper bound for the quantity $|E|$, with $E$ as in Equation \eqref{eq: E}. 
According to our previous estimates,
\[
\begin{aligned}
\left|\sum_{l=0}^{\lfloor n/2-b \rfloor} q^l R(n-2b+1,l)T_{n-2b-2l}(\zeta, \eta) \right| &  \ll 
\sum_{l=0}^{\lfloor n/2-b \rfloor} q^l q^{(n-2b+1)l - l(l+1)} q^{n-2b-2l} \\
& \ll q^{n-2b} \sum_{l=0}^{\lfloor n/2-b \rfloor} q^{(n-2b-1)l - l^2}.
\end{aligned}
\]
Notice again that the length of this sum is $O(1)$, so it suffices to give an upper bound for its largest summand. For a fixed value of $b$, the exponent $(n-2b-1)l - l^2$ is maximal for $l=\frac{n-2b-1}{2}$ (which might not be an integer, but still provides an upper bound for the value of the exponent).
We thus get
\[
\left| \sum_{l=0}^{\lfloor n/2-b \rfloor} q^l R(n-2b+1,l)T_{n-2b-2l}(\zeta, \eta) \right| \ll q^{n-2b} q^{\left( \frac{n-2b-1}{2} \right)^2}.
\]
We now consider the expression
\begin{align*}
&\left| 
 q^{b^2+b} {\displaystyle {\begin{bmatrix}n\\2b\end{bmatrix}}_{q}} \prod_{j=1}^b (q^{2j-1}-1) \sum_{l=0}^{\lfloor n/2-b \rfloor} q^l R(n-2b+1,l)T_{n-2b-2l}(\zeta, \eta) \right| \\ & \ll q^{b^2+b} q^{2bn-4b^2} q^{b^2} q^{n-2b} q^{\left( \frac{n-2b-1}{2} \right)^2},
\end{align*}
corresponding to a fixed value of $b$ in the sum \eqref{eq: E}.
The exponent of $q$ on the right-hand side is again a quadratic function of $b$ (to be precise, it is given by $-b^2 + bn + \frac{1}{4}n^2 + \frac{1}{2}n + 1/4$), which is easily seen to achieve its maximum for $b=n/2$. This maximum value is given by $\frac{1}{2}n^2 + \frac{1}{2}n + \frac{1}{4}$. Thus, $q^{(1/2)n^2+(1/2)n+1/4}$ is an upper bound for each summand.
Keeping once again in mind that the length of the sum is $O(1)$, we have proved that \[|E|\ll q^{n^2-1} q^{\frac12n^2+\frac12n+\frac14} = q^{\frac{3}{2}n^2 + \frac{1}{2}n - \frac{3}{4}}.\] 

We can finally prove:
\begin{lemma} \label{lemma:counteq}For all $g \geq 2$, all primes $\ell$, and all $m$ with $(m, \ell)=1$ we have
    \begin{align}\label{eq: density modulo ell}
           \frac{\#X_t^m(\F_\ell)}{\#\GSp_{2g}(\F_\ell)/(\ell\varphi(\ell))}&= \frac{\# \{ M \in \GSp_{2g}(\F_\ell) : \Tr (M)=t, \mult M = m \} }{\#\GSp_{2g}(\F_\ell)/(\ell\varphi(\ell))} \\&= 1+ O(\ell^{-2})\nonumber,
    \end{align}
    where the constant implicit in the big-$O$ sign depends only on $g$.
\end{lemma}
\begin{proof}
The numerator of \eqref{eq: density modulo ell} is given by \eqref{eq: exact formula number of matrices} (with $n=g$, $q=\ell$, $\zeta=m$ and $\eta=t$). Note that $\ell^{g^2-1} \prod_{j=1}^g \left(\ell^{2j}-1 \right)$ is exactly $\frac{\#\GSp_{2g}(\F_\ell)}{\ell\varphi(\ell)}$. Thus, the ratio in \eqref{eq: density modulo ell} is given by
\[
1 + \frac{E}{\frac{1}{\ell(\ell-1)} \#\GSp_{2g}(\F_\ell)}.
\]
Since
\begin{align*}
\frac{1}{\ell(\ell-1)} \#\GSp_{2g}(\F_\ell) = \frac{1}{\ell(\ell-1)} (\ell-1) \#\Sp_{2g}(\F_\ell) = \ell^{g^2-1} \prod_{j=1}^g(\ell^{2j}-1) \gg \ell^{2g^2+g-1},
\end{align*}
we obtain that \eqref{eq: density modulo ell} is
\[
1+O\left( \ell^{\frac{3}{2}g^2 + \frac{1}{2}g -\frac{3}{4} - (2g^2+g-1)} \right) = 1+O\left( \ell^{-\frac{1}{2}g^2 - \frac{1}{2}g +\frac{1}{4}} \right),
\]
which is $1+O(\ell^{-2})$ for all $g \geq 2$.
\end{proof}

\begin{lemma}\label{lemma:Xsmooth}
Fix $t, m \in \mathbb{Z}$ and let $\ell \geq 3$ be a prime number not dividing $m$. 
Let
\[
    X \colonequals (X_t^m)_{\F_\ell} = \GSp_{2g, \F_\ell} \cap \{\Tr = t\} \cap \{\mult = m\},
\]
considered as a variety over $\F_\ell$. 
Write $X^{\smooth}$ for the smooth locus of $X$. The singular locus $X^{\sing}$ has codimension at least $3$ in $X$. We have $\#X^{\sing}(\F_\ell) = O(\ell^{2g^2+g-4})$ and
\[
\#X^{\smooth}(\F_\ell)=\frac{\#\GSp_{2g, \F_\ell}(\F_\ell)}{\ell \varphi(\ell)} (1 + O(\ell^{-2})).
\]
The implied constants depend on $t$ and $m$, but not on $\ell$.
\end{lemma}
\begin{proof}
We view $X$ as a subvariety of the affine space $\mathbb{A}_{\F_\ell}^{(2g)^2}$, considered as the space of matrices of size $2g \times 2g$. The variety $X$ is the intersection of $\GSp_{2g, \F_\ell}^{m} \cong \Sp_{2g, \F_\ell}$ with the hyperplane $H$ defined by the condition $\Tr(M) = t$. The hyperplane section $\GSp_{2g, \F_\ell}^{m} \cap H$ is smooth at a point $x \in X(\overline{\F_\ell})$ unless the (tangent space to the) hyperplane $H$ contains the tangent space of $\GSp_{2g, \overline{\F_\ell}}^{m}$ at the point $x$. Take any point $x \in X(\overline{\F_\ell})$.
Since $x$ has multiplier $m$, left multiplication by $x \in \GSp_{2g}(\overline{\F_\ell})$ gives an isomorphism $L_x$ between $\Sp_{2g, \overline{\F_\ell}}$ and $\GSp_{2g, \overline{\F_\ell}}^{m}$. The differential of $L_x$ gives an isomorphism between the tangent space at $\Id$ and the tangent space at $x$. If we identify both tangent spaces to subspaces of the tangent space to $\mathbb{A}_{\overline{\F_\ell}}^{(2g)^2}$ (that is, to matrices of size $2g \times 2g$), the differential in question is simply multiplication by $x$ itself. Thus, we may view the tangent space at $x$ as the image via $x$ of the tangent space at $\Id$, which is the Lie algebra of $\Sp_{2g, \overline{\F_\ell}}$. This can be written down explicitly: choose the anti-symmetric bilinear form represented by the matrix
\[
\Omega \colonequals \begin{pmatrix}
    0 & \operatorname{Id}_g \\
    -\operatorname{Id}_g & 0
\end{pmatrix}.
\]
Differentiating the condition ${}^tM \Omega M = \Omega$, we find that the Lie algebra of $\Sp_{2g, \overline{\F_\ell}}$ is given by those matrices $M$ that satisfy ${}^tM\Omega + \Omega M = 0$. Writing $M$ in block form, we obtain that
$\operatorname{Lie} \Sp_{2g, \overline{\F_\ell}}$ is the vector space of $\overline{\F_\ell}$-matrices
\[
\begin{pmatrix}
A & B \\ C & D
\end{pmatrix}
\]
with ${}^t B= B, {}^tC = C, {}^tD = -A$ (see \cite[\S 16.1]{MR1153249} for the identical calculation over the complex numbers). From the previous arguments, it follows that $x$ can only be a singular point if
\[
x \operatorname{Lie}(\Sp_{2g, \overline{\F_\ell}}) \subseteq \{\Tr = 0\},
\]
which is to say
\[
\Tr (xL) = 0 \quad \forall L \in \operatorname{Lie}(\Sp_{2g, \overline{\F_\ell}}).
\]
Write $x = \begin{pmatrix}
\alpha & \beta \\ \gamma & \delta
\end{pmatrix}$ and $L=\begin{pmatrix}
A & B \\ C & D
\end{pmatrix}$ with $B, C$ symmetric and $D=-{}^tA$. This easily gives $\Tr (\beta C) = \Tr(\gamma B) = 0$ for all symmetric $B, C$ (which implies that $\beta, \gamma$ are anti-symmetric) and
\[
\Tr(\alpha A - \delta \cdot {}^t A ) = \Tr(\alpha A - A \cdot {}^t \delta ) = \Tr(\alpha A - {}^t \delta \cdot A  ) = 0
\]
for all $A$ (which implies $\alpha={}^t\delta$).

Thus, the locus of non-smooth points is contained in the linear space defined by the equations \[{}^t\beta=-\beta, {}^t\gamma=-\gamma, {}^t\delta=\alpha.\] This linear space has dimension $g^2 + 2 \frac{g(g-1)}{2} = 2g^2-g$, and hence codimension at least $2g-1 \geq 3$ in $X$, each of whose irreducible components has dimension at least $\dim \GSp_{2g, \F_\ell}^{m}-1 = \dim \Sp_{2g, \F_\ell}-1 = 2g^2+g-1$ (at least one irreducible component has exactly this dimension). 
We now observe that
by the Lang-Weil estimates \cite[Theorem 1]{MR0065218} we have $\#X^{\sing}(\F_\ell) = O( \ell^{\dim X^{\sing}} ) = O(\ell^{\dim X - 3})$, 
with an implicit constant that depends only on $X$ and not $\ell$.
Taking into account the obvious decomposition $X^{\smooth}(\F_\ell) \bigsqcup X^{\sing}(\F_\ell) = X(\F_\ell)$ and the fact that
$
\#X(\F_\ell) = \frac{\#\GSp_{2g, \F_\ell}(\F_\ell)}{\ell \varphi(\ell)} (1 + O(\ell^{-2}))
$
by Lemma \ref{lemma:counteq}, we
obtain the desired
estimate $\#X^{\smooth}(\F_\ell) = \frac{\#\GSp_{2g, \F_\ell}(\F_\ell)}{\ell \varphi(\ell)} (1 + O(\ell^{-2})).$
\end{proof}

\subsection{Convergence of the infinite product \texorpdfstring{\eqref{eq: guess for the density}}{}}
 	\begin{lemma}\label{lemma:O(l-2)}
        Let $g\geq 2$, $q$ be a prime power, and $t\in \Z$. Let $\ell\geq 3$ be a prime that does not divide $q$. We have
		$\nu_{\ell}(q, t)=1+O(\ell^{-2})$, where the implied constant depends on $g$, $q$, and $t$.
	\end{lemma}
	\begin{proof}
Let $X:=X_t^q$. 
We apply \cite[Property (U), page 326]{MR0656627} to
		\[
		X(\Z_\ell) = \{ M \in \operatorname{GSp}_{2g}(\Z_\ell) : \Tr M =t, \mult M = m q \}
		\]
		\[
		m=1, \quad N = (2g)^2, \quad  n=n, \quad B = x_0 + \ell \mathbb{Z}_\ell^{(2g)^2}
		\]
where $x_0 \bmod \ell$ is a matrix lying in $X^{\sing}(\F_\ell)$. 
We first assume that $X_{\Z_\ell}$ is irreducible. Considering $X$ as a scheme over the spectrum of the DVR $\Z_\ell$, \cite[\href{https://stacks.math.columbia.edu/tag/0B2J}{Lemma 0B2J}]{stacks-project} shows that $X_{\F_\ell}$ is equidimensional of some dimension $d$, and Oesterlé's result gives
		\[
		\#\{ \text{closed balls $A$ of radius $\ell^{-n}$} : A \cap X \neq \emptyset \text{ and } A \subseteq B  \} \leq C \ell^{\dim X (n-1)}
		\]
for a constant $C$ that depends only on the degree in dimension $d$ \cite[§0.6]{MR0656627} of $X_{\F_\ell}$, which is clearly bounded independently of $\ell$.
On the other hand, we have
		\begin{align*}
		&\#\{ \text{closed balls A of radius $\ell^{-n}$} : A \cap X \neq \emptyset \text{ and } A \subseteq B  \} \\= &\# \left\{ M \in \GSp_{2g}(\Z/\ell^n\Z) : \begin{array}{c} \exists \tilde{M} \in X(\Z_\ell) \\ \tilde{M} \equiv M \pmod{\ell^n} \\ M \equiv x_0 \pmod{\ell}
		\end{array}
		 \right\}.
		\end{align*}
Hence, summing over the points $x_0 \in X^{\sing}(\F_\ell)$ we obtain
  \begin{equation}\label{eq: lifting non-smooth points}
 \# \left\{ M \in \GSp_{2g}(\Z/\ell^n\Z) : \begin{array}{c} \exists \tilde{M} \in X(\Z_\ell) \\ \tilde{M} \equiv M \pmod{\ell^n} \\ M \bmod \ell \in X^{\sing}(\F_\ell)
		\end{array}
		 \right\} \leq C \#X^{\sing}(\F_\ell) \ell^{(n-1) \dim X}.
   \end{equation}

If $X_{\Z_\ell}$ is not irreducible, we can repeat the above argument with each irreducible component $X_i$. If $C_i$ is the constant that corresponds to the component $X_i$, {applying the previous argument to $X_i$ and summing over $i$ we obtain}
\[
  \begin{aligned}
 \# \left\{ M \in \GSp_{2g}(\Z/\ell^n\Z) : \begin{array}{c} \exists \tilde{M} \in X(\Z_\ell) \\ \tilde{M} \equiv M \pmod{\ell^n} \\ M \bmod \ell \in X^{\sing}(\F_\ell)
		\end{array}
		 \right\} 
   & \leq \left(\sum_i C_i\right)\#X^{\sing}(\F_\ell) \ell^{(n-1) \dim X}.
   \end{aligned}
   \]
   
Note that the number of irreducible components is bounded independently of $\ell$, and so is the constant $\left(\sum_{i} C_i\right)$ (because the degrees are bounded in terms of the equations of $X$, which are independent of $\ell$). 
The conclusion is that there exists a constant $C$ such that \eqref{eq: lifting non-smooth points} holds for all $n$ and all but finitely many $\ell$.

Recall now the definition of $\nu_\ell(q, t)$ from Equation \eqref{eq: nu ell q t}: it is the limit over $k$ of the ratio
\begin{equation}\label{eq: nu ell q t at level k}
\frac{ \# \operatorname{Im}\left( X(\Z_\ell) \to \operatorname{GSp}_{2g}(\Z/\ell^k\Z) \right)  }{\#\GSp_{2g}(\Z/\ell^k\Z) / (\ell^k\varphi(\ell^k))}
\end{equation}
Clearly, a matrix $M$ counted in the numerator of this expression in particular reduces modulo $\ell$ to a point in $X(\F_\ell)$. For a fixed $x_0 \in X(\F_\ell)$, denote by $N(x_0, k)$ the quantity
\[
    N(x_0, k) = \#\left\{
    M \in \operatorname{Im}\left( X(\Z_\ell) \to \operatorname{GSp}_{2g}(\Z/\ell^k\Z) \right) : M \equiv x_0 \pmod{\ell}
     \right\}
\]
When $x_0$ is a smooth point of $X(\F_\ell)$, Hensel's lemma shows that $x_0$ has precisely $\ell^{(k-1)\dim X_{\F_\ell}}$ lifts to $X(\Z/\ell^k\Z)$, and each of these further lifts to a point in $X(\Z_\ell)$ (note that a smooth point necessarily lies on a component of dimension equal to $\dim X_{\F_\ell}$: indeed, $X$ is a hyperplane section of a smooth variety, so every smooth point lies on a component of maximal dimension). Therefore, we have $N(x_0, k)=\ell^{(k-1)\dim X_{\F_\ell}}$ for such $x_0$. On the other hand, Equation \eqref{eq: lifting non-smooth points} and Lemma \ref{lemma:Xsmooth} show that $\sum_{x_0 \in X^{\sing}(\F_\ell)} N(x_0, k) = O(\ell^{k \dim X_{\F_\ell}-3})$.

Thus, the numerator of \eqref{eq: nu ell q t at level k} is given by
\[
\begin{aligned}
  \sum_{x_0 \in X(\F_\ell)} N(x_0, k) & = \sum_{x_0 \in X^{\smooth}(\F_\ell)} N(x_0, k)  + \sum_{x_0 \in X^{\sing}(\F_\ell)} N(x_0, k) \\
   & = \#X^{\smooth}(\F_\ell) \ell^{(k-1)\dim X_{\F_\ell}} + O(\ell^{k \dim X_{\F_\ell}-3})  \\
   & = \frac{\#\GSp_{2g}(\F_\ell)}{\ell\varphi(\ell)} (1+O(\ell^{-2})) \cdot \ell^{(k-1)\dim X_{\F_\ell}}  + O(\ell^{k \dim X_{\F_\ell}-3}),
\end{aligned}
\]
where in the last equality we have applied Lemma \ref{lemma:Xsmooth}. Using $\dim X_{\F_\ell} = \dim \GSp_{2g, \F_\ell} - 2$ and dividing by
\[
\frac{\#\GSp_{2g}(\Z/\ell^k\Z)}{\ell^k\varphi(\ell^k)} = \frac{\#\GSp_{2g}(\F_\ell)}{\ell\varphi(\ell)} \ell^{(k-1) \dim X_{\F_\ell}},
\]
we obtain that \eqref{eq: nu ell q t at level k} is $1+O(\ell^{-2})$. The claim follows upon passing to the limit in $k$.
	\end{proof}
\begin{theorem}
Let $q$ be a prime power and $t \in \mathbb{Z}$. The infinite product
    \[\nu(q,t)=\nu_\infty(q, t)\prod_{\ell<\infty}\nu_{\ell}(q, t)\]
    converges.
\end{theorem}
\begin{proof}
    By Lemma \ref{lemma:O(l-2)} we have $\nu_{\ell}(q, t)=1+O(\ell^{-2})$ as $\ell$ ranges over primes $\ell \geq 3$ that do not divide $q$. The factors $\nu_\infty(q,t)$, $\nu_2(q, t)$ and $\nu_p(q,t)$ are well-defined, as already argued. It follows that the infinite product $\prod_{\ell<\infty}\nu_{\ell}(q, t)$ converges.
\end{proof}

We conclude this section by proving that $\nu(q,t)$ is strictly positive for $t \in \Z$ lying in the interval $(-2g\sqrt{q},2g\sqrt{q})$. This also proves that the denominator in Equation \eqref{eq:nu'} is non-zero and that $\nu'(q,t)$ is strictly positive for $t \in \Z$ lying in the interval $(-2g\sqrt{q},2g\sqrt{q})$.
\begin{lemma}\label{lemma: nu q t is non-zero}
Let $t$ be an integer in the open interval $(-2g\sqrt{q},2g\sqrt{q})$. The quantity $\nu(q,t)$ is non-zero (hence strictly positive).
\end{lemma}
\begin{proof}Since the infinite product defining $\nu(q,t)$ converges, it suffices to show that each factor in this product is non-zero. This is well-known to be true for the infinite factor $\nu_\infty(q,t)$, whose support is the interval $[-2g\sqrt{q}, 2g\sqrt{q}]$. To show that $\nu_\ell(q,t)$ is non-zero (including for $\ell= p$) we proceed as follows. Let $X_t^q$ be as in Definition \ref{def: X_t^m} (for the ring $R=\Q_\ell$) and let for simplicity $X^q := \GSp_{2g, \Q_\ell}^{q}$. We rewrite the definition of $\nu_\ell(q,t)$ in the form of Remark \ref{rmk: local factor at q}, namely,
\[
        \nu_{\ell}(q, t) = \lim_{k \to \infty} \frac{\#\operatorname{Im}\left( X^q_t(\Q_\ell) \cap \operatorname{Mat}_{2g}(\Z_\ell) \to \operatorname{Mat}_{2g}(\Z/\ell^k\Z) \right)}{\#\operatorname{Im}\left( X^q(\Q_\ell) \cap \operatorname{Mat}_{2g}(\Z_\ell) \to \operatorname{Mat}_{2g}(\Z/\ell^k\Z) \right) / \ell^k}.
\]
Set $d := \dim \GSp_{2g,\Q_\ell}-2 = 2g^2+g-1$ and multiply both numerator and denominator by $\ell^{-kd}$. We see both $X^q$ and $X^q_t$ as subschemes of $\mathbb{A}_{\Q_\ell}^{(2g)^2}$, so that their $\Q_\ell$-points are subsets of $\Q_\ell^{(2g)^2}$.
Let $Y_t^q := X_t^q(\Q_\ell) \cap \Z_\ell^{(2g)^2}$ and $Y^q := X^q(\Q_\ell) \cap \Z_\ell^{(2g)^2}$.
The sets $Y_t^q$ and $Y^q$ are closed analytic subsets of $\Z_\ell^{(2g)^2}$. Note that $X^q$ is smooth and irreducible of dimension $d+1$, hence $X_t^q$ -- which is a subscheme of $X^q$ defined by a single non-trivial equation -- has dimension $d$: slicing with a hyperplane makes the dimension drop at most by $1$; on the other hand, the dimension \textit{must} drop (if $X_t^q$ had a component of dimension $d+1$, by the irreducibility of $X^q$ we would have $X_t^q \supseteq X^q$, which is not the case). More precisely, by the same argument, every irreducible component of $X_t^q$ has dimension $d$. We can thus write
\begin{equation}\label{eq: rewriting nu ell q t}
\nu_\ell(q,t) = \lim_{k \to \infty} \frac{\ell^{-dk} \#\operatorname{Im}(Y_t^q \to (\Z/\ell^k\Z)^{(2g)^2})}{\ell^{-(d+1)k} \#\operatorname{Im}(Y^q \to (\Z/\ell^k\Z)^{(2g)^2})}.
\end{equation}

Recall from \cite[\S 3]{MR0656627} the notion of \textit{measure in dimension $d$} of a closed analytic subset $Y$ of $\Z_\ell^{(2g)^2}$ of dimension $\leq d$ (denoted by $\mu_d(Y)$). By \cite[Théorème 2]{MR0656627}, the numerator and denominator of \eqref{eq: rewriting nu ell q t} admit limit as $k \to \infty$, and these limits are given by $\mu_d(Y_t^q)$ and $\mu_{d+1}(Y^q)$, respectively. Hence, $\nu_\ell(q,t) = \frac{\mu_d(Y_t^q)}{\mu_{d+1}(Y^q)}$.

To conclude, it suffices to show that $\mu_{d+1}(Y^q)$ and $\mu_d(Y_t^q)$ are both strictly positive. 
Note that $Y^q$ is open in $X^q(\Q_\ell)$ for the $\ell$-adic topology, since it is the intersection of $X^q(\Q_\ell)$ with the $\ell$-adically open set $(\Z_\ell)^{(2g)^2}$; a similar comment applies to $X_t^q$. {We claim that to check the positivity of $\mu_{d+1}(Y^q)$ and $\mu_d(Y_t^q)$ it suffices to show that $Y^q, Y_t^q$ contain at least one smooth point of $X^q(\Q_\ell), X_t^q(\Q_\ell)$ respectively. To show this implication, we argue as follows (we discuss the case of $X^q_t$, but the case of $X^q$ is completely analogous, and in fact easier since $X^q$ is smooth). The $\ell$-adic analytic variety $X^q_t(\Q_\ell)$ is of pure dimension $d$, and its smooth points $(X^q_t)_{\operatorname{smooth}}$ form an $\ell$-adically open set, so the intersection $Y^q_t \cap (X^q_t)_{\operatorname{smooth}}$ is $\ell$-adically open (recall that $Y^q_t$ is $\ell$-adically open).
In particular, if $Y^q_t$ contains at least one smooth point of $X^q_t(\Q_\ell)$, then it contains an open set of smooth points. The local dimension at each smooth point of $X^q_t(\Q_\ell)$ is $d$. By construction of the measure $\mu_{d}$ (see again \cite[\S 3]{MR0656627}), an open subset of $X^q_t(\Q_\ell)$ consisting of smooth points has positive measure: indeed, in the case of constant dimension $d$ that we are considering here, $\mu_{d}$ is constructed locally by taking an analytic isometry between a ball in $(X^q_t)_{\operatorname{smooth}}$ and an open ball in $\mathbb{Q}_\ell^{d}$, and pulling back the Haar measure $\nu$ of $\mathbb{Q}_\ell^{d}$, normalised by $\nu(\Z_\ell^{d})=1$. It is then clear that any open set in $(X^q_t)_{\operatorname{smooth}}$ has positive measure with respect to $\mu_{d}$, and we have shown that $Y^q_t$ contains an open set of smooth points of $X^q_t(\Q_\ell)$ as soon as it contains one. We are thus reduced to checking that $Y^q, Y_t^q$ contain at least one smooth point of $X^q(\Q_\ell), X_t^q(\Q_\ell)$ respectively.}

For $X^q$, which is smooth, this amounts to constructing a symplectic matrix with coefficients in $\Z_\ell$ and given multiplier; this
follows immediately from either Proposition \ref{prop: Kirby-Rivin for general symplectic group} and Remark \ref{rmk: integral matrices} or from Remark \ref{remark:shiftingGsp} after observing that the identity matrix lies in $\Sp_{2g}(\Z_\ell)$. For $X^q_t$ we construct the relevant point explicitly.

We observe that $X^q_t$ arises as a fibre of the trace map:
$$
\text{trace}: X^q \rightarrow \mathbb{A}^1
$$
i.e., $X^q_t=\text{trace}^{-1}(t)$. A sufficient condition for a point $P \in X^q_t$ to be smooth is the existence of a curve $C \subseteq X^q$ containing $P$ such that the restriction of the trace map
$$
\text{trace}: C \rightarrow \mathbb{A}^1
$$
has non-vanishing differential at $P$. 
To see this, notice that the dimension of the tangent space at $P$ in $X^q_t$ is the dimension of the tangent space at $P$ in $X^q$ minus the dimension of the image of the differential of the trace map (restricted to $X^q$) at $P$. Let us fix the symplectic form
\[
\Omega= \begin{pmatrix}
    0 & \operatorname{Id}_g \\
    -\operatorname{Id}_g & 0
\end{pmatrix}.
\]
We consider the curve $M_a$, parametrised by $a \in \mathbb{A}^1$, given by
\[
M_a=
\begin{bmatrix}
a & z & a-q & z \\
{}^t z & q\operatorname{Id}_{g-1} & {}^t z & 0_{g-1} \\
1 & z & 1 & z \\
{}^t z & 0_{g-1} & {}^t z & \operatorname{Id}_{g-1} 
\end{bmatrix}
\]
where $z$ is the $1 \times (g-1)$ vector $(0,\ldots,0)$.  One checks that $M_a \in X^q(\Q_\ell)$: up to a suitable change of basis, the symplectic form is represented by $\operatorname{diag}\left( \begin{pmatrix}
    0 & 1 \\ -1 & 0
\end{pmatrix}, \ldots, \begin{pmatrix}
    0 & 1 \\ -1 & 0
\end{pmatrix} \right)$, and in the same basis $M_a$ becomes the matrix $\operatorname{diag}\left( \begin{pmatrix}
    a & a-q \\ 1 & 1
\end{pmatrix}, \begin{pmatrix}
    1 & 0 \\ 0 & q
\end{pmatrix}, \ldots, \begin{pmatrix}
    1 & 0 \\ 0 & q
\end{pmatrix} \right)$, which is manifestly symplectic since every $2 \times 2$ block has determinant $q$. Moreover, $\text{trace}(M_a)=a+qg-q+g$; the composition
$$
a \rightarrow M_a \rightarrow \text{trace}(M_a)=a+qg-q+g
$$
is just a translation of $\mathbb{A}^1$, which implies that the differential of the trace map at $M_a$ is surjective. Therefore, the point $M_{t-qg+q-g} \in X^{q}_t$ is smooth and its entries are elements of $\mathbb{Z}_{\ell}$. This concludes the proof.

\end{proof}

\section{Proof of Theorem \texorpdfstring{\ref{thm:OddCharacteristic}}{}}\label{sec:poddeven}
The goal of this section is to show that the set $\mathcal{P}_g(\mathbb{F}_q)$ of Definition \ref{def:Pg} spans a $\Q$-vector space of dimension $g+1$ for all pairs $(g,q)$. For a fixed genus $g$ and $q \gg_g 1$, this follows from Theorem \ref{thm: equidistribution of charpolys} (see Remark \ref{rmk: equidistribution implies Kaczorowski-Perelli}). Studying more precisely the set $\mathcal{P}_{g, 2}(\mathbb{F}_q)$ for every fixed value of $q$, we prove the statement for all $q$ and $g$. Recall that $\mathcal{P}_g(\mathbb{F}_q)$ is defined in Definition \ref{def:Pg} and $\mathcal{P}_{g, 2}(\mathbb{F}_q)$ is its reduction modulo $2$. As we pointed out in the introduction, we split our proof of Theorem \ref{thm:OddCharacteristic} into two parts, one for the case $p$ odd and one for the case $p=2$, since the properties of the $2$-torsion points are slightly different when the characteristic is odd or even.

\subsection{Proof of Theorem \texorpdfstring{\ref{thm:OddCharacteristic}}{}: \texorpdfstring{$p$}{}  odd}\label{sec:podd}

Throughout this section, the prime $p=\operatorname{char}(\mathbb{F}_q)$ is assumed to be odd. 
Thanks to Theorem \ref{GF-formula}, it makes sense to define $f_{C}(t) \in \mathbb{Z}[t]$ as $f_{C, \ell^\infty}(t)$, where $\ell$ is any prime different from $p$; from now on, we shall choose $\ell=2$.
This choice has the additional advantage that working modulo $2$ makes the connection between the $L$-polynomial and the characteristic polynomial of Frobenius particularly simple:
\begin{corollary}\label{cor:LPolyFPoly}
We have $P_{C}(t) \equiv f_{C}(t) \pmod 2$.
\end{corollary}
\begin{proof}
Write $P_{C}(t)=\sum_{i=0}^{2g} a_i t^i \in \mathbb{Z}[t]$ and $f_{C}(t)=\sum_{i=0}^{2g} b_i t^i$.
By Theorem \ref{GF-formula} we have the equality $b_i=a_{2g-i}$, and since $q$ is odd we also have 
$
b_i=a_{2g-i}=q^{g-i}a_i \equiv a_i \pmod 2
$.
\end{proof}
We now recall a concrete description for the vector space of 2-torsion points of a hyperelliptic Jacobian, at least in the case when the hyperelliptic model is given by a polynomial of odd degree.
Let $f(x) \in \mathbb{F}_q[x]$ be a separable polynomial of degree $2g+1$ and let $C/\mathbb{F}_q$ be the unique smooth projective curve birational to the affine curve
$
y^2=f(x).
$
Furthermore, let $J/\mathbb{F}_q$ be the Jacobian of $C$ and $\{\alpha_1,\ldots,\alpha_{2g+1}\}$ be the set of roots of $f(x)$ in $\overline{\mathbb{F}_q}$.
Then for $i=1,\ldots,2g+1$ we have a point $(\alpha_i,0) \in C(\overline{\mathbb{F}_q})$; also notice that $C$, being given by an odd-degree model, has a unique point at infinity, which we denote by $\infty$. We denote by $R_i = [(\alpha_i,0)-\infty]$ the classes of the divisors $Q_i=(\alpha_i,0)-\infty$ in $J(\overline{\F_q})$. We then have the following well-known description for the $2$-torsion of $J$ (see for example \cite[Section 4]{MR3058664}):

\begin{lemma}\label{lemma: description of 2-torsion}
The following hold:
\begin{enumerate}
\item Each of the divisor classes $R_i \in J(\overline{\mathbb{F}_q})$ represents a point of order 2. 
\item The classes $R_i$ span $J[2]$.
\item The only linear relation satisfied by the $R_i$ is $R_1+\cdots+R_{2g+1}=0$.
\end{enumerate}
\end{lemma}

We can now compute the action of Frobenius on the 2-torsion points of $C$. A similar result appeared independently in \cite[Proposition 2.4]{zbMATH07912221}.

\begin{lemma}\label{lemma: charpoly mod 2}
With notation as above, write $f(x)=\prod_{i=1}^r f_i(x)$ for the factorisation of $f(x)$ as a product of irreducible polynomials in $\mathbb{F}_q[x]$, and let $d_i=\deg(f_i)$. Let $\rho_2 : \Gal(\overline{\mathbb{F}_q}/\mathbb{F}_q) \to \Aut_{\mathbb{F}_2}(J[2])$ be the Galois representation attached to the 2-torsion points of $J$. Then
\[
f_{C,2}(t) =\det( t\operatorname{Id}- \rho_2(\Frob))=(t-1)^{-1}  \prod_{i=1}^r (t^{d_i}-1) \in \mathbb{F}_2[t].
\]
\end{lemma}
\begin{proof}
As above, let $\infty$ be the unique point at infinity of $C$, and for $i=1,\ldots,2g+1$ let $Q_i=(\alpha_i,0)-\infty \in \operatorname{Div}_{C}(\overline{\mathbb{F}_q})$. Write $P_i$ for the image of $Q_i$ in the $\mathbb{F}_2$-vector space $\operatorname{Div}_{C}(\overline{\mathbb{F}_q})\otimes \mathbb{F}_2$, and let $V$ be the $(2g+1)$-dimensional $\mathbb{F}_2$-vector subspace of $\operatorname{Div}_{C}(\overline{\mathbb{F}_q})\otimes \mathbb{F}_2$ spanned by the $P_i$.
There is a natural action of $\Gal(\overline{\mathbb{F}_q}/\mathbb{F}_q)$ on $V$, which we consider as a representation $\rho:\Gal(\overline{\mathbb{F}_q}/\mathbb{F}_q) \to \operatorname{GL}(V)$. By Galois theory, it is clear that $\Frob$ acts on the set $\{\alpha_i\}_{i=1}^{2g+1}$ with $r$ orbits, one corresponding to each irreducible factor of $f(x)$. The lengths of the orbits are given by the degrees $d_i$ of the factors $f_i(x)$. This means that, in the natural basis of $V$ given by the $P_i$, the action of Frobenius is given by a permutation matrix corresponding to a permutation of cycle type
$
(d_1, d_2, \ldots, d_r)
$.
It follows immediately that the characteristic polynomial of $\rho(\Frob)$ is
\[
\det(t \operatorname{Id}-\rho(\Frob))=(t^{d_1}-1) \cdots (t^{d_r}-1) \in \mathbb{F}_2[t].
\]
On the other hand, by Lemma \ref{lemma: description of 2-torsion} there is a Galois-equivariant exact sequence
\[
0 \to \mathbb{F}_2 \to V \to J[2] \to 0,
\]
where the first map is given by $1\to P_1+P_2+\dots P_{2g+1}$ and the action of $\Frob$ on the sum $P_1+\cdots+P_{2g+1}$ is trivial. This implies that
\[
\det(t \operatorname{Id}-\rho(\Frob))=\det(t\operatorname{Id}-\rho_2(\Frob)) (t-1),
\]
which, combined with our previous determination of the characteristic polynomial of $\rho(\Frob)$, concludes the proof.
\end{proof}

Thanks to the previous lemma, it is easy to obtain the reduction modulo 2 of the $L$-polynomial of any given hyperelliptic curve with an odd degree model. In the next corollary, we use this to produce curves whose $L$-polynomials have particularly simple reductions modulo $2$.

\begin{corollary}\label{cor:PolyCongruences}
Let $f_0(x)=1$ and, for $d=1,\ldots,2g+1$, let $f_d(x) \in \mathbb{F}_q[x]$ be an irreducible polynomial of degree $d$. Further set $f_0(x)=1$. For $d=0,\ldots,g$ consider the unique smooth projective curve $C_d$ birational to the affine curve 
\[
y^2=f_d(x)f_{2g+1-d}(x).
\]
For $d=1,\ldots,g$ we have the congruence
\[
(t-1)P_{C_d}(t) \equiv (t^d-1)(t^{2g+1-d}-1) \equiv t^{2g+1}+t^{2g+1-d} +t^d +1 \pmod{2},
\]
while for $d=0$ we have 
\[
(t-1)P_{C_0}(t) \equiv t^{2g+1}-1 \equiv t^{2g+1} +1 \pmod{2}.
\]
\end{corollary}
\begin{proof}
This is a direct application of Lemma \ref{lemma: charpoly mod 2}, combined with the fact that by Corollary \ref{cor:LPolyFPoly} we have $P_{C}(t) \equiv f_{C}(t) \pmod 2$.
\end{proof}

\begin{proof}[Proof of Theorem \ref{thm:OddCharacteristic} for $p$ odd]
The inequality $\dim_{\mathbb{Q}} L_g(\mathbb{F}_q) \leq g+1$ follows immediately from the symmetry relation $a_{g+i}=q^ia_{g-i}$ satisfied by the coefficients of the $L$-polynomials; it thus suffices to establish the lower bound $\dim_{\mathbb{Q}} L_g(\mathbb{F}_q) \geq g+1$.
Consider the $g+1$ curves $C_0,\ldots,C_g$ of Corollary \ref{cor:PolyCongruences} (any choice of the irreducible polynomials $f_d(x)$ will work) and the corresponding $L$-polynomials $P_{C_0}(t),\ldots,P_{C_g}(t)$. Let $M \subseteq \mathbb{Z}[t]$ be the $\mathbb{Z}$-module generated by these polynomials; it is clear that in order to prove the theorem it suffices to show that $\operatorname{rank}_{\mathbb{Z}} M \geq g+1$. Notice that $M \otimes \mathbb{F}_2$ is in a natural way a vector subspace of $\mathbb{F}_2[t]$, and that
\[
\operatorname{rank}_{\mathbb{Z}} M \geq \dim_{\mathbb{F}_2} (M \otimes \mathbb{F}_2).
\]
Let $N \subset \mathbb{F}_2[t]$ be the image of the linear map
\[
\begin{array}{ccc}
M \otimes \mathbb{F}_2 & \to & \mathbb{F}_2[t] \\
q(t) & \mapsto & (t-1)q(t).
\end{array}
\]
The $\mathbb{F}_2$-vector space $N$ is generated by the $g+1$ polynomials $(t-1)P_{C_i}(t)$ for $i=0,\ldots,g$, hence, by Corollary \ref{cor:PolyCongruences}, by the $g+1$ polynomials
\[
t^{2g+1} +1 \quad\text{and}\quad t^{2g+1}+t^{2g+1-i}+t^i+1 \text{ for }i=1,\ldots,g.
\]
It is immediate to check that these $g+1$ polynomials are $\mathbb{F}_2$-linearly independent, which implies
\[
\operatorname{rank}_{\mathbb{Z}} M \geq \dim_{\mathbb{F}_2} (M \otimes \mathbb{F}_2) = \dim_{\mathbb{F}_2} N=g+1.
\]
\end{proof}

\subsection{Proof of Theorem \texorpdfstring{\ref{thm:OddCharacteristic}}{}: \texorpdfstring{$p=2$}{}}\label{sec:even}
We now give the proof of Theorem \ref{thm:OddCharacteristic} in the case $p=2$. As in the case of odd characteristic, we will exhibit $g+1$ curves whose $L$-polynomials form a basis of $L_g(\mathbb{F}_q)$. Recall from Definition \ref{def:Pg} the set $\mathcal{P}_g(\mathbb{F}_q)$.

\begin{proof}[Proof of Theorem \ref{thm:OddCharacteristic} for $p=2$]
Fix $0\leq r\leq g$. Let $h(x) \in \F_q[x]$ be a separable polynomial of degree $r$ such that $h(0)\neq 0$. 
Such a polynomial exists: for $r=0, 1$ we may take $h(x)=1$ or $h(x)=x+1$, respectively, and for $r \geq 2$ it suffices to take as $h(x)$ the minimal polynomial of any element that generates $\mathbb{F}_{q^r}$ over $\mathbb{F}_q$.

Consider the affine curve defined by the equation $y^2+yh(x)=x^{2g+1-r}h(x)$. We claim that this curve is smooth. Indeed, an $\overline{\F_q}$-point $(x_0,y_0)$ on the curve is singular if and only if
\[
\begin{cases}
y_0^2+y_0h(x_0)=x_0^{2g+1-r}h(x_0)\\
h(x_0)=0\\
y_0h'(x_0)=(2g+1-r)x_0^{2g-r}h(x_0)+x_0^{2g+1-r}h'(x_0)
\end{cases}
\]
Here the second and third equations are given by the vanishing of the partial derivatives in $y$ and $x$ of the defining equation, respectively.
By the second equation, $x_0$ is a root of $h$. So, by the first one, $y_0=0$. Hence, the third equation becomes $x_0^{2g+1-r}h'(x_0)=0$: but $x_0\neq 0$ since $h(0)\neq 0$, and $h'(x_0)\neq 0$ since $h$ is separable, so the above system has no solutions. 
Let $C/\F_q$ be the smooth projective curve given by the completion of the curve above. The curve $C$ has genus $g$, because the degree of $x^{2g+1-r}h(x)$ is $2g+1$ and the degree of $h(x)$ is at most $g$.
In particular, $P_C(t)$ is an element of $\mathcal{P}_g(\mathbb{F}_q)$. We will show that the reduction of $P_C(t)$ modulo $2$ has degree $r$.

Let $\ell$ be an odd prime and let $T_\ell J$ be the $\ell$-adic Tate module of the Jacobian $J$ of $C$.
Let $f_{C, \ell^\infty}(t):=\det(t\Id-\rho_{\ell^\infty}(\Frob)\mid T_\ell J)$.
If $\alpha\in \overline{\F_q}$ is a root of $f_{C, \ell^\infty}(t)$ with multiplicity $d$, then $q/\alpha$ is a root of $f_{C, \ell^\infty}(t)$ with multiplicity $d$. Hence, we can write $f_{C, \ell^\infty}(t)=t^gQ_C(t+q/t)$ with $Q_C(t)\in \Z[t]$ of degree $g$. Let $r_2$ be the $2$-rank of $J$, as defined in \cite[Section 1]{p-rank}. By \cite[Proposition 3.1]{p-rank}, $r_2$ is equal to the sum of the multiplicities of the non-zero roots of $Q_C(t)$ modulo $2$. Hence,
\[
Q_C(t)\equiv t^{g-r_2}\tilde{Q}_C(t)\pmod 2
\] 
with $\tilde{Q}_C(t)\in \F_2[t]$ a polynomial of degree $r_2$ such that $\tilde{Q}_C(0)\neq 0$ (in $\F_2$). In \cite[Proof of Theorem 23]{2rank2torsion}, the authors show that the $2$-rank of $J$ is equal to one less than the number of distinct projective points where $H_1(X, Z) \colonequals h(X/Z)Z^{g+1}$ vanishes (see also \cite{MR3095219}). In our case, since $h(x)$ is separable, this implies $r_2 = \deg h(x) = r$.
Hence, we have  
\[
Q_C(t)\equiv t^{g-r}\tilde{Q}_C(t)\pmod 2
\]
with $\tilde{Q}_C(t)$ of degree $r$. As $q$ is a power of $2$, we obtain
\[
f_{C, \ell^\infty}(t)\equiv t^gQ_C\left(t+\frac qt\right)\equiv t^gQ_C\left(t\right)\equiv t^{2g-r}\tilde{Q}_C(t)\pmod 2.
\]
By Theorem \ref{GF-formula},
\begin{equation}\label{eq:PCmod2}
P_C(t)\equiv t^{2g}f_{C, \ell^\infty}\left(t^{-1}\right)\equiv t^{2g}t^{-2g+r}\tilde{Q}_C\left(t^{-1}\right)\equiv t^r \tilde{Q}_C\left(t^{-1}\right)\pmod 2.
\end{equation}
Since $\tilde{Q}_C(0)\not\equiv 0\pmod 2$ we see that the reduction of $P_C(t)$ modulo $2$ has degree $r$.
So, for each $0\leq r\leq g$, we can find a smooth hyperelliptic curve $C_r$ of genus $g$ such that $P_{C_r}(t)$ modulo $2$ has degree $r$. Therefore, the polynomials $\{P_{C_r}(t)\mid 0\leq r\leq g\}$ are linearly independent modulo $2$. The result follows as in the proof of Theorem \ref{thm:OddCharacteristic}.

\end{proof}
\begin{remark}
The polynomial $f_{C, \ell^\infty}(t)$ is monic by definition, which implies that also $Q_C(t)$ and $\tilde{Q}_C(t)$ are monic. By \eqref{eq:PCmod2}, the constant term of $P_C(t)$ modulo $2$ is $1$.
Hence,
\[
P_{C_r}(t)\equiv t^r+1+\sum_{i=1}^{r-1} a_{i,r}t^i\pmod 2.
\]

In fact, one can show that $P_{C_r}(t) \equiv t^r+1\pmod{2}$. To see this, recall from \cite[Theorem 3.1]{SGA7} that, for a smooth projective curve $C / \mathbb{F}_q$, with $q=2^f$, one has
\[
P_{C}(t) \equiv \det\left(1-t \varphi_q^{-1} \bigm\vert H^1_{\text{ét}}\left( C_{\overline{\F_q}}, \Z/2\Z \right) \right) \pmod{2},
\]
where $\varphi_q : \F_q \to \F_q$ is the Frobenius automorphism $x \mapsto x^q$. Next, recall that $H^1_{\text{ét}}\left( C_{\overline{\F_q}}, \Z/2\Z \right)$ is canonically dual to $J(\overline{\F_q})[2]$, so that we may compute $P_{C}(t)$ as the inverse characteristic polynomial of Frobenius acting on $J[2]$. For the curve $C_r$, the explicit description of $J[2]$ given in \cite[Proof of Theorem 23]{2rank2torsion} shows that the action of $\varphi_q$ on $J[2]$ is the natural Galois action on the roots of $h(x)$, that is, an $r$-cycle. It follows that the characteristic polynomial in question is $P_{C}(t) \equiv t^r-1 \pmod{2}$, as claimed.
\end{remark}

\section{Algebraic independence}\label{sec:algind}
Theorem \ref{thm:OddCharacteristic} asserts that Lemma \ref{lemma: properties of L-polynomials} captures all the linear relations among the coefficients of the polynomials $P_C(t)$. In this section, we prove an analogous result that deals with higher-order polynomial relations on the coefficients.
Lemma \ref{lemma: properties of L-polynomials} already gives a number of constraints: for $P_C(t)=\sum_{i=0}^{2g} a_i t^i$ we have $a_0=1$ and $a_{g+i}=q^i a_{g-i}$ for every $i=0, \ldots, g$; it is therefore natural to restrict our analysis to $a_1,\ldots,a_g$.
The following is the main result of this section:

\begin{theorem}\label{thm:PolynomialRelations}
Let $g,d$ be positive integers. There is a constant $e_{g,d}$ such that for any prime power $q>e_{g,d}$ and for any non-zero polynomial $f(x_1,\ldots,x_g) \in \Z[x_1,\ldots,x_g]$ of degree $\le d$ in each variable there is a curve $C \in \mathcal{M}_g(\mathbb{F}_q)$ with $L$-polynomial $P_C(t)=\sum_{i=0}^{2g} a_i t^i$ such that $f(a_1,\ldots,a_g) \neq 0$.
\end{theorem}

Notice that, unlike Theorem \ref{thm:OddCharacteristic}, $e_{g,d}$ cannot be equal to $0$ for all $g$ and $d$, 
since for fixed $q$ and $g$ we can always find a polynomial $f(x_1,\ldots,x_g)$ 
(that may depend on $q$) which vanishes on all the finitely many values of $(a_1,\ldots,a_g)$.

As is the case for Theorem \ref{thm:OddCharacteristic}, the proof of Theorem \ref{thm:PolynomialRelations} exploits the reduction of $f(x_1,\ldots,x_g)$ modulo a positive integer $N$. In this case, instead of a direct computation of the action of the Frobenius on the $N$-torsion points, we use Theorem \ref{thm: equidistribution of charpolys}, which guarantees that, for $q$ large enough, all the characteristic polynomials of the matrices in $\GSp_{2g}^q(\Z/N\Z)$ come from some element of $\mathcal{P}_{g,N}(\mathbb{F}_q)$.

To be more precise, for a curve $C \in \mathcal{M}_g(\mathbb{F}_q)$ and $P_C(t) \in \mathbb{Z}[t]$ its $L$-polynomial, let $f_C(t)=t^{2g}P_C(1/t)$ be its reciprocal polynomial. By Theorem \ref{GF-formula} $f_C(t)$ is equal to the characteristic polynomial of the action of the Frobenius of $C$ (modulo every $\ell$). Theorem \ref{thm: equidistribution of charpolys} implies that, for $q$ large enough (in terms of $N$) and for any $M \in \GSp_{2g}^q(\Z/N\Z)$, the characteristic polynomial of $M$ is equal to the reduction of $f_C(t)$ modulo $N$ for some $C \in \mathcal{M}_g(\mathbb{F}_q)$.
We then prove that there are too many characteristic polynomials of elements of $\GSp_{2g}^q(\Z/N\Z)$ for their coefficients to lie in the zero locus of some $f(x_1,\ldots,x_g)$ of fixed degree.
We are free to choose $N$, and we will always take it to be an odd prime number. We set $N=r$ and use the letter $r$ to avoid confusion.

The following lemma is a version of the well-known Schwartz-Zippel bound. Notice that a polynomial in $g$ variables having degree at most $d$ in each of them has total degree at most $dg$. 
\begin{lemma}\label{lemma: upper bound for hypersurfaces}
Let $g,d$ be natural numbers with $g \ge 1$, let $r$ be a prime number and let $f(x_1,\ldots,x_g) \in \mathbb{F}_r[x_1,\ldots,x_g]$ be a non-zero polynomial of degree $\le d$ in each variable. We have
$$
\# \{(u_1,\ldots,u_g) \in \mathbb{F}_r^g \mid f(u_1,\ldots,u_g)=0 \} \le dg \cdot r^{g-1}.
$$
\end{lemma}

Next, we identify the set of characteristic polynomials of matrices in $\GSp_{2g}^q(\F_r)$. We show the following more general result:
\begin{proposition}\label{prop: Kirby-Rivin for general symplectic group}
    Let $n$ be a positive integer, let $R$ be a commutative ring with $1$, and let $q \in R^\times$.
    Let $p(x) = a_0 + a_1 x + \cdots + a_{2n} x^{2n} \in R[x]$ be a monic polynomial satisfying $a_{n-i}=q^i a_{n+i}$ for all $i=0,\ldots,n$. There exists $M \in \GSp_{2n}(R)$ with multiplier $q$ and characteristic polynomial $p(x)$.
\end{proposition}
\begin{remark}\label{rmk: Rivin}
        The statement is a simple variant of \cite[Theorem A.1]{MR2401624}. We give a detailed argument since, unfortunately, the proof of \cite[Theorem A.1]{MR2401624} seems to contain some typos. For example, in op.~cit.~the matrix $B$ is declared to have determinant $1$, but the construction does not ensure this property; more importantly, in some examples we tried, the given construction does not seem to yield matrices with the claimed characteristic polynomials. Our construction is therefore slightly different from that of \cite[Theorem A.1]{MR2401624}, which we could not fully understand.
\end{remark}

\begin{proof}   
    We work with the symplectic form given by the matrix $J=\begin{pmatrix}
        0 & \operatorname{Id}_n \\
        -\operatorname{Id}_n & 0
    \end{pmatrix}$. We construct the desired $M$ as a block-matrix $M=\begin{pmatrix}
        0 & B \\ C & D
    \end{pmatrix}$, where $B$, $C$, $D$ satisfy the following:
    \begin{enumerate}
        \item $B, C, D$ are square $n \times n$ matrices with $B$ invertible;
        \item $B$ is the symmetric matrix
        \[
        B= \begin{pmatrix}
            0 & 0 & 0 & \cdots & 0 & 1 \\
            0 & 0 & 0 & \cdots & 1 & b_2 \\
            0 & 0 & 0 & \cdots & b_2 & b_3 \\
            &&& \iddots \\
            0 & 1 & b_2 & \cdots & b_{n-2} & b_{n-1} \\
            1 & b_2 & b_3 & \cdots & b_{n-1} & b_n \\
        \end{pmatrix},
        \]
        or, in symbols,
        \[
        B_{ij}= b_{i+j-n} \delta_{i+j \geq n+1} = \begin{cases}
            0, \text{ if } i+j \leq n \\
            1, \text{ if } i+j = n+1 \\
            b_{i+j-n}, \text{ if } i+j>n+1,
        \end{cases}
        \]
where we have set $b_1=1$ and $\delta_{i+j \geq n+1} = \begin{cases}
            1, \text{ if } i+j \geq n+1 \\
            0, \text{ otherwise.}
        \end{cases}$.
Note that any matrix $B$ of this form is invertible for any choice of the $b_i$;
        \item $C = -q ({}^tB)^{-1} = -q B^{-1}$;
        \item $D$ is the companion matrix given by $D=\begin{pmatrix}
            0 & 0 & 0 & \cdots & 0 & 0 & d_1 \\
            1 & 0 & 0 & \cdots & 0 & 0 & d_2 \\
             &   &  & \ddots &   &  \\
             0 & 0 & 0 & \cdots & 0 & 0 & d_{n-2} \\
            0 & 0 & 0 & \cdots & 1 & 0 &d_{n-1} \\
            0 & 0 & 0 & \cdots & 0 & 1 & d_n \\
        \end{pmatrix}$. In symbols, 
        \[
        D_{ij} = \begin{cases}
            1, \text{ if } i=j+1 \\
            d_i, \text{ if } j=n \\
            0, \text{ otherwise.}
        \end{cases}
        \]
    \end{enumerate}
Here $b_2, \ldots, b_n \in R$ and $d_1, \ldots, d_n \in R$ are coefficients to be chosen later. We check the conditions for the matrix $M$ to be symplectic with multiplier $q$. We compute
    \[
    {}^t M J M = \begin{pmatrix}
        0 & -{}^tC B \\ {}^tB C & {}^t B D - {}^tD B
    \end{pmatrix},
    \]
which is equal to $qJ$ if and only if
    \[
    \begin{cases}
        -{}^tC B = q\operatorname{Id} \\
        {}^tB C = -q\operatorname{Id} \\
        {}^t B D - {}^tD B = 0.
    \end{cases}
    \]
The first two equations are equivalent to one another and automatically satisfied by our choice of $C$. The third equation is equivalent to the matrix ${}^tB D = BD$ being symmetric. We claim that this is achieved by taking ($b_1=1$ and) $b_{k+1} = \sum_{i=1}^k b_i d_{n+i-k}$ for $k=1,\ldots,n-1$ (notice that $d_1$ does not occur).
Indeed, the first $n-1$ columns of the product $BD$ are given by the second, third, $\ldots$, $n$-th column of $B$, while the last one is the linear combination $d_1 B^1 + d_2 B^2 + \cdots + d_n B^n$, where we denote by $B^i$ the $i$-th column of $B$. From this, it is immediate to check that the top-left block of $BD$ of size $(n-1) \times (n-1)$ is symmetric (independently of the values of $b_2, \ldots, b_n, d_1, \ldots, d_n$), and we only need to impose that the last line of $BD$ is equal to (the transpose of) its last column. We can also ignore the coefficient in position $(n, n)$, so we compare the first $n-1$ coefficients of the last line of $BD$ with the first $n-1$ coefficients of its last column.
The $k$-th coefficient on the last line is the coefficient on the last line of the $(k+1)$-th column of $B$, that is, $b_{k+1}$. The $k$-th coefficient on the last column is given by
    \begin{align*}
    d_1 B_{k1} + d_2 B_{k2} + \cdots + d_n B_{kn} = \sum_{i=1}^n d_i B_{ki} = \sum_{i=1}^n d_i \delta_{k+i \geq n+1} b_{k+i-n} =
    \sum_{i'=1}^{k} b_{i'} d_{i'+n-k}.
    \end{align*}
Thus, the symmetry condition is satisfied if and only if for $k=1,\ldots,n-1$ we have $b_{k+1} = \sum_{i=1}^{k} b_{i} d_{n+i-k}$, as claimed. Also note that a symplectic matrix with invertible multiplier is itself invertible (because the determinant of a symplectic matrix is a power of its multiplier), so $M$ is invertible and therefore an element of $\GSp_{2n}(R)$.
In particular, for any choice of $d_1, \ldots, d_n$, we have constructed a corresponding matrix $M$ that is symplectic of multiplier $q$ and has $D$ as its bottom-right block of size $n \times n$. 
We now compute the characteristic polynomial of this matrix $M$. Consider the identity
    \[
    \begin{aligned}
    \begin{pmatrix}
        x\operatorname{Id}_n  & -B \\
        -C & x\operatorname{Id}_n-D
    \end{pmatrix} \begin{pmatrix}
        B  & 0 \\
        x \operatorname{Id}_n & B^{-1}
    \end{pmatrix}  & = \begin{pmatrix}
       0  & -\operatorname{Id}_n \\
        x^2 \operatorname{Id}_n -xD -CB & x B^{-1}-DB^{-1}
    \end{pmatrix}  \\ & = \begin{pmatrix}
       0  & -\operatorname{Id}_n \\
        (x^2+q) \operatorname{Id}_n -xD  & x B^{-1}-DB^{-1}
    \end{pmatrix},
    \end{aligned}
    \]
where we have used that -- by definition -- $CB=-q\operatorname{Id}$. Taking determinants on both sides and using that the determinant of the block-matrix $\begin{pmatrix}
        B  & 0 \\
        x \operatorname{Id} & B^{-1}
    \end{pmatrix}$ is $1$, we obtain
    \begin{align*}
    \det(x\operatorname{Id}_{2n} - M) = \det \begin{pmatrix}
       0  & -\operatorname{Id}_n \\
        (x^2+q) \operatorname{Id}_n -xD  & x B^{-1}-DB^{-1}
    \end{pmatrix}  = \det((x^2+q) \operatorname{Id}_n -xD),
    \end{align*}
where the last equality uses basic properties of the determinant of block matrices. Finally, we can rewrite this in the form
    \[
    \det(x\operatorname{Id}_{2n} - M) = x^n \det\left( \left(x + \frac{q}{x}\right) \operatorname{Id}_n - D\right),
    \]
so the characteristic polynomial of $M$ is equal to $x^np_D\left(x+\frac{q}{x} \right)$, where $p_D(x)$ is the characteristic polynomial of $D$.
To conclude the proof, it suffices to show that we can choose $D$ in such a way that $x^np_D\left(x+\frac{q}{x} \right)=p(x)$, where $p(x)$ is the polynomial given in the statement. This is easy: $D$ is a companion matrix, so any monic polynomial with coefficients in $R$ can be realised as $p_D(x)$ for suitable values of $d_1, \ldots, d_n$. Finally, it is an easy exercise to show that a monic polynomial $p(x) = \sum_{i=0}^{2n} a_i x^i$ that satisfies $a_{n-i}=q^ia_{n+i}$ for all $i=0,\ldots,n$ can be written as $x^n p_1\left( x + \frac{q}{x}\right)$ for some monic polynomial $p_1 \in R[x]$ of degree $n$.
\end{proof}

\begin{remark}\label{rmk: integral matrices}
    Inspection of the proof shows that the following slightly stronger statement is true for the case of $R$ being the fraction field of a domain $A$: if the characteristic polynomial $p(x)$ has coefficients in $A$ and $q \in A$, then we may choose $M$ to have coefficients in $A$, \textit{even if the multiplier $q$ is not invertible in $A$}. This applies in particular when $A=\Z_\ell$ and $R=\Q_\ell$.
\end{remark}

\begin{corollary}\label{cor: number of characteristic polynomials}
    Let $r$ be a prime and let $q$ be an integer prime to $r$. The set $\{ \charpol M : M \in \GSp_{2g}^q(\F_r) \}$ has cardinality $r^{g}$.
\end{corollary}
\begin{proof}
    By Proposition \ref{prop: Kirby-Rivin for general symplectic group}, the set in question is the set of all monic polynomials in $\F_r[x]$ of degree $2g$ whose coefficients $a_i$ satisfy $a_{g-i}=q^ia_{g+i}$ for all $i=0,\ldots,g$. Since any choice of the coefficients $a_1, \ldots, a_g$ corresponds to precisely one such polynomial, the total number of polynomials is $r^g$.
\end{proof}

Finally, we connect characteristic polynomials of matrices in $\GSp_{2g}^q(\F_r)$ with characteristic polynomials of Frobenius:
\begin{lemma}\label{lemma: equidistribution gives all gsp}
Let $g,r$ be positive integers. There is a constant $h_{g,r}$ such that for any prime power $q > h_{g,r}$ with $(q,r)=1$ and for any element $M$ of $\GSp_{2g}^q(\mathbb{Z}/r\mathbb{Z})$, there is a curve $C \in \mathcal{M}_g(\mathbb{F}_q)$ such that the reduction of $f_{C}(t)$ modulo $r$ is the characteristic polynomial of $M$.
\end{lemma}
\begin{proof}
This is an immediate consequence of results of Katz-Sarnak \cite{MR1659828}. We give a proof in the language of this paper.

{
For $g=1$, the result follows from the fact that (writing $q=p^n$) every polynomial of the form $t^2+at+q$ with $p \nmid a$ and $|a| \leq 2\sqrt{q}$ is the $L$-polynomial of an elliptic curve over $\F_q$ (see \cite[Theorem 4.1]{MR265369}). Consider first the prime powers $q=p^n$ for which $p$ satisfies $p> 2\sqrt{p}>r$. The integers $a=1, \ldots, r$ realise all the residue classes modulo $r$, are not divisible by $p$, and satisfy $|a| \leq 2\sqrt{q}$, so the corresponding polynomials $t^2+at+q$ are all realised by elliptic curves over $\mathbb{F}_q$, and give all the characteristic polynomials of elements in $\GSp_{2g}^q(\mathbb{Z}/r\mathbb{Z})$. Consider now the prime powers $q=p^n$ for the finitely many primes $p$ that safisfy $2\sqrt{p} \leq r$ or $p \leq 2\sqrt{p}$, with $(p,r)=1$. Suppose that $\lfloor 2\sqrt{q} \rfloor \geq pr$, which holds for all $n$ large enough (with respect to $p$). The integers $1, \ldots, \lfloor 2\sqrt{q} \rfloor$ cover all residue classes modulo $pr$, hence in particular for every residue class modulo $r$ there is $a \in \{1, \ldots, \lfloor 2\sqrt{q} \rfloor\}$ that realises the given class modulo $r$ and is not divisible by $p$ (recall that $(p,r)=1$). As above, $t^2+at+q$ is the $L$-polynomial of an elliptic curve over $\F_q$, and we are done.
}

For $g \geq 2$, the result follows from Theorem \ref{thm: equidistribution of charpolys}, as we now show.
Let $p(t)$ be the characteristic polynomial of $M$.
Notice that $\mu_r^q$ gives positive mass to the singleton $\{p(t)\}$, since $\GSp_{2g}^q(\mathbb{Z}/r\mathbb{Z})$ is a finite set. In fact, since the cardinality of $\GSp_{2g}^q(\mathbb{Z}/r\mathbb{Z})$ is independent of $q$ (it is equal to $\#\Sp_{2g}(\mathbb{Z}/r\mathbb{Z})$, provided only that $(q,r)=1$), we have $\mu_r^q \{p(t)\} \geq c_{g,r} > 0$ for some absolute constant $c_{g,r}$. By Theorem \ref{thm: equidistribution of charpolys}, this implies that $(\charpol_r)_* \mnaive$ is positive at $\{p(t)\}$ for $q$ large enough. Repeating the argument for the finitely many possible polynomials $p(t)$ concludes the proof.
\end{proof}
We can now combine our bounds to conclude the proof of Theorem \ref{thm:PolynomialRelations}.

\begin{proof}[Proof of Theorem \ref{thm:PolynomialRelations}]
Let $r$ be an odd prime number, which will later be required to be large enough. We prove the result for every $q$ which is not a power of $r$; repeating the argument with a different $r$ will prove the statement for every $q$.
First, we can assume that our polynomial $f(x_1,\ldots,x_g) \in \mathbb{Z}[x_1,\ldots,x_g]$ has a coefficient which is non-zero modulo $r$ (otherwise, divide by an appropriate power of $r$). Hence, its reduction modulo $r$ is non-zero.
By Lemma \ref{lemma: equidistribution gives all gsp}, the set of characteristic polynomials of curves in $\mathcal{M}_g(\mathbb{F}_q)$ modulo $r$ is the same as the set of characteristic polynomials of matrices of $\GSp_{2g}^q(\mathbb{F}_r)$ for $q$ large enough and relatively prime with $r$.
Suppose that for every $M \in \GSp_{2g}^q(\mathbb{F}_r)$, writing $\charpol(M)=\sum_{i=0}^{2g} a_i t^i$, we have $f(a_1,\ldots,a_g)=0$. By combining Lemma \ref{lemma: upper bound for hypersurfaces} and Corollary \ref{cor: number of characteristic polynomials} we obtain
$
r^g \le dg \cdot r^{g-1},
$
which implies $r \le dg$. If $r$ is chosen larger than this quantity, we obtain a contradiction.
\end{proof}

\noindent Francesco Ballini, University of Oxford, Andrew Wiles building, Woodstock Road, Oxford, United Kingdom\\
\textit{E-mail address}: \url{Francesco.Ballini@maths.ox.ac.uk}

\smallskip

\noindent Davide Lombardo, Università di Pisa, Dipartimento di matematica, Largo Bruno Pontecorvo 5, Pisa, Italy\\
\textit{E-mail address}: \url{davide.lombardo@unipi.it}

\smallskip

\noindent Matteo Verzobio, IST Austria, Am Campus 1, Klosterneuburg, Austria\\
\textit{E-mail address}: \url{matteo.verzobio@gmail.com}


\begin{thebibliography}{10}

\bibitem{MR2459984}
J.~D. Achter.
\newblock Results of {C}ohen-{L}enstra type for quadratic function fields.
\newblock In {\em Computational arithmetic geometry}, volume 463 of {\em
  Contemp. Math.}, pages 1--7. Amer. Math. Soc., Providence, RI, 2008.

\bibitem{MR4595380}
J.~D. Achter, S.~A. Altu\u{g}, L.~Garcia, and J.~Gordon.
\newblock Counting abelian varieties over finite fields via {F}robenius
  densities.
\newblock {\em Algebra Number Theory}, 17(7):1239--1280, 2023.
\newblock Appendix by Wen-Wei Li and Thomas R\"{u}d.

\bibitem{MR3338118}
J.~D. Achter, D.~Erman, K.~S. Kedlaya, M.~M. Wood, and D.~Zureick-Brown.
\newblock A heuristic for the distribution of point counts for random curves
  over finite field.
\newblock {\em Philos. Trans. Roy. Soc. A}, 373(2040):1--12, 2015.

\bibitem{MR3582398}
J.~D. Achter and J.~Gordon.
\newblock Elliptic curves, random matrices and orbital integrals.
\newblock {\em Pacific J. Math.}, 286(1):1--24, 2017.
\newblock With an appendix by S. Ali Altu\u{g}.

\bibitem{MR2142226}
J.~D. Achter and J.~Holden.
\newblock Notes on an analogue of the {F}ontaine-{M}azur conjecture.
\newblock {\em J. Th\'{e}or. Nombres Bordeaux}, 15(3):627--637, 2003.

\bibitem{OurScripts}
F.~Ballini, D.~Lombardo, and M.~Verzobio.
\newblock {Statistics of $L$-polynomials over finite fields}, 2023.
\newblock Online at
  \url{https://github.com/DavideLombardoMath/distribution-L-polynomials}.

\bibitem{ballini2024lpolynomialscurvesfinitefields}
F.~Ballini, D.~Lombardo, and M.~Verzobio.
\newblock {On the $L$-polynomials of curves over finite fields}, 2024.
\newblock Online at \url{https://arxiv.org/abs/1807.07370}.

\bibitem{bergström2023refinements}
J.~Bergström, E.~W. Howe, E.~Lorenzo~García, and C.~Ritzenthaler.
\newblock Refinements of {K}atz–{S}arnak theory for the number of points on
  curves over finite fields.
\newblock {\em Canadian Journal of Mathematics}, page 1–27, 2024.

\bibitem{MR0230682}
B.~J. Birch.
\newblock How the number of points of an elliptic curve over a fixed prime
  field varies.
\newblock {\em J. London Math. Soc.}, 43:57--60, 1968.

\bibitem{MR1484478}
W.~Bosma, J.~Cannon, and C.~Playoust.
\newblock The {M}agma algebra system. {I}. {T}he user language.
\newblock {\em J. Symbolic Comput.}, 24(3-4):235--265, 1997.
\newblock Computational algebra and number theory (London, 1993).

\bibitem{MR2946086}
W.~Castryck, A.~Folsom, H.~Hubrechts, and A.~V. Sutherland.
\newblock The probability that the number of points on the {J}acobian of a
  genus 2 curve is prime.
\newblock {\em Proc. Lond. Math. Soc. (3)}, 104(6):1235--1270, 2012.

\bibitem{MR3017927}
W.~Castryck and H.~Hubrechts.
\newblock The distribution of the number of points modulo an integer on
  elliptic curves over finite fields.
\newblock {\em Ramanujan J.}, 30(2):223--242, 2013.

\bibitem{2rank2torsion}
W.~Castryck, M.~Streng, and D.~Testa.
\newblock Curves in characteristic 2 with non-trivial 2-torsion.
\newblock {\em Adv. Math. Commun.}, 8(4):479--495, 2014.

\bibitem{MR1440067}
N.~Chavdarov.
\newblock The generic irreducibility of the numerator of the zeta function in a
  family of curves with large monodromy.
\newblock {\em Duke Math. J.}, 87(1):151--180, 1997.

\bibitem{zbMATH07912221}
E.~Costa, R.~Donepudi, R.~Fernando, V.~Karemaker, C.~Springer, and M.~West.
\newblock Restrictions on {Weil} polynomials of {Jacobians} of hyperelliptic
  curves.
\newblock In {\em Arithmetic geometry, number theory, and computation}, pages
  259--276. Cham: Springer, 2021.

\bibitem{MR463174}
P.~Deligne.
\newblock {\em Cohomologie \'etale}, volume 569 of {\em Lecture Notes in
  Mathematics}.
\newblock Springer-Verlag, Berlin, 1977.
\newblock S\'eminaire de g\'eom\'etrie alg\'ebrique du Bois-Marie SGA
  $4\frac{1}{2}$.

\bibitem{SGA7}
P.~Deligne and N.~M. Katz.
\newblock {\em Groupes de monodromie en g\'{e}om\'{e}trie alg\'{e}brique.
  {II}}, volume SGA 7 II of {\em Lecture Notes in Mathematics, Vol. 340}.
\newblock Springer-Verlag, Berlin-New York, 1967--1969.
\newblock {S\'{e}minaire de G\'{e}om\'{e}trie Alg\'{e}brique du Bois-Marie}.

\bibitem{MR0262240}
P.~Deligne and D.~Mumford.
\newblock The irreducibility of the space of curves of given genus.
\newblock {\em Inst. Hautes \'{E}tudes Sci. Publ. Math.}, 36:75--109, 1969.

\bibitem{MR0005125}
M.~Deuring.
\newblock Die {T}ypen der {M}ultiplikatorenringe elliptischer
  {F}unktionenk\"{o}rper.
\newblock {\em Abh. Math. Sem. Hansischen Univ.}, 14:197--272, 1941.

\bibitem{MR3095219}
A.~Elkin and R.~Pries.
\newblock Ekedahl-{O}ort strata of hyperelliptic curves in characteristic 2.
\newblock {\em Algebra Number Theory}, 7(3):507--532, 2013.

\bibitem{MR1153249}
W.~Fulton and J.~Harris.
\newblock {\em Representation theory}, volume 129 of {\em Graduate Texts in
  Mathematics}.
\newblock Springer-Verlag, New York, 1991.
\newblock A first course, Readings in Mathematics.

\bibitem{gekeler2003frobenius}
E.-U. Gekeler.
\newblock Frobenius distributions of elliptic curves over finite prime fields.
\newblock {\em Int. Math. Res. Not.}, 37:1999--2018, 2003.

\bibitem{p-rank}
J.~Gonz\'{a}lez.
\newblock On the {$p$}-rank of an abelian variety and its endomorphism algebra.
\newblock {\em Publ. Mat.}, 42(1):119--130, 1998.

\bibitem{MR3058664}
B.~H. Gross.
\newblock Hanoi lectures on the arithmetic of hyperelliptic curves.
\newblock {\em Acta Math. Vietnam.}, 37(4):579--588, 2012.

\bibitem{hartl2020crystalline}
U.~Hartl and A.~Pal.
\newblock Crystalline {Chebotar\"ev} density theorems, 2020.

\bibitem{MR2514865}
E.~W. Howe, E.~Nart, and C.~Ritzenthaler.
\newblock Jacobians in isogeny classes of abelian surfaces over finite fields.
\newblock {\em Ann. Inst. Fourier (Grenoble)}, 59(1):239--289, 2009.

\bibitem{PerelliKaczorowski}
J.~Kaczorowski and A.~Perelli.
\newblock Zeta functions of finite fields and the {S}elberg class.
\newblock {\em Acta Arith.}, 184(3):247--265, 2018.

\bibitem{MR1659828}
N.~M. Katz and P.~Sarnak.
\newblock {\em Random matrices, {F}robenius eigenvalues, and monodromy},
  volume~45 of {\em American Mathematical Society Colloquium Publications}.
\newblock American Mathematical Society, Providence, RI, 1999.

\bibitem{MR4410030}
K.~S. Kedlaya.
\newblock Notes on isocrystals.
\newblock {\em J. Number Theory}, 237:353--394, 2022.

\bibitem{MR2555991}
K.~S. Kedlaya and A.~V. Sutherland.
\newblock Hyperelliptic curves, {$L$}-polynomials, and random matrices.
\newblock In {\em Arithmetic, geometry, cryptography and coding theory}, volume
  487 of {\em Contemp. Math.}, pages 119--162. Amer. Math. Soc., Providence,
  RI, 2009.

\bibitem{MR0258855}
D.~Kirby.
\newblock Integer matrices of finite order.
\newblock {\em Rend. Mat. (6)}, 2:403--408, 1969.

\bibitem{MR0153749}
W.~Klingenberg.
\newblock Symplectic groups over local rings.
\newblock {\em Amer. J. Math.}, 85:232--240, 1963.

\bibitem{MR3502944}
G.~Lachaud.
\newblock On the distribution of the trace in the unitary symplectic group and
  the distribution of {F}robenius.
\newblock In {\em Frobenius distributions: {L}ang-{T}rotter and {S}ato-{T}ate
  conjectures}, volume 663 of {\em Contemp. Math.}, pages 185--221. Amer. Math.
  Soc., Providence, RI, 2016.

\bibitem{MR3981312}
A.~Landesman, A.~Swaminathan, J.~Tao, and Y.~Xu.
\newblock Surjectivity of {G}alois representations in rational families of
  abelian varieties.
\newblock {\em Algebra Number Theory}, 13(5):995--1038, 2019.
\newblock With an appendix by Davide Lombardo.

\bibitem{MR0568299}
S.~Lang and H.~Trotter.
\newblock {\em Frobenius distributions in {${\rm GL}\sb{2}$}-extensions},
  volume Vol. 504 of {\em Lecture Notes in Mathematics}.
\newblock Springer-Verlag, Berlin-New York, 1976.
\newblock Distribution of Frobenius automorphisms in ${\rm
  GL}\sb{2}$-extensions of the rational numbers.

\bibitem{MR0065218}
S.~Lang and A.~Weil.
\newblock Number of points of varieties in finite fields.
\newblock {\em Amer. J. Math.}, 76:819--827, 1954.

\bibitem{MR1795548}
K.~Lauter.
\newblock Geometric methods for improving the upper bounds on the number of
  rational points on algebraic curves over finite fields.
\newblock {\em J. Algebraic Geom.}, 10(1):19--36, 2001.
\newblock With an appendix in French by J.-P. Serre.

\bibitem{MR1812305}
K.~Lee.
\newblock A counting formula about the symplectic similitude group.
\newblock {\em Bull. Austral. Math. Soc.}, 63(1):15--20, 2001.

\bibitem{MR3240800}
R.~Lercier, C.~Ritzenthaler, F.~Rovetta, and J.~Sijsling.
\newblock Parametrizing the moduli space of curves and applications to smooth
  plane quartics over finite fields.
\newblock {\em LMS J. Comput. Math.}, 17:128--147, 2014.

\bibitem{MR3726904}
D.~A. Levin and Y.~Peres.
\newblock {\em Markov chains and mixing times}.
\newblock American Mathematical Society, Providence, RI, second edition, 2017.
\newblock With contributions by Elizabeth L. Wilmer, With a chapter on
  ``Coupling from the past'' by James G. Propp and David B. Wilson.

\bibitem{ma2023refinements}
Z.~Y. Ma.
\newblock Refinements on vertical {S}ato-{T}ate, 2023.

\bibitem{MR0214602}
D.~Mumford.
\newblock {\em Geometric invariant theory}, volume Band 34 of {\em Ergebnisse
  der Mathematik und ihrer Grenzgebiete, (N.F.)}.
\newblock Springer-Verlag, Berlin-New York, 1965.

\bibitem{MR0656627}
J.~Oesterl\'{e}.
\newblock R\'{e}duction modulo {$p\sp{n}$} des sous-ensembles analytiques
  ferm\'{e}s de {${\bf Z}\sp{N}\sb{p}$}.
\newblock {\em Invent. Math.}, 66(2):325--341, 1982.

\bibitem{MR2401624}
I.~Rivin.
\newblock Walks on groups, counting reducible matrices, polynomials, and
  surface and free group automorphisms.
\newblock {\em Duke Math. J.}, 142(2):353--379, 2008.

\bibitem{MR1545199}
F.~K. Schmidt.
\newblock Analytische {Z}ahlentheorie in {K}\"orpern der {C}harakteristik
  {$p$}.
\newblock {\em Math. Z.}, 33(1):1--32, 1931.

\bibitem{MR0644559}
J.-P. Serre.
\newblock Quelques applications du th\'{e}or\`eme de densit\'{e} de
  {C}hebotarev.
\newblock {\em Inst. Hautes \'{E}tudes Sci. Publ. Math.}, 54:323--401, 1981.

\bibitem{MR2920749}
J.-P. Serre.
\newblock {\em Lectures on {$N_X (p)$}}, volume~11 of {\em Chapman \& Hall/CRC
  Research Notes in Mathematics}.
\newblock CRC Press, Boca Raton, FL, 2012.

\bibitem{shmakov2023cohomological}
A.~Shmakov.
\newblock Cohomological arithmetic statistics for principally polarized abelian
  varieties over finite fields, 2023.

\bibitem{stacks-project}
{The Stacks project authors}.
\newblock {The Stacks project}.
\newblock \url{https://stacks.math.columbia.edu}, 2022.

\bibitem{MR265369}
W.~C. Waterhouse.
\newblock Abelian varieties over finite fields.
\newblock {\em Ann. Sci. \'Ecole Norm. Sup. (4)}, 2:521--560, 1969.

\bibitem{MR0027151}
A.~Weil.
\newblock {\em Sur les courbes alg\'{e}briques et les vari\'{e}t\'{e}s qui s'en
  d\'{e}duisent}, volume 7 (1945) of {\em Publications de l'Institut de
  Math\'{e}matiques de l'Universit\'{e} de Strasbourg}.
\newblock Hermann \& Cie, Paris, 1948.

\end{thebibliography}
\end{document}